\documentclass{compositio}
\usepackage[latin1]{inputenc}

\usepackage{amsmath}
\usepackage{amscd}
\usepackage{mathrsfs}
\usepackage[all]{xy}
\usepackage{graphicx}

\usepackage[dvips,colorlinks=false,linkcolor=red]{hyperref}

\makeindex

\newtheorem{theorem}{\textsc{Theorem}}[section]
\newtheorem{proposition}[theorem]{\textsc{Proposition}}
\newtheorem{lemma}[theorem]{\textsc{Lemma}}
\newtheorem{corollary}[theorem]{\textsc{Corollary}}
\theoremstyle{definition}
\newtheorem{definition}[theorem]{\textsc{Definition}}

\theoremstyle{remark}
\newtheorem{remark}[theorem]{\textsc{Remark}}
\newtheorem{hypothesis}[theorem]{\textsc{Hypothesis}}

\newtheorem{notation}[theorem]{\textsc{Notation}}
\numberwithin{equation}{subsection}

\newcommand{\dem}{\textsc{Proof. }}
\newcommand{\CVD}{$\Box$}

\newcommand{\simto}{\xrightarrow[]{\sim}}

\renewcommand{\a}{\mathcal{A}}

\newcommand{\bs}[1]{\boldsymbol{#1}}
\newcommand{\B}{\mathrm{B}}
\newcommand{\CW}{\bs{\mathrm{CW}}}
\newcommand{\C}{\mathrm{C}}

\newcommand{\D}{\mathrm{D}}
\newcommand{\E}{\mathrm{E}}
\newcommand{\Ed}{\mathcal{E}^{\dag}}

\newcommand{\e}{\mathbf{e}}
\newcommand{\et}{\mathrm{e}}
\newcommand{\F}{\mathrm{F}}
\newcommand{\Fb}{\bar{\mathrm{F}}}
\newcommand{\G}{\mathrm{G}}
\newcommand{\Gal}{\mathrm{Gal}}

\newcommand{\Hom}{\mathrm{Hom}}

\newcommand{\J}{\mathrm{J}_p}
\newcommand{\Ker}{\mathrm{Ker}}
\newcommand{\lb}{\bs{\lambda}}

\newcommand{\Mod}{\mathrm{Mod}}

\newcommand{\M}{\mathrm{M}}

\renewcommand{\O}{\mathcal{O}}
\newcommand{\ph}[1]{\langle #1\rangle}

\newcommand{\R}{\mathcal{R}}

\newcommand{\W}{\mathbf{W}}
\newcommand{\V}{\mathrm{V}}

\newcommand{\sw}{\mathrm{sw}}

\begin{document}
\allowdisplaybreaks

\title[Arithmetic and differential Swan conductors]{Arithmetic and differential Swan conductors
of rank one representations with finite local monodromy.}

\author{Bruno Chiarellotto}
\email{chiarbru@math.unipd.it}
\address{Dipartimento di Matematica Pura e Appl., Universita' degli Studi, Via Trieste 63, 35121 Padova, Italy.}

\author{Andrea Pulita}
\email{pulita@math.jussieu.fr}
\address{Fakultät für Mathematik,
Universität Bielefeld, P.O.Box 100 131, D-33501 Bielefeld,
Germany}

\subjclass{Primary 12h25; Secondary
11S15; 11S20; 14F30}

\keywords{$p$-adic differential equations, Swan conductor, Artin
Hasse exponential, Kummer Artin Schreier Witt extensions, class
field theory}

\begin{abstract}
We consider a complete discrete valuation field of characteristic
$p$, with possibly non perfect residue field. Let $\V$ be  a rank one 
continuous representation of its
absolute Galois group  with finite local monodromy. We will prove that the \emph{arithmetic
Swan conductor} of $\V$ (defined after   K.Kato in \cite{Kato-Swan}  which fits in the more general theory of      \cite{Ab-Sa-Ramification-groups}  and  \cite{Ab-Sa}) coincides with the
\emph{differential Swan conductor} of the associated differential module
$\D^{\dag}(\V)$ defined by K.Kedlaya in \cite{Ked-Swan}. This
construction is   a generalization   to the non perfect residue case of the Fontaine's
formalism as presented in \cite{Ts}.  Our method of proof will allow us to give a new interpretation of the refined Swan conductor.
\end{abstract}

\maketitle

\setcounter{tocdepth}{1} \tableofcontents

\section*{Introduction}
Let $\E$ be a complete discrete valuation field of characteristic
$p$, and let  $\G_{\E}$ be its absolute Galois group. It has been
known now for some time (cf. \cite{Fo}, \cite{Ka-Antewerp}) that, if the residue
field $k$ is perfect, then one has an equivalence between the
category of continuous $p$-adic representation of $\G_{\E}$ with
coefficients in a finite extension $K$ of $\mathbb{Q}_p$, and the
category of unit-root $\varphi$-modules over a Cohen ring $\mathcal{E}_L$ of
$\E$. If $k$ were perfect, then the introduction of a finiteness
hypothesis (finite local monodromy)  on the representation  has
allowed R.Crew, S.Matsuda and N.Tsuzuki  to associate a
$(\varphi,\nabla)$-module (with one derivative) not only on the
``arithmetic ring'' $\mathcal{E}_L$, but in some more geometric
ring $\mathcal{E}^{\dag}_L$, which we could indicated as
``\emph{the bounded Robba ring}'' (cf. \cite{Ts}, \cite{Tsu-swan},
\cite{Ma}, \cite{Matsuda-Swan-Kato}, \cite{Crew-fin} \ldots). Here ``geometric''
means that the elements of this ring can be understood as analytic
functions converging in some annulus. S.Matsuda and N.Tsuzuki
actually provide a more general framework: they allow the field of
constants $L$ to be a certain totally ramified extension of the
field of fractions of $\W(k)$. In this context we would find
ourselves with two theories of ramification between two parallel
worlds. In the arithmetic setting we have the classical ramification
theory as presented in \cite{Se} and \cite{Katz}. In particular,
we have the notion of \emph{arithmetic Swan conductor} of a
representation $\V$ with finite local monodromy. In the
differential framework, G.Christol, Z.Mebkhout, and P.Robba
obtained a theory of \emph{$p$-adic slopes} of a solvable $\nabla$-module
$\M$ (cf. \cite{RoIV}, \cite{Ch-Ro}, \cite{Ch-Me-3},
\cite{Ch-Me-4}, for example). In particular we have the notion of
\emph{$p$-adic Irregularity} of the $\nabla$-module underling the
$(\varphi,\nabla)$-module $\D^{\dag}(\V)$ associated to $\V$. The
works of R.Crew, S.Matsuda and N.Tsuzuki have shown that the
\emph{arithmetic Swan conductor} of $\V$ coincides with the
\emph{$p$-adic irregularity} of $\D^{\dag}(\V)$ (cf.
\cite{Crew-Can-ext}, \cite{Tsu-swan}, \cite{Ma}).

If, now, we allow  $k$  be non-perfect, then K.Kedlaya has
recently shown that  the category of $p$-adic representations
of $\G_{\E}$ with finite local monodromy is again  equivalent to a
category of   unit-root $(\varphi,\nabla)$-modules over $\Ed_L$ in a slightly
generalized sense with respect to the work of S.Matsuda and
N.Tsuzuki (cf. \cite{Ked-Swan}). In this new context we can no longer 
expect  to have a single derivative, indeed since the
residual field $k$ is not necessarily perfect, the number of indipendent
derivatives will depend on the cardinality of a $p$-basis of the residue field. In
fact K.Kedlaya is able to associate to a representation $\V$, with
finite local monodromy, a $(\varphi,\nabla)$-module, denoted again
by $\D^{\dag}(\V)$, where now $\nabla:\D^{\dag}(\V)\to
\D^{\dag}(\V)\otimes \widehat{\Omega}^1_{\mathcal{E}^{\dag}_L/K}$
is an integrable connection. If we fix a uniformizer parameter $t$
of $\E$ and an isomorphism $\E\simto k(\!(t)\!)$, then
$\widehat{\Omega}^1_{\mathcal{E}^{\dag}_L/K}$ is a finite free
module over $\mathcal{E}^{\dag}_L$ generated by $d/dT$, $d/du_1$,
\ldots, $d/du_r$, where $T$ is a lifting of $t$, and
$u_1,\ldots,u_r$ are lifting of a $p$-basis of the residual field
$k$. If $k$ is perfect, then
$\widehat{\Omega}^1_{\mathcal{E}^{\dag}_L/K}=\widehat{\Omega}^1_{\mathcal{E}^{\dag}_L/L}$
and we obtain exactly the theory as presented by S.Matsuda and
N.Tsuzuki. From this K.Kedlaya obtains a generalized definition of
irregularity involving all the derivations. We will call it
\emph{differential Swan conductor} of $\D^{\dag}(\V)$. This new
notion of irregularity should be seen as a re-interpretation of
the (now) classical definition of irregularity by means of slopes
of the radius of solutions given by P.Robba, G.Christol, and
Z.Mebkhout. Unfortunately K.Kedlaya does not have any interpretation
of this new irregularity in term of index. We recall that on the
arithmetic side we have now a good theory of ramification and
relative Swan conductors for non perfect residue field obtained by
A.Abbes and T.Saito (cf. \cite{Ab-Sa-Ramification-groups}). For
rank one representations it has been shown that the new definition  coincides with
the older one given by K.Kato (cf. \cite{Ab-Sa},
\cite{Kato-Swan}).

In this paper we prove that for a rank one $p$-adic representation
$\V$ of $\G_{\E}$ with finite local monodromy, the
\emph{arithmetic Swan conductor} of $\V$, defined by A.Abbes,
T.Saito and K.Kato, coincides with the \emph{differential Swan
conductor} of $\D^{\dag}(\V)$, defined by K.Kedlaya. This will be
done by understanding  Kato's definition via the techniques
introduced in \cite{Rk1}, where one has  a description of the
character group $\mathrm{H}^1(\E,\mathbb{Q}_p/\mathbb{Z}_p)$, and
one computes the functor $\D^{\dag}$ in rank one case. Moreover
within this  framework,  we will interpretate  the notion of refined Swan conductor (as in \cite{Kato-Swan}, see also \cite{Matsuda-Swan-Kato} and \cite{Ab-Sa}).

We expect a generalization of the previous results to representations of any
rank by comparison between irregularity and  
Abbes-Saito  Swan conductor. Recently  Liang Xiao    announced the  proof of  such a generalization: his methods are completely different from ours.

\subsection*{Plan of the article:}
After introducing some notation and basic definitions in  section
\ref{Generalities and Notation}, in section \ref{Radius of convergence and irregularities: the perfect residue field case} we will recall  some known facts about  (local) irregularity of differential operators  in one variable  over a non archimedean complete  valued field and its various interpretations.    In section
\ref{Kedlaya's differential Swan conductor} we introduce,  Kedlaya's definition
of the $p$-adic differential Swan conductor for $(\varphi,\nabla)$-modules.

In section \ref{Arithmetic Swan conductor for rank one
representations with finite local monodromy}, we recall  Kato's
definition of Swan conductor for characters
$\alpha:\G_{\E}\to\mathbb{Q}/\mathbb{Z}$ \emph{with finite image}.
We then interpret  an explicit description of
$\mathrm{H}^1(\G_{\E},\mathbb{Q}_p/\mathbb{Z}_p)$ obtained in
\cite[Section 4.1]{Rk1} in term of Kato's filtration of
$\mathrm{H}^1(\G_{\E},\mathbb{Q}/\mathbb{Z})$ (cf. Theorem
\ref{description of H^1}). As a corollary we prove that
\begin{equation}
\mathrm{Gr}_d(\mathrm{H}^1(\G_{\E},\mathbb{Q}_p/\mathbb{Z}_p))\;\;
\cong \;\;\left\{
\begin{array}{lcl}
\W_{v_p(d)}(k)/p \W_{v_p(d)}(k) & \textrm{ if } & d>0\;,\\
&&\\
\mathrm{H}^1(k,\mathbb{Q}_p/\mathbb{Z}_p) & \textrm{ if } & d=0\;,
\end{array}
\right.
\end{equation}
where $\mathrm{Gr}_d$ means the $d$-th graded piece with respect
to the Kato's filtration, and $v_p(d)$ is the $p$-adic valuation
of $d$. This simple expression of
$\mathrm{Gr}_d(\;\mathrm{H}^1(\G_{\E},\mathbb{Q}_p/\mathbb{Z}_p)\;)$
is obtained by considering any such a character as the image of a
Witt co-vector in the quotient $\CW(\E)/(\Fb-1)\CW(\E)$, which is
isomorphic to $\mathrm{H}^1(\G_{\E},\mathbb{Q}_p/\mathbb{Z}_p)$
via the Artin-Schreier-Witt theory (cf. Sequence
\eqref{artin-schreier-diagram-covectors}).

In the same section we recall (and precise) a decomposition of
$\G_{\E}^{\mathrm{ab}}$ obtained in \cite[Equation (4.32)]{Rk1}.
This allow us to generalize the Kato's definition of the Swan
conductor to rank one representations \emph{with finite local
monodromy} (not necessarily with finite image) (cf. Definition
\ref{Definition of Swan for finite char}).  In section \ref{Kato's refined Swan conductor},
by means of the isomorphism $(0.0.1)$, we will give a new construction of 
Kato's refined Swan conductor according to \cite{Kato-Swan} and  \cite{Ab-Sa}.

Finally, in the last sections \ref{computation of the functor-yt-}
and \ref{Comparison between arithmetic and differential Swan
conductors}, we prove our theorem: this will be done by computing
$\D^{\dag}$ in the rank one case along the lines of \cite{Rk1}.
Firstly we reduce the problem to the case of a representation
defined by a ``\emph{pure co-monomial}'' (cf. Def. \ref{minimal
lifting}), then we prove that in this case the differential Swan
conductor (defined as the slope of $T(\M,\rho)$ at $1^{-}$ (cf.
Section \ref{definition of diff Swan conductor})) coincides with the
slope of $T(\M,\rho)$ at $0^+$ which can be computed explicitly in
term of the coefficients of the equation (Lemma \ref{radius at zero -yhnct}).

\begin{acknowledgements}
Imprimis we would like to thank A. Abbes for his inspiring suggestions and advice throughout  the preparation of the article.
We acknowledge the  help given by  J.Oesterle for suggesting us  the proof of Lemma \ref{tretretyi--} in this
short form, using Cartier operator. We thank also I. Fesenko, F. Sullivan  and the anonymous referee. The work of the first author has been done in the Framework of the EC  contract CT2003-504917 ``Arithmetic Algebraic Geometry"   and Ministero Universit\`a PRIN ``Varieta' Algebriche, Teoria dei Motivi e Geometria Aritmetica".  While the work of the second author has been done at the University of
Bielefeld in the SFB 701 project B5 ``Crystalline cohomology and
Abelian manifolds''.
\end{acknowledgements}

\section{Generalities and Notation}
\label{Generalities and Notation}

\subsection{Rings}
\label{rings} Let $(L,|\cdot|)$ be an ultrametric complete valued
field, and let $I\subseteq\mathbb{R}_{\geq 0}$ be an interval. We
set
\begin{eqnarray}
\quad\a_L(I)&:=&\{\;\sum_{i\in\mathbb{Z}}a_iT^i \;|\; a_i\in L\;,\; |a_i|\rho^i\to 0\;,\;\textrm{ for }i\to\pm\infty\;,\;\forall \rho\in I\;  \}\;,\\
\mathcal{E}_L&:=&\{\;f=\sum_{i\in\mathbb{Z}}a_iT^i\;|\;a_i\in L\;,\;\sup_{i\in\mathbb{Z}}|a_i|<+\infty\;,\;\lim_{i\to-\infty}a_i=0\;\}\;,\\
\R_L&:=&\bigcup_{\varepsilon>0}\;\a_L(]\varepsilon,1[)\;,\\
\Ed_L&:=&\mathcal{R}_L\cap\mathcal{E}_L\;.
\end{eqnarray}
The ring $\a_L(I)$ is complete with respect to the topology given
by the family of norms $\{|\cdot|_\rho\}_{\rho\in I}$, where
\begin{equation}
|\sum_{i}a_iT^i|_\rho:=\sup_i|a_i|\rho^i\;.
\end{equation}
The Robba ring $\R_L$ is complete with respect to the
limit-Frechet topology induced by the family of topologies of
$\a_L(]\varepsilon,1[)$. The ring $\mathcal{E}_L$ is complete for
the topology given by the Gauss norm $|\cdot|_1 =
|\cdot|_{\mathcal{E}_L} = |\cdot|_{\mathrm{Gauss}}$: $ |\sum
a_iT^i|_1:=\sup_i|a_i|$. The ring $\Ed_L$ has two topologies
arising from $\R_L$ and $\mathcal{E}_L$ respectively, moreover
$\Ed_L$ is dense in $\mathcal{E}_L$ and in $\R_L$ for the
respective topologies. If the valuation on $L$ is discrete, then $\Ed_L$ is
an henselian field.

Let $[\rho]:=\{\rho\}$, $\rho>0$, be the (closed) interval which consists of only 
 the point $\rho$. The completion of the fraction field of
$\a_L([\rho])$ will be denoted by
\begin{equation}\label{F_rho}
(\mathcal{F}_{L,\rho},|\cdot|_\rho)\;:=\;(\;\mathrm{Frac}(\a_L([\rho]))\;,\;|\cdot|_\rho\;)^{\widehat{\quad}}\;.
\end{equation}
Since rational fractions, without poles of norm $\rho$, are dense
in $\a_L([\rho])$ with respect to the norm $|\cdot|_{\rho}$, one has that 
 $\mathcal{F}_{L,\rho}$ is the completion of the field of
rational fraction $L(T)$ with respect to the norm $|\cdot|_\rho$.
If $\rho\in I$, one has the inclusions
\begin{equation}
\a_L(I)\;\subset\;\a_{L}([\rho])\;\subset\;
\mathcal{F}_{L,\rho}\;.
\end{equation}

\subsection{Witt vectors and Witt covectors}
For all rings $R$ (not necessarily with unit element), we denote
by $\W_m(R)$ the ring of Witt vectors with coefficients in $R$.
All notation comes from \cite{Bou}, with the exception that in our
setting $\W_m(R):=\W(R)/\V^{m+1}(\W(R))$, where
$\V(\lambda_0,\lambda_1,\ldots)=(0,\lambda_0,\lambda_1,\ldots)$ is
the Verschiebung. We denote by
$\phi_j:=\phi_j(\lambda_0,\lambda_1,\ldots):=\lambda_0^{p^j}+p\lambda_1^{p^{j-1}}+\cdots
+p^j\lambda_j$ the $j$-th phantom component of
$(\lambda_0,\lambda_1,\ldots)$. We distinguish phantom vectors
from Witt vectors by using the notation
$\ph{\phi_0,\phi_1,\ldots}\in R^{\mathbb{N}}$ instead of
$(\phi_0,\phi_1,\ldots)$. We denote by $\F$ the Frobenius of
$\W(R)$ (i.e. the one that induce the map
$\ph{\phi_0,\phi_1,\ldots}\mapsto\ph{\phi_1,\phi_2,\ldots}$ on the
phantom components). If $R$ is a $\mathbb{F}_p$-algebra, we define
a Frobenius $\Fb:\W_m(R)\to\W_m(R)$ by
$\Fb(\lambda_0,\lambda_1,\ldots):=(\lambda_0^p,\lambda_1^p,\ldots)$.

\subsubsection{}\label{CW(I)} We recall that if
$I\subseteq R$ is a subgroup closed by multiplication, then
$\W(I):=\{(\lambda_0,\lambda_1,\ldots)$ $\in\W(R)\;|\; \lambda_i\in
I,\textrm{ for all } i\geq 0 \}$ is a subgroup closed by
multiplication of $\W(R)$. In the sequel we will apply this to the
ring $R[[t]][t^{-1}]$, and $I=tR[[t]]$, or $I=t^{-1}R[t^{-1}]$.

\subsubsection{The $R$-modules $\CW(R)$ and $\widetilde{\CW}(R)$.}
We set
$\bs{\mathrm{CW}}(R):=\varinjlim(\W_m(R)\xrightarrow[]{\V}\W_{m+1}(R)\xrightarrow[]{\V}\cdots)$
and, if $R$ is an $\mathbb{F}_p$-algebra, we set
$\widetilde{\CW}(R):=\varinjlim(\W_m(R)\xrightarrow[]{p}\W_{m+1}(R)\xrightarrow[]{p}\cdots)$,
where $p : \W_m(R) \to \W_{m+1}(R)$ denotes the morphism $\V\Fb :
(\lambda_0,\ldots,\lambda_m)\mapsto
(0,\lambda_0^p,\ldots,\lambda_m^p)$.

\begin{remark}
If $R$ is an $\mathbb{F}_p$-algebra, and if the Frobenius
$x\mapsto x^p:R\to R$ is a bijection, then $\widetilde{\CW}(R)$
and $\bs{\mathrm{CW}}(R)$ are isomorphic (cf. \cite[Section
1.3.4]{Rk1}). In general they are not isomorphic.
\end{remark}

\subsubsection{Canonical Filtration of $\CW(R)$ and $\widetilde{\CW}(R)$.} 
\label{CAN}
Both $\CW(R)$ and $\widetilde{\CW}(R)$ have a
natural \emph{filtration} given by
$\mathrm{Fil}_m(\CW(R)):=\W_m(R)\subset\CW(R)$ and
$\mathrm{Fil}_m(\widetilde{\CW}(R)):=\W_m(R)\subset
\widetilde{\CW}(R)$, respectively. Defining $\mathrm{Gr}_m$ as
$\mathrm{Fil}_m/\mathrm{Fil}_{m-1}$, one has :
\begin{eqnarray}\label{Natural filtration of CW and CW tilde}
\qquad\mathrm{Gr}_m(\CW(R))&=&\W_m(R)/\V(\W_{m-1}(R))\;\;\cong\;\;
R\;,\\
\qquad\mathrm{Gr}_m(\widetilde{\CW}(R))&=&\W_m(R)/p\cdot\W_{m}(R)\;.
\end{eqnarray}
Indeed $p\cdot\W_{m}(R) = \V\Fb(\W_{m-1}(R))$. We recall that if
$R$ is an $\mathbb{F}_p$-algebra then
\begin{equation}\label{explicit def of pW(R)}
p\cdot\W_{m}(R)=\{\;\lb\in\W_m(R)\;|\;\lb=(0,\lambda_1^p,\ldots,\lambda_m^p)\;,\;\lambda_1,\ldots,\lambda_m\in
R \;\}\;.
\end{equation}

\subsection{Notation on Artin-Schreier-Witt theory}
\label{Notation in IAS theory} Let $\kappa$ be a field of
characteristic $p>0$ and let $\kappa^{\mathrm{sep}}/\kappa$ be a
fixed separable closure of $\kappa$. We set
\begin{equation}
\G_\kappa:=\mathrm{Gal}(\kappa^{\mathrm{sep}}/\kappa)\;.
\end{equation}
If $\kappa$ is a complete discrete valuation field, we denote by
$\mathcal{I}_{\G_\kappa}$ the inertia group and by
$\mathcal{P}_{\G_\kappa}$ the pro-$p$-Sylow subgroup of
$\mathcal{I}_{\G_\kappa}$. We have
$\textrm{H}^1(\G_\kappa,\mathbb{Z}/p^{m}\mathbb{Z})\stackrel{\sim}{\to}
\Hom(\G_\kappa,\mathbb{Z}/p^{m}\mathbb{Z})$ (cf. \cite[Ch.$X$, $\S
3$]{Se}). The situation is then expressed by the following
commutative diagram:
\begin{equation}\label{artin-screier-diagram}
\xymatrix{ 0\ar[r]&\mathbb{Z}/p^{m+1}\mathbb{Z}\ar[r]
\ar@{}[dr]|{\odot}\ar[d]^{\imath} &
\W_m(\kappa)\ar[r]^{\Fb-1}\ar[d]^{\V}\ar@{}[dr]|{\odot}&
\W_m(\kappa)\ar[r]^-{\delta}\ar@{}[dr]|{\odot}\ar[d]^{\V}&
\Hom(\G_\kappa,\mathbb{Z}/p^{m+1}\mathbb{Z})
\ar[d]^{\jmath}\ar[r]&0\\
0\ar[r]&\mathbb{Z}/p^{m+2}\mathbb{Z}\ar[r]&
\W_{m+1}(\kappa)\ar[r]^{\Fb-1}&\W_{m+1}(\kappa)\ar[r]^-{\delta}&
\Hom(\G_\kappa,\mathbb{Z}/p^{m+2}\mathbb{Z})\ar[r]&0}
\end{equation}
where $\imath:1\mapsto p$ is the usual inclusion, and $\jmath$ is
the composition with $\imath$. For $\lb\in \W_m(\kappa)$, the
character $\alpha=\delta(\lb)\in
\Hom(\G_\kappa,\mathbb{Z}/p^{m+1}\mathbb{Z})$ satisfies
$\alpha(\gamma)=\gamma(\bs{\nu})-\bs{\nu}\;\in\;
\mathbb{Z}/p^{m+1}\mathbb{Z}$, for all $\gamma\in
\mathrm{G}_\kappa$, where $\bs{\nu}\in
\W_m(\kappa^{\mathrm{sep}})$ is a solution of the equation
$\Fb(\bs{\nu})-\bs{\nu}=\lb$. Passing to the inductive limit, we
obtain the exact sequence:
\begin{equation}\label{artin-schreier-diagram-covectors}
0\to\mathbb{Q}_p/\mathbb{Z}_p\xrightarrow[]{\;\;\;\;\;}
\bs{\mathrm{CW}}(\kappa)\xrightarrow[]{\Fb-1}\bs{\mathrm{CW}}(\kappa)\xrightarrow[]{\;\;\delta\;\;}
\Hom^{\textrm{cont}}(\G_\kappa,\mathbb{Q}_p/\mathbb{Z}_p)\to 0\;,
\end{equation}
where $\mathbb{Q}_p/\mathbb{Z}_p$ is considered with the discrete
topology, in order that the word ``$\textrm{cont}$'' means that
all characters $\G_\kappa\to\mathbb{Q}_p/\mathbb{Z}_p$ has finite
image. Indeed
$\varinjlim_m\Hom(\G_\kappa,\mathbb{Z}/p^m\mathbb{Z})$ can be viewed
as the subset of $\Hom(\G_\kappa,\mathbb{Q}_p/\mathbb{Z}_p)$
formed by the elements killed by a power of $p$ (cf. Remark
\ref{meaning of continuous characters}).
\begin{remark}\label{VF is the same of F}
If the vertical arrows $\V$ are replaced by $\V\Fb$ in the diagram
\eqref{artin-screier-diagram}, then the morphisms $\imath$ and
$\jmath$ remain the same. Indeed $\delta(\lb)=\delta(\Fb(\lb))$,
because $\Fb(\lb) = \lb + (\Fb-1)(\lb)$, for all
$\lb\in\W_m(\kappa)$. Hence we have also a sequence as follows:
\begin{equation}
0\to\mathbb{Q}_p/\mathbb{Z}_p
\to\widetilde{\CW}(\kappa)\xrightarrow[]{\Fb-1}\widetilde{\CW}(\kappa)\xrightarrow[]{\;\;\widetilde{\delta}\;\;}
\Hom^{\textrm{cont}}(\G_\kappa,\mathbb{Q}_p/\mathbb{Z}_p)\to 0\;.
\end{equation}
\end{remark}
\subsubsection{Artin-Schreier-Witt
extensions.}\label{R(nu_0,...,nu_m)} Let
$\lb=(\lambda_0,\ldots,\lambda_m)\in\W_m(\kappa)$. Let
$\alpha:=\delta(\lb)$, then
\begin{equation}\label{(R^sep)^Ker(alpha).;.}
(\kappa^{\mathrm{sep}})^{\mathrm{Ker}(\alpha)}=\kappa(\{\nu_0,\ldots,\nu_m\})
\end{equation}
(i.e. the smallest sub-field of $\kappa^{\mathrm{sep}}$ containing
$k$ and the set $\{\nu_0,\ldots,\nu_m\}$), where
$\bs{\nu}=(\nu_0,\ldots,\nu_m)\in \W_m(\kappa^{\mathrm{sep}})$ is
a solution of $\Fb(\bs{\nu})-\bs{\nu}=\lb$. All cyclic (separable)
extensions of $\kappa$, whose degree is a power of $p$, are of
this form for a suitable $m\geq 0$, and $\lb\in\W_m(\kappa)$.

\subsection{Witt (co-)vectors of filtered and graded rings}
\label{Witt co-vectors of a filtered ring...li} Let $A$ be a ring
(not necessarily with unit element). We denote
$\mathbb{N}:=\mathbb{Z}_{\geq 0}$. A \emph{filtration} on $A$
indexed by $\mathbb{N}$ is a family of subgroups of $(A,+)$
\begin{equation}
0=\mathrm{Fil}_{-1}(A)\;\subseteq\;
\mathrm{Fil}_{0}(A)\;\subseteq\; \mathrm{Fil}_{1}(A)\; \subseteq\;
\mathrm{Fil}_{2}(A)\;\subseteq\;\cdots
\end{equation}
such that
\begin{equation}
\mathrm{Fil}_{d_1}(A) \cdot \mathrm{Fil}_{d_2}(A) \;\subseteq\;
\mathrm{Fil}_{d_1+d_2}(A)\;,
\end{equation}
for all $d_1,d_2 \geq 0$. The pair $(A,\{\mathrm{Fil}_{d}(A)\}_{d\geq
0})$ is called \emph{filtered ring}. We say that a ring $A$ is
\emph{graded} if   $A=\oplus_{d\geq 0}\mathrm{Gr}_d(A)$ (as
additive groups), where  $\mathrm{Gr}_d(A)$, $d\geq 0$ are  subgroups of $A$,  with the property that $\mathrm{Gr}_{d_1}(A)\cdot
\mathrm{Gr}_{d_2}(A)\subseteq \mathrm{Gr}_{d_1+d_2}(A)$. To a graded ring $A=\oplus_{d\geq 0}\mathrm{Gr}_d(A)$ it is associated a natural filtration given by $\mathrm{Fil}_{n}(A)= \oplus_{0 \leq d \leq n}\mathrm{Gr}_d(A)$, which makes $A$ a filtered ring.

\subsubsection{Filtration associated to a valuation.}
\label{filtration of a valuation} A \emph{valuation} on $A$ with values in $\mathbb{Z}$
is a map $v:A\longrightarrow\mathbb{Z}\cup\{\infty\}$ such that
$v(a)=\infty$ if and only if $a=0$, and for all $a_1,a_2\in A$ one
has $v(a_1+a_2)\geq \min(v(a_1),v(a_2))$, and $v(a_1\cdot a_2) =
v(a_1)+v(a_2)$. If $A$ is a valued ring, then $A$ is a domain. By
setting
\begin{equation}\label{Filtration induced by the valuation}
\mathrm{Fil}_{d}(A) \; := \; \{ a\in A \;|\; v(a) \geq -d
\}\;,\quad\textrm{for all }d\geq 0,
\end{equation}
one obtains a  filtration on $A$. In particular, this
applies if $A$ is equal to $R[T]$, $TR[T]$, $R[T,T^{-1}]$,
$R[[T]][T^{-1}]$, $T^{-1}R[T^{-1}]$ (where  $R$ is a ring), or if $A$ is a valued field of any characteristic.

\subsubsection{Witt (co-)vectors of a filtered ring.}

\begin{definition}
For all integers $d,i \geq 0$ we set
\begin{equation}
\{d:p^i\}\;:=\;\left\{
\begin{array}{rcl}
d/p^i&\textrm{ if }&p^i | d\\
-1&\textrm{ if }&i> v_p(d)
\end{array}
\right.
\end{equation}
\end{definition}

We then have

\begin{definition}
Let $(A,\{\mathrm{Fil}_d(A)\}_{d\geq 0})$ be a filtered ring. We
define for all $m,d\geq 0$
\begin{equation}\label{Filtration of W_m(A)}
\mathrm{Fil}_d(\W_m(A))\;:=\;\{(\lambda_0,\ldots,\lambda_m)\in\W_m(A)\;|\;
\lambda_{m-i} \in \mathrm{Fil}_{\{d:p^i\}}(A)\}\;.
\end{equation}

\end{definition}

\begin{lemma}
The ring $\W_m(A)$ together with
$\{\mathrm{Fil}_d(\W_m(A))\}_{d\geq 0}$ is a filtered ring.
\end{lemma}
\begin{proof}
Let $\lb,\bs{\mu}\in \mathrm{Fil}_d(\W_m(A))$, and let
$\bs{\nu}:=\lb+\bs{\mu}=(\nu_0,\ldots,\nu_m)$. Then, for all
$0\leq n\leq m$, $\nu_n$ is an isobaric polynomial in
$\lambda_0,\ldots,\lambda_n,\mu_0,\ldots,\mu_n$, more precisely
\begin{equation}\label{sum of witt vectors}
\nu_n=\sum_{i_0,\ldots,i_n,j_0,\ldots,j_n\geq
0}k_{i_0,\ldots,i_n,j_0,\ldots,j_n} \lambda_0^{i_0}\cdots
\lambda_{n}^{i_n}\mu_0^{j_0}\cdots \mu_{n}^{j_n}\;,
\end{equation}
where, $k_{0,\ldots,0}=0$, $k_{i_0,\ldots,i_n,j_0,\ldots,j_n} \in
\mathbb{N}$, and, for all $i_0,\ldots,i_n,j_0,\ldots,j_n\geq 0$,
one has
\begin{equation}
(i_0+i_1p+ \cdots +i_np^n)+(j_0+j_1p^1+\cdots+j_np^n) = p^n\;.
\end{equation}
Hence the condition $\lambda_{m-i},\mu_{m-i}\in
\mathrm{Fil}_{\{d:p^i\}}(A)$, implies $\nu_{m-i}\in
\mathrm{Fil}_{\{d:p^i\}}(A)$.

On the other hand, if $\lb \in \mathrm{Fil}_{d_1}(\W_m(A))$,
${\bs \mu} \in \mathrm{Fil}_{d_2}(\W_m(A))$, let
$\bs{\eta}:=\lb\cdot\bs{\mu}=(\eta_0,\ldots,\eta_m)$. Then for all
$0\leq n\leq m$, one has
\begin{equation}
\eta_{n}=\sum_{i_0,\ldots,i_n,j_0,\ldots,j_n\geq
0}k'_{i_0,\ldots,i_n,j_0,\ldots,j_n} \lambda_0^{i_0}\cdots
\lambda_{n}^{i_n}\mu_0^{j_0}\cdots \mu_{n}^{j_n}
\end{equation}
where, $k'_{0,\ldots,0}=0$, $k'_{i_0,\ldots,i_n,j_0,\ldots,j_n}
\in \mathbb{N}$, and, for all $i_0,\ldots,i_n,j_0,\ldots,j_n\geq
0$, one has
\begin{equation}
(i_0+i_1p+ \cdots +i_np^n)\;=\;p^n\;=\;(j_0+j_1p+\cdots+j_np^n)\;.
\end{equation}
Then the conditions $\lambda_{m-i}\in
\mathrm{Fil}_{\{d_1: p^i\}}(A)$, $\mu_{m-i}\in
\mathrm{Fil}_{\{d_2: p^i\}}(A)$ imply $\eta_{m-i}\in
\mathrm{Fil}_{\{(d_1+d_2) : p^i\} }(A)$.
\end{proof}

\subsubsection{Witt (co-)vectors of a graded ring.}\label{Witt (co-)vectors of a graded ring}
Let now $A=\oplus_{d\geq 0}\mathrm{Gr}_d(A)$ be a \emph{graded}
ring. We denote by
\begin{equation}
\W_m^{(d)}(A)\;,\qquad(\textrm{resp. } \CW^{(d)}(A))
\end{equation}
the subset of $\W_m(A)$ (resp. $\CW(A)$) formed by vectors
$(\lambda_0,\ldots,\lambda_m) \in \W_m(A)$ (resp. co-vectors
$(\cdots,0,0,\lambda_0,\ldots,\lambda_m) \in \CW(A)$) satisfying,
for all $i=0,\ldots,m$:

\begin{equation}
\lambda_{m-i}\in\mathrm{Gr}_{\{d:p^i\}}(A)\;,
\end{equation}

where by definition we set $\mathrm{Gr}_{-1}(A)=0$. It follows from  the equation \eqref{sum of witt vectors}, and from the
fact that $\mathrm{Gr}_{d_1}(A)\cdot \mathrm{Gr}_{d_2}(A)\subseteq
\mathrm{Gr}_{d_1+d_2}(A)$, that $\W_m^{(d)}(A)$ (resp.
$\CW^{(d)}(A)$) is closed under the sum in $\W_m(A)$ (resp.
$\CW(A)$). Moreover
\begin{equation}
\V(\W_{m}^{(d)}(A)) \; \; \subseteq \; \; \W_{m+1}^{(d)}(A)\;,
\end{equation}
and for all $d_1,d_2\geq 0$ one has
\begin{equation}
\W_m^{(d_1)}(A)\cdot \W_m^{(d_2)}(A) \; \subseteq \;
\W_m^{(d_1+d_2)}(A) \;.
\end{equation}
We notice that $\V(\W_{m}^{(d)}(A)) = \W_{m+1}^{(d)}(A)$ if
$d=np^v$, with $(n,p)=1$, and $v\leq m$. This proves that
\begin{equation}\label{CW^(d)(A)= varinjlim_m W_m^(d)(A)}
\CW^{(d)}(A)\;=\;\varinjlim_m\W_m^{(d)}(A)\;.
\end{equation}

\begin{lemma}\label{witt vector of a graduete lemma}
If $A=\oplus_{d\geq 0}\mathrm{Gr}_d(A)$ is graded, then the
filtration \eqref{Filtration of W_m(A)} on $\W_m(A)$ is a
grading and
\begin{equation}
\mathrm{Gr}_d(\W_m(A))\;=\; \W_m^{(d)}(A).
\end{equation}
The decomposition $\W_m(A)=\oplus_{d\geq
0}\W_m^{(d)}(A)=\oplus_{d\geq 0}\mathrm{Gr}_d(\W_m(A))$ pass to
the limit and defines a decomposition
\begin{equation}
\CW(A)\;=\;\oplus_{d\geq 0}\CW^{(d)}(A)\;=\;\oplus_{d\geq
0}\mathrm{Gr}_d(\CW(A))\;.
\end{equation}
\end{lemma}
\begin{proof}
We prove first that $\W_m(A)\cong\oplus_{d\geq 0}\W_m^{(d)}(A)$.
If $m=0$, then $\W_0(A)=A$, the lemma holds. Consider now a vector
$(\lambda_0,\ldots,\lambda_m)\in\W_m(A)$. If
$\lambda_0=\sum_{d\geq 0}\lambda_{0}^{(d)}\in A$, with
$\lambda_{0}^{(d)}\in\mathrm{Gr}_d(A)$, then
\begin{equation}
(\lambda_0,\lambda_1,\ldots,\lambda_m)=(\sum_{d\geq
0}\lambda_{0}^{(d)},0,\ldots,0)+(0,\lambda_1,\ldots,\lambda_m)=\sum_{d\geq
0}(\lambda_{0}^{(d)},0,\ldots,0)+(0,\mu_1,\ldots,\mu_m) ,
\end{equation}
for some $\mu_1,\ldots,\mu_m \in A$. The process can be iterated for 
$(0,\mu_1,\ldots,\mu_m)$.
This proves that $\W_m(A)$ is generated by the elements of the
form $(0,\ldots,0,x,0,\ldots,0)$, with $x\in\mathrm{Gr}_d(A)$,
placed in the $i$-th position, where $i$ (resp. $d$) runs on the
set $\{0,\ldots,m\}$ (resp. $\mathbb{Z}_{\geq 0}$). Hence
$\W_m(A)=\sum_{d\geq 0}\W_m^{(d)}(A)$.

We prove now that the sum is direct by induction on $m\geq 0$. For
$m=0$ this reduces to $A\cong\oplus_{d\geq 0}\mathrm{Gr}_d(A)$.
Let now $m\geq 1$. If $\sum_{d\geq 0}\lb^{(d)}=0$, with
$\lb^{(d)}\in \W_m^{(d)}(A)$, for all $d\geq 0$, then
$0=\sum_{d\geq 0}\lambda_0^{(d)}$, hence $\lambda_0^{(d)}=0$ for
all $d\geq 0$, because $A\cong\oplus_{d\geq 0}\mathrm{Gr}_d(A)$.
Then every $\lb^{(d)}$ belongs to $\V(\W_{m-1}^{(d)}(A))$. By
induction we have $\lb^{(d)}=0$ for all $d\geq 0$. Finally the
decomposition pass to the limit by \eqref{CW^(d)(A)= varinjlim_m
W_m^(d)(A)} and because every Witt co-vector belongs to
$\W_m(A)$, for some $m\geq 0$.
\end{proof}

\subsection{Settings}\label{setting}

\begin{hypothesis}\label{hypothesis of finite p-basis}
Let $p>0$, and let $q=p^h$, $h>0$. Let $k$ be a field of
characteristic $p$ containing $\mathbb{F}_q$. We assume that $k$
admits a finite $p$-basis $\{\bar{u}_1,\ldots,\bar{u}_r\}$.
\end{hypothesis}

\begin{notation}
\emph{For all field $\kappa$ of characteristic $p>0$, we denote by
$\C_{\kappa}$ a Cohen ring of $\kappa$ (i.e. a complete discrete
valued ring of characteristic $0$ with maximal ideal
$p\C_{\kappa}$ and with residue field $\kappa$.).}
\end{notation}
Let $\mathrm{E}$ be a complete discrete valuation  field of characteristic $p$,
with residue field $k$. Let $t$ be a uniformizer of $\O_{\E}$. If
a section $k\subset \O_{\E}$ of the projection $\O_{\E}\to k$ is
chosen, then one has an isomorphism (cf. \cite[\protect{Ch.9,$\S
3$,$\textsc{n}^{\mathrm{o}}3$,Th.1}]{Bou})
\begin{equation}\label{trivialization}
\mathrm{E}\cong k(\!(t)\!)\;,\quad \O_{\E}\cong k[\![t]\!]\;.
\end{equation}

\begin{lemma}\label{lifting of frobenius ggg} Let
$\kappa_1,\kappa_2$ be two fields, and let
$\bar{f}:\kappa_1\to\kappa_2$ be a morphism of fields. Let
$\{\bar{x}_i\}_{i\in I}$ be a $p$-basis of $\kappa_1$, and let
$\bar{y}_i:=\bar{f}(\bar{x}_i)$. Let $\C_{\kappa_1},\C_{\kappa_2}$
be two Cohen rings  with residue field
$\kappa_1,\kappa_2$ respectively, and let $\{x_i\}_{i\in
I}\subset\C_{\kappa_1}$, $\{y_i\}_{i\in I}\subset\C_{\kappa_2}$ be
arbitrary liftings of $\{\bar{x}_i\}_{i\in I}$ and
$\{\bar{y}_i\}_{i\in I}$. Then there exists a unique ring morphism
$f:\C_{\kappa_1}\to \C_{\kappa_2}$ of $\bar{f}$ sending $x_i$ into
$y_i$.
\end{lemma}
\begin{proof}
See \cite[$\S1$, $\textsc{n}^{\mathrm{o}}1$, Prop.2; Ex.4 of
$\S2$, p.70]{Bou}. See also \cite{Whitney}.
\end{proof}

Let $\C_k$ be a Cohen ring of $k$, and let $L_0$ be its field of
fractions. Since $\O_{\mathcal{E}_{L_0}}$ is a complete discrete
valued ring with maximal ideal generated by $p$, and with residue
field isomorphic to $\E$, then
\begin{equation}\label{C_E=O_E_L_0}
\C_k\cong\O_{L_0}\;,\qquad\C_{\E}\cong \O_{\mathcal{E}_{L_0}}\;.
\end{equation}

Let $K$ be a finite extension of $\mathbb{Q}_p$ with residue field
$\mathbb{F}_q$.  We set $ L := K \otimes_{\W(\mathbb{F}_q)}
\O_{L_{0}} $ and hence
\begin{equation}
 \O_L := \O_K \otimes_{\W(\mathbb{F}_q)} \O_{L_{0}} \;.
\end{equation}
One sees then that
$\O_{\mathcal{E}_L}=\O_{\mathcal{E}_{L_0}}\otimes_{\W(\mathbb{F}_q)}\O_K$.

\begin{definition}\label{Definition of Frobenius} Let $\{u_1,\ldots,u_r\}\subset\O_{L_0}$ be an arbitrary
lifting of the $p$-basis $\{\bar{u}_1,\ldots,\bar{u}_r\}$ of $k$
(cf. Hypothesis \ref{hypothesis of finite p-basis}). By Lemma
\ref{lifting of frobenius ggg}, we fix a  Frobenius $\varphi$ of
$\O_{L_0}$ sending $u_1,\ldots,u_r$ into $u_1^q,\ldots,u_r^q$. We
denote again by $\varphi$ the Frobenius
$\mathrm{Id}_{\O_K}\otimes\varphi$ on $\O_L$. We extend to
$\O_{\mathcal{E}_{L}}$, $\O_{\Ed_L}$, and $\R_L$ this Frobenius by
sending $T,u_1,\ldots,u_r$ into $T^q,u_1^q,\ldots,u_r^q$.
\end{definition}

The situation is exposed in the following diagram:
\begin{equation}
\xymatrix{ &\O_K\ar[r]&\O_L\ar[r]&\O_{\mathcal{E}_L}\\
\mathbb{Z}_p\ar[r]\ar[d]&\W(\mathbb{F}_q)\ar[r]\ar[u]\ar[d]&\O_{L_0}\ar[r]\ar[u]\ar[d]&\O_{\mathcal{E}_{L_0}}\ar[u]\ar[d]\\
\mathbb{F}_p\ar[r]&\mathbb{F}_q\ar[r]&k\ar[r]&\E}.
\end{equation}
We denote by $K$, $L_0$, $L$, $\mathcal{E}_L$,
$\mathcal{E}_{L_0}$, the fractions fields of $\O_K$, $\O_{L_0}$,
$\O_L$, $\O_{\mathcal{E}_L}$, and $\O_{\mathcal{E}_{L_0}}$
respectively.

\subsection{The Cohen rings of
\protect{$k_0(\bar{u}_1,\ldots,\bar{u}_r)$} and
\protect{$k_0(\!(\bar{u}_1)\!) \ldots  (\!( \bar{u}_r)\!)$}} \label{explicit
Cohen if k is polynomial} \label{Cohen ring of k_0((u_1,...u_r))}
Let $k_0$ be a perfect field, and let
$\{\bar{u}_1,\ldots,\bar{u}_r\}$ be an algebraically independent
family over $k_0$. In this subsection we give a description of the
Cohen ring $\C_k=\O_{L_0}$ in the case in which
$k=k_0(\bar{u}_1,\ldots,\bar{u}_r)$ or
$k=k_0(\!(\bar{u}_1)\!) \ldots  (\!( \bar{u}_r)\!)$.

Let $\{u_1,\ldots,u_r\}\subset \O_{L_0}$ be an algebraically
independent family over $F_0:=\mathrm{Frac}(\W(k_0))$, lifting
$\{\bar{u}_1,\ldots,\bar{u}_r\}$. We will define the Cohen ring of 
$k=k_0(\!(\bar{u}_1)\!) \ldots  (\!( \bar{u}_r)\!)$,  $\mathcal{E}_{F_0,(u_1, \dots , u_r)}$ inductively on $r$.  For $r=1$
we define  
$\mathcal{E}_{F_0,(u_1)}$ as the $F_0$-vector space
whose elements are series
$f(u_1)=\sum_{i\in\mathbb{Z}}a_{i}
u_1^{i}$, with $a_{i}\in F_0$,
satisfying
\begin{eqnarray}\label{def of p-ring of k_0((u))-1}
\sup_{i}|a_{i}|_{_{F_0}}& \;\; < \;\; &+\infty\;,\\
\lim_{i\to -\infty}|a_{i}|_{_{F_0}}& \;\; = \;\; &0\;.\label{def
of p-ring of k_0((u))-2}
\end{eqnarray}

This  is a complete field under the absolute value 
$$
|f(u_1)|_{\mathcal{E}_{F_0,(u_1)}}=\max_i |a_{i}|_{_{F_0}},
$$
 it is also unramified  over $\mathbb{Q}_p$ and its residue field is $ k_0(\!(\bar{u}_1)\!)$. Hence its  valuation ring is a Cohen ring of   $ k_0(\!(\bar{u}_1)\!)$. For $r=2$ using the case $r=1$  we can define $\mathcal{E}_{{\mathcal{E}_{F_0,(u_1)}},(u_2)}$. This is a $\mathcal{E}_{F_0,(u_1)}$ vector space whose elements are Laurent series in $u_2$ with coefficients in 
$\mathcal{E}_{F_0,(u_1)}$  satisfying the previous conditions  \eqref{def of p-ring of k_0((u))-1} and 
\eqref{def of p-ring of k_0((u))-2}.  It is a complete field  under the extension of the previous norm. In fact,  
if  $f(u_2)=\sum_{i\in\mathbb{Z}}a_{i}
u_2^{i}$, with $a_{i}\in \mathcal{E}_{F_0,(u_1)}$ we define 
$$
|f(u_2)|_{\mathcal{E}_{{\mathcal{E}_{F_0,(u_1)}},(u_2)}}=\max_i |a_{i}|_{\mathcal{E}_{F_0,(u_1)}},
$$
it is  again non ramified and its residue field is $k_0(\!(\bar{u}_1)\!)   (\!( \bar{u}_2)\!)$. This norm  is an extension of the Gauss norm on rational functions.  We may then continue for a general $r$ by defining 
$\mathcal{E}_{F_0,(u_1, \dots , u_r)}= \mathcal{E}_{\mathcal{E}_{F_0,(u_1, \dots , u_{r-1})},( u_r)}$ a Cohen ring for  $k_0(\!(\bar{u}_1)\!) \ldots  (\!( \bar{u}_r)\!)$.

\begin{remark}
If $k=k_0(\bar{u}_1,\ldots,\bar{u}_r)$, then the Cohen ring
$\C_{k}=\O_{L_0}$ is the ring of integers of the completion of
$F_0(u_1,\ldots,u_r)$, with respect to the Gauss norm.
\end{remark}

\subsection{Differentials}

We fix an arbitrary lifting $\{u_1,\ldots,u_r\}\subset
\C_k=\O_{L_0}$ of the \emph{finite} $p$-basis
$\{\bar{u}_1,\ldots,\bar{u}_r\}\subset k$ of $k$.
\begin{lemma}\label{Omega^1 is free}
Let $\widehat{\Omega}^1$ denote the module of continuous
differentials. The following assertions hold:
\begin{enumerate}
\item $\widehat{\Omega}_{\O_L/\O_K}^{1}$ is a free
$\O_L$-module of rank $r$, generated by $du_1,\ldots,du_r$, \item
$\mathrm{dim}_L\widehat{\Omega}_{L/K}^{1}=r$, and
$\widehat{\Omega}_{L/K}^{1}$
is generated by $du_1,\ldots,du_r$.
\end{enumerate}
\end{lemma}
\begin{proof}
The proof can be done by using \cite[$\bs{0}_{IV}$, Prop.19.8.2,
Th.20.4.9, Cor.20.4.10]{EGAIV-1}, \cite[$\bs{0}_{I}$,
Prop.7.2.9]{EGAI-1}, and \cite[$\bs{0}_{III}$, 10.1.3]{EGAIII-1},
\end{proof}

\label{u_1,....,u_r alge independent} The elements
$\bar{u}_1,\ldots,\bar{u}_r\in k$ are algebraically independent
over $\mathbb{F}_q$ (cf. \cite[Ch.V, $\S 8$ Ex.9]{Bou-Alg-4-5}).
So the elements $u_1,\ldots,u_r$ are algebraically independent
over $\mathbb{Q}_p$ and over $K$, since a polynomial relation
$P(u_1,\ldots,u_r)=0$ will imply, by reduction, a relation on
$\{\bar{u}_1,\ldots,\bar{u}_r\}$. Since $L$ has the $p$-adic
topology, hence the topology induced by $L$ on $K(u_1,\ldots,u_r)$
is the $p$-adic one, which coincides with that defined by the
Gauss norm $|\sum a_{i_1,\ldots,i_r}u_1^{i_1}\cdots
u_r^{i_r}|_{\mathrm{Gauss}}:=\sup |a_{i_1,\ldots,i_r}|_K$. In
particular $L$ contains the field
$\C_{\mathbb{F}_q(\bar{u}_1,\ldots,\bar{u}_r)}
\otimes_{\W(\mathbb{F}_q)}K= (K(u_1,\ldots,u_r),
|\cdot|_{\mathrm{Gauss}})\widehat{\;\;}$, where
$(\cdot)^{\widehat{\;}}$ means the \emph{completion with respect
to the Gauss norm}.

Let $I\subseteq\mathbb{R}_{\geq 0}$. Let
$\widehat{\Omega}^1_{\a_{L}(I)/K}$ denote \emph{continuous}
differentials, with respect to the topology of $\a_L(I)$ (cf.
section \ref{rings}), then
\begin{equation}
\widehat{\Omega}^1_{\a_{L}(I)/K} \;\; \cong \;\; \a_L(I)\cdot
dT\oplus\Bigl(\bigoplus_{i=1}^r\a_L(I)\cdot du_i\Bigr)\;.
\end{equation}
Consequently, also $\widehat{\Omega}^1_{\mathcal{E}_L/K}$,
$\widehat{\Omega}^1_{\Ed_L/K}$, $\widehat{\Omega}^1_{\R_L/K}$ are
freely generated by $dT,du_1,\ldots,du_r$.

\subsection{Spectral norms}

\label{spectral norm ???} For all complete valued ring
$(H,|.|_H)/(\mathbb{Z}_p,|.|)$, and all $\mathbb{Z}_p$-derivation
$\partial:H\to H$, we set
\begin{equation}
|\partial^n|_{H} := \sup_{h\in H, h \not=0}
|\partial^n(h)|_H/|h|_{H}\;,\qquad |\partial|_{H,\mathrm{Sp}} :=
\lim_{n\to\infty}|\partial^n|^{1/n}_{H} \;,
\end{equation}

This section is devote to proving  the following

\begin{lemma}[(\protect{\cite[Remark 2.1.5]{Ked-Swan}})]
\label{False lemma} Let $L$ be as in section \ref{setting}. Let
\begin{equation}
\omega\;:=\;|p|^{\frac{1}{p-1}}\;.
\end{equation}
Then one has
\begin{equation}\begin{array}{lclclcl}
|d/dT|_{\mathcal{F}_{L,\rho}}&=&\rho^{-1}\;,&&|d/du_i|_{\mathcal{F}_{L,\rho}}&=&1\;,\\
|d/dT|_{\mathcal{F}_{L,\rho},\mathrm{Sp}}&=&\omega\cdot\rho^{-1}\;,&&|d/du_i|_{\mathcal{F}_{L,\rho},\mathrm{Sp}}&=&\omega\;.\\
\end{array}
\end{equation}
\end{lemma}
\begin{proof}
The assertions on $d/dT$ are well known (cf. \cite{Ch}). We study
now only $d/du_i$. We split the proof in three lemmas:

\begin{lemma} For all
$i=1,\ldots,r$, one has
$|d/du_i|_{\mathcal{F}_{L,\rho}}=|d/du_i|_{L}$, and
$|d/du_i|_{\mathcal{F}_{L,\rho},\mathrm{Sp}}=|d/du_i|_{L,\mathrm{Sp}}$.
\end{lemma}
\begin{proof} It is enough to show that, for all $n\geq 0$, one has
$|(d/du_i)^n|_{\mathcal{F}_{L,\rho}} = |(d/du_i)^n|_{L}$. Since
$L\subset \mathcal{F}_{L,\rho}$, then
$|(d/du_i)^n|_{\mathcal{F}_{L,\rho}} \geq |(d/du_i)^n|_{L}$.
Conversely, for all $n\geq 0$, one has
$|(d/du_i)^n|_{\mathcal{F}_{L,\rho}} = \sup_{f\in
\mathcal{F}_{L,\rho}\;,\; |f|_\rho=1}|(d/du_i)^n(f)|_\rho$. We
observe that since $L[T,T^{-1}]$ is dense in
$\mathcal{F}_{L,\rho}$, we can consider the sup on the set
$\{f=\sum a_j T^j\in L[T,T^{-1}]\;,\;|f|_\rho=1\}$. For all such
$f=\sum a_j T^j$, one has $|(d/du_i)^n(\sum a_j T^j)|_\rho =
|\sum(d/du_i)^n(a_j) T^j|_\rho =\sup_j |(d/du_i)^n(a_j)| \rho^j
\leq\sup_j |(d/du_i)^n|_{L}|a_j| \rho^j=|(d/du_i)^n|_{L}|f|_\rho$,
hence $|(d/du_i)^n|_{\mathcal{F}_{L,\rho}} \leq |(d/du_i)^n|_{L}$.
\end{proof}

\begin{lemma}
One has $|d/du_i|_{L}=1$.
\end{lemma}
\begin{proof} 
 Since $d/du_i(\O_L)\subset\O_L$, one has $|d/du_i(a)|_L\leq
1=|a|$, for all $a\in\O_L$, $|a|=1$. Hence $|d/du_i|_{L}=
\sup_{|a|=1,a\in L} |d/du_i(a)|_L\leq 1$. Conversely, we have that 
$K(u_1,\ldots,u_r) \subset L$, and the valuation induced by $L$ is
the Gauss norm. If $\B$ denotes the completion of
$K(u_1,\ldots,u_r)$, then it is easy to prove explicitly, using
the description given in Section \ref{Cohen ring of
k_0((u_1,...u_r))}, that $|d/du_i|_{\B}=1$, and
$|d/du_i|_{\B,\mathrm{Sp}}=\omega$. The computations are analogous to the 
classical ones for $d/dT$ (cf.
\cite{Ch}).  Since $\B\subseteq L$, we have
the easy inequality
\begin{equation}\label{fffff_-1}
1=|d/du_i|_{\B}:=\sup_{ b\in \B\;,\; |b|=1} |d/du_i(b)|_{\B}\leq
\sup_{ b\in L\;,\; |b|=1 } |d/du_i(b)|_{L}=|d/du_i|_{L}\leq
1\;.\qquad
\end{equation}
\end{proof}

Analogously one proves that $|n!|=|(d/du_i)^n|_{\B}\leq
|(d/du_i)^n|_{L}$, $\omega=|d/du_i|_{\B,\mathrm{Sp}}\leq
|d/du_i|_{L,\mathrm{Sp}}$. It is now sufficient to prove that
$|d/du_i|_{L,\mathrm{Sp}}\leq \omega$. We will prove this
inequality first over $\mathrm{Frac}(\C_k)$, and then over $L$.

\begin{lemma}\label{ d/du_i = n!}
One has $|(d/d{u_i})^n|_{\mathrm{Frac}(\C_k)} = |n!|$.
\end{lemma}
\begin{proof}
Let $A:=\mathbb{F}_q(\bar{u}_1,\ldots,\bar{u}_r)\subseteq k$. Let
$\C_k$ and $\C_A$  be the Cohen rings attached to $k$ and $A$
respectively. Fix an inclusion $\C_{A}\subset\C_k$. Let
$\B:=\C_k^p+p\C_k$, then $\B$ is a closed sub-ring of $\C_k$, and
hence complete. Since $k=A\otimes_{A^p}k^p$, and since $(\B /
p\C_k) = k^p$, then $\C_k$ is generated by $\C_A$ and $\B$. By the
description given in Section \ref{Cohen ring of
k_0((u_1,...u_r))}, one sees that $|(d/du_i)^n|_{A} = |n!|$. On
the other hand, $d/du_i(\B)\subseteq p\C_k\subseteq \B$, hence
$|d/du_i|_{F_{\B}}\leq |p|$, where $F_{\B}$ is the completion of
$\mathrm{Frac}(\B)$. Hence $|(d/du_i)^n|_{F_{\B}}\leq |p|^n\leq
|n!|$. Since every element $x\in\C_k$ can be written as
$x=\sum_{i=1}^n a_i b_i$, with $a_i\in\C_A$, and $b_i\in \B$, then
$|(d/du_i)^n|_{\mathrm{Frac}(\C_k)} = |n!|$.
\end{proof}

{\emph{End of the proof of Lemma \ref{False lemma} :}} One has
$L=K\otimes_{\W(\mathbb{F}_q)}\C_k$. Every element $x\in L$ can be
written as $x=\sum_{i=1}^n a_ib_i$, $a_i\in K$, $b_i\in \C_k$.
Since $|(d/du_i)|_K=0$, one sees that $|(d/du_i)^n|_{L}=|n!|$.
Hence $|d/du_i|_{L,\mathrm{Sp}}=\omega$.
\end{proof}

\section{Radius of convergence and irregularities: the perfect residue field case}
\label{Radius of convergence and irregularities: the perfect residue field case}
This section is introductory. The main goal is to connect the
fundamental work of G.Christol and Z.Mebkhout with other results,
definitions, and notations introduced by B.Dwork, L.Garnier,
B.Malgrange, K.Kedlaya, P.Robba, N.Tsuzuki, for example. In all this section  we assume that $L$ has a perfect residual field $k$.

\subsection{Formal Irregularity}

The notion of irregularity of a differential equation finds its
genesis in \cite{Malgrange}, in which the definition has been
introduced for the first time for a differential operator with
coefficients in $\C(\!(T)\!)$, here  $\C$ is a field  of characteristic
$0$. We recall briefly the setting. If
$P(T,\frac{d}{dT}):=\sum_{k=0}^{n}g_k(T)(\frac{d}{dT})^k$, with
$g_k(T)\in \C(\!(T)\!)$, the \emph{Formal Irregularity}, and the
\emph{Formal Slope} of $P$ are defined as
\begin{eqnarray}
\mathrm{Irr}_{\mathrm{Formal}}(P)&:=&\max_{0\leq k\leq
n}\{k-v_T(g_k)\}-(n-v_T(g_n))\;,\\
\mathrm{Slope}_{\mathrm{Formal}}(P) &=& \max\Bigl(\;0\;,\;
\max_{k=0,\ldots,n}\Bigl(\;\frac{v_T(g_n)-v_T(g_k)}{n-k}-1\;\Bigr)
\;\Bigr)\;,
\end{eqnarray}
where $v_T$ is the $T$-adic valuation. One defines the
\emph{Formal Newton polygon} of $P$, denoted by $NP(P)$, as 
the convex hull in $\mathbb{R}^2$ of the set
$\{(\;k\;,\;(v_T(g_k)-k)-(v_T(g_n)-n)\;)\}_{k=0,\ldots,n}$
together with the extra points $ \{(-\infty,0)\}$ and
$\{(0,+\infty)\}$. The formal slope is then the largest slope of
the formal Newton polygon of $P$, and the \emph{irregularity} of
$P$ is the height of the Newton polygon:
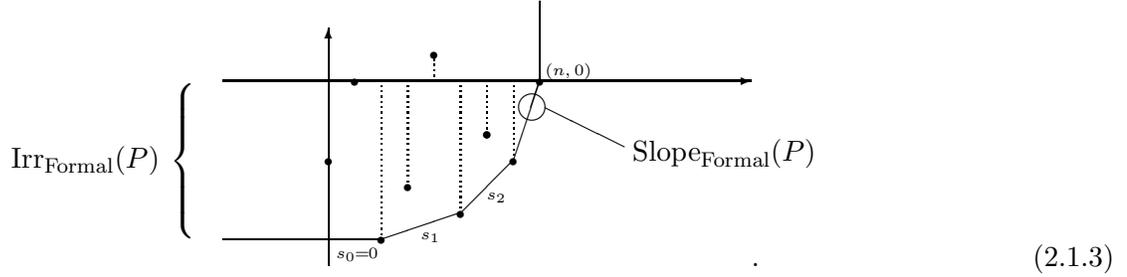
\begin{equation}
\begin{picture}(200,100) %
\put(40,0){\vector(0,1){90}} %
\put(0,70){\vector(1,0){200}} %

\put(0,10){\line(1,0){60}} %
\put(60,10){\line(3,1){30}} %
\put(90,20){\line(1,1){20}}
\put(110,40){\line(1,3){10}} %
\put(120,70){\line(0,1){30}} %

\put(117,60){\circle{10}} %
\put(122,60){\line(2,-1){30}} %
\put(155,40){$\mathrm{Slope}_{\mathrm{Formal}}(P)$}

\put(118,68){\begin{tiny}$\bullet$\end{tiny}}
\put(122,72){\begin{tiny}$(n,0)$\end{tiny}}

\qbezier[30](60,10)(60,40)(60,70)   
\put(58,8){\begin{tiny}$\bullet$\end{tiny}}
\qbezier[25](90,20)(90,45)(90,70)    
\put(88,18){\begin{tiny}$\bullet$\end{tiny}}
\qbezier[15](110,40)(110,55)(110,70) 
\put(108,38){\begin{tiny}$\bullet$\end{tiny}}

\qbezier[10](100,50)(100,60)(100,70) 
\put(98,48){\begin{tiny}$\bullet$\end{tiny}}

\qbezier[5](80,70)(80,75)(80,80) 
\put(78,78){\begin{tiny}$\bullet$\end{tiny}}

\qbezier[20](70,70)(70,50)(70,30) 
\put(68,28){\begin{tiny}$\bullet$\end{tiny}}

\qbezier[10](50,70)(50,70)(50,70) 
\put(48,68){\begin{tiny}$\bullet$\end{tiny}}

\put(38,38){\begin{tiny}$\bullet$\end{tiny}}

\qbezier[10](100,50)(100,60)(100,70) 
\put(98,48){\begin{tiny}$\bullet$\end{tiny}}

\put(-80,37){$\mathrm{Irr}_{\mathrm{Formal}}(P)$}
\put(-20,37){$\left\{
\begin{smallmatrix}
\\\\\\\\\\\\\\\\\\\\\\
\end{smallmatrix}\right.$} 

\put(43,3){\begin{tiny}$s_0\!\!=\!\!0$\end{tiny}}
\put(75,10){\begin{tiny}$s_1$\end{tiny}}
\put(100,25){\begin{tiny}$s_2$\end{tiny}}
\end{picture}.
\end{equation}
Let $\mathcal{D}:=\C(\!(T)\!)[\frac{d}{dT}]$. These definitions
are actually attached to the differential module
$\M=\mathcal{D}/\mathcal{D}\cdot P$ over $\C(\!(T)\!)$ defined by
$P$, and are independent on the particular \emph{cyclic} basis of
$\M$. We hence use the notation $NP(\M)$,
$\mathrm{Irr}_{\mathrm{Formal}}(\M)$,
$\mathrm{Slope}_{\mathrm{Formal}}(\M)$.  The differential module
$\M$ admits a so called \emph{break decomposition}
\begin{equation}
\M=\oplus_{s\geq 0}\M(s)\;,
\end{equation}
into $\mathcal{D}$-submodules, in which $\M(s)$ is characterized
by the fact that it is the unique sub-module of $\M$ whose Newton
polygon consists in a single slope $s$ (counted with multiplicity)
of the Newton polygon of $\M$. Hence the Formal irregularity can
be written as
\begin{equation}\label{formal break decomposition}
\mathrm{Irr}_{\mathrm{Formal}}(\M)=\sum_{s\geq
0}s\cdot\mathrm{dim}_{\C(\!(T)\!)}\M(s)\;.
\end{equation}

\subsubsection{Formal indices.}
We preserve the previous notations. It can be shown (cf.
\cite[1.3.1]{Malgrange}), that, if $L_P:\C[\![T]\!] \to
\C[\![T]\!]$ denotes the $\C$-linear map $f(T)\mapsto P(f(T))$,
then
\begin{equation}
\mathrm{Irr}_{\mathrm{Formal}}(\M)\;=\;\chi(\;L_P\;;\;\C[\![T]\!]\;)-(n-v_T(g_n))\;,
\end{equation}
where $\chi(L_P;\C[\![T]\!])=\dim_{\C}\Ker(L_P) -
\dim_{\C}\mathrm{Coker}(L_P)$.

Assume moreover that $\C=\mathbb{C}$ is the field of complex
numbers, and that $\M$ is a differential module over
$\mathbb{C}(\{T\}):=\mathrm{Frac}(\mathbb{C}\{T\})$, the fraction field of convergent power series. B.Malgrange
proved (cf. \cite[3.3]{Malgrange}) that if $G(T)\in
M_n(\mathbb{C}\{T\})$ is the matrix of $T\frac{d}{dT}$ acting on
$\M$, and if $T\frac{d}{dT}+\;\!^tG(T):\mathbb{C}(\{T\})^n\to
\mathbb{C}(\{T\})^n$ is the differential operator attached to $\M$
in this basis, then one has also
\begin{equation}
\mathrm{Irr}_{\mathrm{Formal}}(\M)\;=\;-\chi\Bigl(\;\;T\frac{d}{dT}+\;\!^tG\;\;;\;\;\mathbb{C}(\{T\})^n\;\;\Bigr)\;,
\end{equation}
where
$\chi(T\frac{d}{dT}+\;\!^tG;\mathbb{C}(\{T\})^n):=\dim_{\C}\Ker(T\frac{d}{dT}+\;\!^tG)
- \dim_{\C}\mathrm{Coker}(T\frac{d}{dT}+\;\!^tG)$. Moreover the
quantity $(n-v_T(g_n))$  can be related to  the
characteristic variety (at $0$) of the $\mathcal{D}$-module $\M$.

\subsection{$p$-adic framework, the irregularity of Robba and Christol-Mebkhout}
Assume now that $\C$ is a complete ultrametric field. Denote by
$\mathrm{C}(\{T\})$ the subfield of $\C(\!(T)\!)$ of  convergent  ($p$-adically) power series
around $0$ with meromorphic singularities in $0$.
It has been showed by F.Baldassarri (cf. \cite{Ba}), that if $\M$
is a $\C(\{T\})$-differential module, the above Break
decomposition of $\M\otimes_{\C(\{T\})}\C(\!(T)\!)$
\emph{descends} to a break decomposition of $\M$, over
$\C(\{T\})$. In this sense, from the point of view of the
Irregularity, the ``\emph{convergent theory}''  in the ultrametric setting over a germ of
punctured disk, with meromorphic singularities,   offers  nothing more than
the ``\emph{formal theory}''.

\subsubsection{The Robba ring as the completion of the generic point of a curve.}
A more interesting class of rings are those arising from Rigid
Geometry. Let $L$ be the field of Section \ref{setting}. Let $X$
be a projective, connected, non singular curve over $\O_L$, with
special fiber $X_{k}$ and generic fiber $X_L$. %
Let $X_L^{\mathrm{an}}$ be the rigid analytic curve over $L$ defined by
the generic fiber of $X$  (because $X$ was projective  this rigid analytic space also coincides with that obtained by completion along the special fiber). For every closed point $x_0\in X_k$
we denote by $]x_0[ \;\subset X_L^{\mathrm{an}}$ the tube of $x_0$
in $X_L^{\mathrm{an}}$. Let $x_0,\ldots,x_n\in X_k$ be a family of
closed points, and let $U_k:=X_k \setminus \{x_0,\ldots,x_n\}$ and denote $j$  the open immersion of $U_k$ in $ X_k$ . Let $]U_k[
:= X_L^{\mathrm{an}} \setminus  \cup_{i=0}^n]x_i[$ be its inverse image in
$X_L^{\mathrm{an}}$. We denote by $j^{\dag}\O_{]U_k[}$  the sheaf on $X_L^{\mathrm{an}}$  of functions 
overconvergent  along $X_k \setminus U_k$. Let $\mathcal{M}$ be a (locally
free)  $j^{\dag}\O_{]U_k[}$-sheaf of differential modules  admitting a Frobenius structure. By Frobenius structure we
mean the existence of an isomorphism
$(\varphi^*)^h(\mathcal{M})\xrightarrow[]{\;\;\sim\;\;}
\mathcal{M}$, for some $h\geq 0$, where $\varphi^*$ is associated with 
 the absolute Frobenius and
$(\varphi^*)^{h}:=\varphi^*\circ\cdots\circ\varphi^*$, $h$-times.
For all $i=0,\ldots,n$ we fix an isomorphism $\alpha_i:]x_i[\to\mathrm{D}^-(0,1)$. We define the Robba ring $\mathcal{R}_{x_i}$ as the pull back of $\mathcal{R}_L$ via the morphism $\alpha_i$.  Let
$V$ be a strict neighborhood of $]U_k[$. Every (overconvergent)
function of $\Gamma(V,\O_{X_L^{\mathrm{an}}})$ defines, via the isomorphism
$\alpha$, a germ of analytic function in an annulus
$1-\varepsilon<|T|<1$ of $\mathrm{D}^-(0,1)$.  For all
$i=0,\ldots,n$, and for all such $V$, we have then a natural
restriction morphism $\Gamma(V,\O_{X_L^{\mathrm{an}}}) \to \R_{x_i}$
compatible with the derivation, and with restrictions maps of the
sheaf $j^{\dag}\O_{]U_k[}$. Then it makes sense to consider the
restriction $\mathcal{M}_{x_i}$ of $\mathcal{M}$ to $]x_i[$ as a
differential module over $\R_{x_i}$.

The Robba ring $\R_L$ is actually isomorphic  to $\varinjlim_{V}
\Gamma(V\cap ]x_i[,\O_{X_L^{\mathrm{an}}})$ where $V$ runs on the set of
strict neighborhoods of $]U_k[$ in $X_L^{\mathrm{an}}$ (cf.
\cite{Cre}, see also \cite[After Cor.3.3]{Matsuda-unipotent}). In this sense
$\R_{x_i}$ is the analogous of the ring $\C(\!(T-x_i)\!)$ of the
previous section viewed as the completion at $x_i$ of the field of
functions of a non singular projective curve over $\C$.

\subsubsection{$p$-adic irregularity of differential modules over the Robba ring.}

With the notation of the previous section, let $d/dT+G$, $G\in
M_n(\R_L)$, be an operator attached to $\mathcal{M}_{x_i}$ in some
basis. Following P.Robba, G.Christol and Z.Mebkhout \cite[
8.3-8]{Ch-Me-3} we set:
\begin{equation}
\mathrm{Irr}_{x_i}(\mathcal{M}_{x_i}):=\widetilde{\chi}(\mathcal{M}_{x_i},\a_L(0,1))
\end{equation}
where $\widetilde{\chi}$ is the \emph{generalized index} of
$d/dT+G$ on $\a_L(0,1)^{n}$ (cf. \cite[8.2-1]{Ch-Me-3}). Notice
that $d/dT+G$ does not act on $\a_L(0,1)^{n}$ since $G$ has
coefficients in $\R_L$, this is the reason for which one introduces
the \emph{generalized} index. If $\mathcal{M}$ is an holonomic
differential module on  $X_L$ such that it is associated to an isocrystal overconvergent  along its singular points and endowed with a Frobenius structure, then $\mathcal{M}_{x_i}$ is a
$\a_L(0,1)$-differential module endowed with a Frobenius Structure, and one has (cf. \cite[8.3-9]{Ch-Me-3} and \cite[6.2.3]{Ch-Me-1})
\begin{equation}
\mathrm{Irr}_{x_i}(\mathcal{M}_{x_i}) =
\chi(\mathcal{M}_{x_i},\a_L(0,1)) -
(n-\mathrm{ord}^{-}_{x_i}(\mathcal{M}))
\end{equation}
where now $\chi$ denotes the index of $d/dT+G$ acting on
$\a_L(0,1)$, and $\mathrm{ord}^{-}_{x_i}(\mathcal{M})$ is the sum
of all the multiplicities of the vertical components of the
characteristic variety of $\mathcal{M}$ at every singular point
contained in $]x_i[ \; \cong  \mathrm{D}^-(0,1)$. We notice that
in this case one has moreover
\begin{equation}
\chi(\mathcal{M}_{x_i},\a_L(0,1))=
\mathrm{dim}_{L}\mathrm{Hom}_{\a_L[d/dT]}(\mathcal{M}_{x_i},\a_L(0,1))-
\mathrm{dim}_{L}\mathrm{Ext}^1_{\a_L[d/dT]}(\mathcal{M}_{x_i},\a_L(0,1))\;.
\end{equation}

\subsubsection{Geometric interpretation: formal and $p$-adic irregularities as slopes of the generic radius of convergence.}

We concentrate now our attention to differential modules over the
Robba ring. Robba first and then Christol-Mebkhout have indicated
that the behaviour of the \emph{generic Radius of convergence} of
the solutions of a finite free $\R_L$-differential module is
strictly connected to its irregularity. Let $\M$ be a finite free
$\R_L$-differential module. Let $\e$ be a fixed basis of $\M$, and
let $G(T)\in M_n(\R_L)$ be the matrix of the connection
$\nabla_T:\M\to\M$ in this basis. Let $\varepsilon>0$ be such that
$G(T)\in M_n(\a_L(1-\varepsilon,1))$. Consider the generic Taylor
solution
\begin{equation}
Y_G(x,y)\;:=\;\sum_{n\geq 0}G_n(y)\frac{(x-y)^n}{n!}
\end{equation}
where the matrices $\{G_n(T)\}_{n\geq 0}$ are defined inductively
by the rule $G_0=\mathrm{Id}$, $G_{1}=G$,
$G_{n+1}=d/dT(G_n)+G_nG_1$. The matrix $Y_G$ verifies
$(d/dx)^n(Y_G(x,y))=G_n(x)Y_G(x,y)$, for all $n\geq 0$. The
generic radius of convergence of $Y$ at $\rho$ is defined as
$Ray(Y_G,\rho):=\liminf_{n}(|G_n|_\rho/|n!|)^{-1/n}$. Since
$|G_n|_\rho=\max_{|y|=\rho}|G_n(y)|$, $Ray(Y_G,\rho)$ is the
minimum among all possible radii (with respect to $T$) of
$Y_G(T,y)$ with $|y|=\rho$. The quantity $Ray(Y_G,\rho)$ is not
invariant under base changes. Indeed in the basis $\B\e$, $\B\in
GL_n(\a_L(1-\varepsilon,1))$, the new solution is $\B\cdot Y_G$
which is possibly not convergent outside the disk
$\mathrm{D}^-(y,\rho)$ (where $|y|=\rho$). We set then
\begin{equation}
Ray(\M,\rho):=\min(Ray(Y,\rho),\rho)\;, 
\end{equation}
which is independent of the choice of  basis. The function $\rho\mapsto
Ray(\M,\rho):]1-\varepsilon,1[\to\mathbb{R}_+$ is continuous,
piecewise of the type $a\rho^b$, for convenable
$a,b\in\mathbb{R}$, and log-concave (cf. Section
\ref{log-definition}, \cite{Ch-Dw}). We refer to the fact that  $M$ has   $\lim_{\rho\to 1^-}Ray(\M,\rho)=1$ (cf.
\cite{Ch-Dw}),as  the  \emph{solvability} of
$\M$. This is the case, for example, if $M$ admits a Frobenius structure.

The Formal slope has a meaning in term of radius of convergence:

\begin{proposition}\label{Formal slope equals the slope of the radius at 0+}
Let $(\M,\nabla_T^{\M})$ be a solvable  differential module over
$L(\!(T)\!)\cap\R_L\;\subset\;\a_K(]0,1[)$. Then the Formal Slope
of $\M$ as differential module over $L(\!(T)\!)$ coincides with
the log-slope of the function $\rho\mapsto Ray(\M,\rho)/\rho$, for
$\rho$ sufficiently close to $0$.
\end{proposition}
\begin{proof}
We fix a cyclic basis for which $\M$ is represented by an operator
$\sum_{i=0}^ng_i(T)(d/dT)^{i}$, $g_n=1$. For all $f(T)\in
L(\!(T)\!)\cap\R_L$, if $\rho$ is close to zero, then the
log-slope of $\rho\mapsto |f(T)|_\rho$ is equal to $v_T(f(T))$.
Since $\rho\mapsto Ray(\M,\rho)$ is $\log$-concave, and since
$\lim_{\rho\to 1^{-}}Ray(\M,\rho)= 1$, then we have two
possibilities $Ray(\M,\rho)=\rho$ for all $\rho\in]0,1[$, or there
exists an interval $I=]0,\delta[$, such that, for all $\rho\in I$
one has $Ray(\M,\rho)< |\omega|\rho$, $\omega=|p|^{1/(p-1)}$. We can then apply
\cite[6.2]{Ch-Me}. $Ray(\M,\rho)=\rho$ if and only if
$v_T(g_i)\geq i-n$ for all $i<n$ i.e. if and only if the formal
slope is $0$, the second case arises if and only if $v_T(g_i)<
i-n$, for some $i$, and in this case one has the formula
$Ray(\M,\rho)=\rho\cdot|p|^{1/(p-1)}\cdot \min_{0\leq
i\leq}|g_i(T)|_\rho^{-1/(n-i)}$. This prove the proposition since
for $\rho$ close to zero $|g_i(T)|_\rho=a_i\rho^{v_T(g_i)}$, for
some $a_i\in\mathbb{R}_+$.
\end{proof}

If $(\M,\nabla_T^{\M})$ is solvable, one proves that there exists
$\beta\geq 0$, and $\varepsilon'>0$, such that
$Ray(\M,\rho)=\rho^{\beta}$, for all $1-\varepsilon'<\rho<1$. One
defines then the \emph{$p$-adic slope} of $\M$ as $\beta-1$.
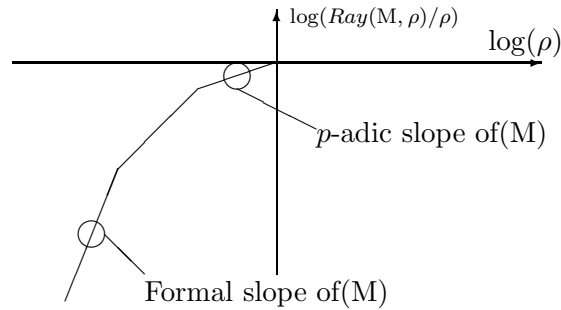
\begin{equation*}
\begin{picture}(210,105)
\put(0,5){\begin{picture}(210,110) %

\put(100,0){\vector(0,1){100}} %
\put(0,80){\vector(1,0){200}} %

\put(20,-10){\line(2,5){20}} %
\put(40,40){\line(1,1){30}} %
\put(70,70){\line(3,1){30}} %

\put(30,15){\circle{10}} %
\put(85,75){\circle{10}} %
\put(85,70){\line(2,-1){30}} %
\put(35,15){\line(1,-1){15}} %
\put(50,-10){$\textrm{Formal slope of}(\M)$} %
\put(115,50){$p\textrm{-adic slope of}(\M)$}

\put(180,85){$\log(\rho)$} 
\put(105,95){\begin{scriptsize}$\log(Ray(\M,\rho)/\rho)$\end{scriptsize}}
\end{picture}}
\end{picture}
\end{equation*}

Notice that the Formal slope is directly defined by the $T$-adic
valuations of the $g_i$'s whereas the $p$-adic slope is implicitly
defined by the coefficients. No explicit formulas expressing the
$p$-adic slope as function of the valuations of the $g_i$'s is
known. For this reason Christol and Mebkhout provide then a break
decomposition theorem for these $p$-adic slopes, reflecting the
analogous decomposition in the formal framework, and then define
the Newton polygon by means of the $p$-adic analogue of the
formula \eqref{formal break decomposition} given by the following:
\begin{theorem}[(\cite{Ch-Me-3})]
Let $(\M,\nabla_T^{\M})$ be $\R_L$-differential module endowed with a Frobenius structure (hence   solvable).
Then $\M$ admits a break decomposition $\M=\oplus_{\beta\geq
0}\M(\beta)$, where $\M(\beta)$ is characterized by the following
properties: there exists $\varepsilon
>0$ such that
\begin{enumerate}

\item For all $\rho\in]1-\varepsilon,1[$, $\M(\beta)$ is the
biggest submodule of $\M$  trivialized by every ring
$\a_L(y,\rho^{\beta+1})$, for all $\Omega/L$, and all $y\in\Omega$, with $|y|=\rho$;%

\item For all $\rho\in]1-\varepsilon,1[$, for all $|y|=\rho$,
$y\in\Omega$, for all $\Omega/L$, and for all $\beta' < \beta$,
$\M(\beta)$ has no solutions in $\a_L(y,\rho^{\beta'+1})$.
\end{enumerate}
The number $\mathrm{Irr}(\M):= \sum_{\beta\geq
0}\beta\cdot\textrm{rank}_{\R_K}(\M(\beta))$ is called
\emph{$p$-adic irregularity} of $\M$, and  lies in $\mathbb{N}$.
\end{theorem}

\begin{theorem} {\cite[8.3.7]{Ch-Me-3}}
Let $(\M,\nabla_T^{\M})$ be $\R_L$-differential module endowed with a Frobenius structure (hence   solvable). Then 
$\mathrm{Irr}(\M):= \sum_{\beta\geq
0}\beta\cdot\textrm{rank}_{\R_K}(\M(\beta))$  coincides with $\widetilde{\chi}({\M},\a_L(0,1))$ defined in subsection 2.2.2
\end{theorem}

\begin{remark}
In  practice the explicit computation of the $p$-adic slope is
possible only if it is equal to the formal slope, i.e. if the the
log-graphic of the function $\rho\mapsto Ray(\M,\rho)$ has no
breaks and if, for $\rho\in ]1-\varepsilon,1[$ sufficiently close
to $1-\varepsilon$, the $Ray(\M,\rho)$ is \emph{sufficiently
small} to be  computed explicitly  via \cite[6.2]{Ch-Me} as in
the proof of Proposition \ref{Formal slope equals the slope of the
radius at 0+}. This is often done by considering Frobenius
antecedents of $\M$.
\end{remark}

\subsubsection{Radius of convergence and spectral norm of the connection.}
\label{Radius/spectral}
The generic radius of convergence admits the following
description. For $\rho\in]1-\varepsilon,1[$ set
$\M_\rho:=\M\otimes_{\a_L(]1-\varepsilon,1[)}\mathcal{F}_{L,\rho}$.
Let $|.|_{\M_\rho}$ be a norm on $\M_\rho$ compatible with
$|.|_\rho$. Denote by $|\nabla_{T}^{\M_\rho}|_{\M_\rho}$ and
$|\nabla_{T}^{\M_\rho}|_{\M_\rho,\mathrm{Sp}}$ the norm and the
spectral norm of $\nabla_T^{\M_\rho}$ respectively as operator on
$\M_{\rho}$ (cf. def. \ref{spectral norm of connect}). Making
these definitions explicit  with respect to a basis of $\M$ for
which the matrix of the connection is $G$, one finds that
$|\nabla_{T}^{\M_\rho}|_{\M_\rho,\mathrm{Sp}}=\max(\;\omega\cdot\rho^{-1}\;,
\;\limsup_{n\to\infty}|G_n|^{1/n}_\rho\;)$, where
$\omega=|p|^{1/(p-1)}$ and where $G_n$ is the matrix of
$(\nabla_T^{\M})^n$ (cf. Lemma \ref{expliciting the definition of
the Swan conductor}). Hence by Lemma \ref{False lemma} one has
\begin{equation}
Ray(\M,\rho) =\rho\cdot\frac{|d/dT|_{\mathcal{F}_{L,\rho},\mathrm{Sp}}}{|\nabla_{T}^{\M_\rho}|_{\M_\rho,\mathrm{Sp}}}\;.
\end{equation}
The quantity
$\frac{|d/dT|_{\mathcal{F}_{L,\rho},\mathrm{Sp}}}{|\nabla_{T}^{\M_\rho}|_{\M_\rho,\mathrm{Sp}}}$
is more intrinsic, and actually the $p$-adic slope is defined as
the slope, for $\rho$ close to $1^{-}$, of the function
$\rho\mapsto
\frac{|d/dT|_{\mathcal{F}_{L,\rho},\mathrm{Sp}}}{|\nabla_{T}^{\M_\rho}|_{\M_\rho,\mathrm{Sp}}}$.

\section{Differential Swan conductor in
the non perfect residue field case}
\label{Kedlaya's differential Swan conductor}
In this section we recall some basic definitions given in
\cite{Ked}, generalizing to the non perfect residual case the more
classical analogous notions in   \cite{Fo} and \cite{Ts}.

\subsection{$(\phi,\nabla)$-modules and differential modules}

\label{phi-Nabla-mod and diff mod}

\begin{definition}
Let $\B_L$ be one of the rings $\mathcal{E}_L$, $\Ed_L$, $\R_L$,
$\a_L(I)$ (resp. $\O_{\mathcal{E}_L}$, $\O_{\Ed_L}$). A
\emph{differential module} $\M$ over $\B_L$ is a finite free
$\B_L$-module, together with an integrable connection
\begin{equation}
\M\xrightarrow[]{}\M\otimes \widehat{\Omega}^1_{\B_L/K}
\end{equation}
(resp. $\M\to\M\otimes\widehat{\Omega}^1_{\B_L/\O_K}$). Morphisms
between differential modules commute with the connections. The
category of differential modules will be called $\nabla \!-\!
\Mod(\B_L/K)$ (resp. $\nabla \!-\! \Mod(\B_L/\O_K)$).
\end{definition}

\if{
\begin{remark} The data of the integrable connection on $\M$
is equivalent to the data of the single action of
$d/dT,d/du_1,\ldots,d/du_r$ on $\M$. So $\M$ can be considered as
a finite free module over $\a_L(]\varepsilon,1[)$ (resp.
$\mathcal{E}_{L}$,$\Ed_L$,$\R_L$), together with $r+1$ connections
$\nabla_{T}^{\M},\nabla_{u_1}^{\M},\ldots,\nabla_{u_r}^{\M}:\M\to\M$,
satisfying
$\nabla_T^{\M}\circ\nabla_{u_i}^{\M}=\nabla_{u_i}^{\M}\circ\nabla_{T}^{\M}$,
and
$\nabla_{u_j}^{\M}\circ\nabla_{u_i}^{\M}=\nabla_{u_i}^{\M}\circ\nabla_{u_j}^{\M}$,
for all $i,j=1,\ldots,r$.
\end{remark}
}\fi

\begin{definition}
Let $\B_L$ be one of the rings $\mathcal{E}_L$, $\Ed_L$, $\R_L$,
$\O_{\mathcal{E}_L}$, $\O_{\Ed_L}$. A \emph{$\phi$-module} (resp.
a \emph{$(\varphi,\nabla)$-module}) over $\B_L$ is a finite free
$\B_L$-module (resp. $\nabla$-module) $\D$, together with an
isomorphism
\begin{equation}
\phi^{\D} :\varphi^*(\D) \simto \D
\end{equation}
of $\B_L$-modules (resp. of $\nabla$-modules). We interpret
$\phi^{\D}$ as a semi-linear action of $\varphi$ on $\D$.
Morphisms between $\phi$-modules (resp.
$(\varphi,\nabla)$-modules) commute with the Frobenius (resp. with
the Frobenius and the connection). The category of $\phi$-modules
(resp. $(\varphi,\nabla)$-modules) over $\B_L$ will be denoted by
$\varphi\!-\!\Mod(\B_L)$ (resp.
$(\varphi,\nabla)\!-\!\Mod(\B_L/K)$ or
$(\varphi,\nabla)\!-\!\Mod(\B_L/\O_K)$ if
$\B_L=\O_{\mathcal{E}_L},\O_{\Ed_L}$).
\end{definition}

\begin{notation}
We denote  $(\varphi,\nabla)$-modules with the letter $\D$, and
$\nabla$-modules with the letter $\M$.
\end{notation}

\subsection{Representations with finite local monodromy
and $(\varphi,\nabla)$-modules} \label{definition of the functor--}
We set
\begin{eqnarray}
\G_{\E}&:=&\mathrm{Gal}(\E^{\mathrm{sep}}/\E)\;,\\
\mathcal{I}_{\G_{\E}}&:=&\textrm{Inertia of }\G_{\E}\;,\\
\mathcal{P}_{\G_{\E}}&:=&\textrm{Wild inertia of }\G_{\E}\;.
\end{eqnarray}
\begin{definition}
We say that a representation $\alpha:\G_{\E}\to GL_n(\O_K)$ has
\emph{finite local monodromy} if the image of
$\mathcal{I}_{\G_{\E}}$ under $\alpha$ is finite. We denote by
$\V(\alpha)$ the representation defined by $\alpha$, and by
$\mathrm{Rep}^{\mathrm{fin}}_{\O_K}(\G_{\E})$ the category of
representations with finite local monodromy.
\end{definition}

Following \cite{Ked-Swan} (and \cite{Ts}) we denote by
$\widetilde{\Ed_L}$ be the quotient field of the ring
\begin{equation}\label{def. of tilde Ed_L}
\O_{\widetilde{\Ed_L}}:=\widehat{\O_{L}^{\mathrm{unr}}}
\otimes_{\O_{L}^{\mathrm{unr}}}\O_{\Ed_L}^{\mathrm{unr}} \; .
\end{equation}
We denote again by $\varphi$, $d/dT$, $d/du_1,\ldots,d/du_r$ the
unique extension to $\widetilde{\Ed_L}$ of $\varphi$, $d/dT$,
$d/du_1,\dotsc,$ $d/du_r$ on $\Ed_L$. If
$\V\in\mathrm{Rep}_{\O_K}^{\mathrm{fin}}(\G_{\E})$, we consider
$\V\otimes_{\O_K}\O_{\widetilde{\Ed_L}}$ together with the action
of $\G_{\E}$ given by $\gamma(v\otimes x):=\gamma(v)\otimes
\gamma(x)$, where $\gamma\in\G_{\E}$, $v\in\V$,
$x\in\widetilde{\Ed_L}$. We define
\begin{equation}
\D^{\dag}(\V)=(\V\otimes_{\O_K}\O_{\widetilde{\Ed_L}}
)^{\G_{\E}}\;,
\end{equation}
and we consider it as an object of
$(\varphi,\nabla)\!-\!\Mod(\O_{\Ed_L}/\O_K)$ with the action of
$\phi^{\D^{\dag}(\V)}:=1\otimes\varphi$,
$\nabla_T^{\D^{\dag}(\V)}:=1\otimes \frac{d}{dT}$,
$\nabla_{u_1}^{\D^{\dag}(\V)}:=1\otimes \frac{d}{du_1}$, \ldots,
$\nabla_{u_r}^{\D^{\dag}(\V)}:=1\otimes \frac{d}{du_r}$.
Reciprocally, for every
$(\D,\phi^{\D},\nabla^{\D})\in(\varphi,\nabla)\!-\!\Mod(\O_{\Ed_L}/\O_K)$
we consider $\D \otimes_{\O_{\Ed_L}} \O_{\widetilde{\Ed_L}}$
together with the action of $\phi^{\D} \otimes \varphi$,
$(\nabla_{T}^{\D}\otimes 1+1\otimes d/dT)$,
$(\nabla_{u_i}^{\D}\otimes 1+1\otimes d/du_i)$, $i=1,\ldots, r$.
We set then
\begin{equation}
\V^{\dag}(\D) :=
(\D\otimes_{\O_{\Ed_L}}\O_{\widetilde{\Ed_{L}}})^{(\phi^{\D}\otimes\varphi)=1}
\;,
\end{equation}
and we consider it as an object of $\mathrm{Rep}_{\O_K}(\G_{\E})$,
with the action of $\G_{\E}$ given by $\gamma(x\otimes
y):=x\otimes \gamma(y)$, with $\gamma\in\G_{\E}$, $x\in\D$,
$y\in\O_{\widetilde{\Ed}_L}$.

\begin{proposition}[(\protect{\cite[3.3.6]{Ked-Swan}} and
\cite{Ts})]\label{Kedlaya's equivalence} The above representation
$\V^{\dag}(\D)$ has finite local monodromy. Moreover the functor
\begin{equation}
\V\mapsto
\D^{\dag}(\V):\mathrm{Rep}_{\O_K}^{\mathrm{fin}}(\G_{\E})\longrightarrow
(\varphi,\nabla)\!-\!\Mod(\O_{\Ed_L}/\O_K)\;,
\end{equation}
is an equivalence of categories with quasi inverse
$\D\mapsto\V^{\dag}(\D)$.\hfill\CVD
\end{proposition}

\subsection{The differential Swan conductor.}

Let $\M\in \nabla\!-\!\Mod(\a_L(I))$. For all $\rho\in I$ we set
(cf. \eqref{F_rho})
\begin{equation}
\M_{\rho}:=\M\otimes_{\a_L(I)}\mathcal{F}_{L,\rho}\;.
\end{equation}

\begin{definition}\label{spectral norm of connect}
Let $\M_\rho$ be a $\nabla$-module over $\mathcal{F}_{L,\rho}$.
Let $|\cdot|_{\M_\rho}$ be a norm on $\M_\rho$ compatible with the
norm $|\cdot|_{\rho}$ of $\mathcal{F}_{L,\rho}$. For all
$\nabla^{\M_\rho}\in\{\nabla_T^{\M_\rho},\nabla_{u_1}^{\M_\rho},\ldots,\nabla_{u_r}^{\M_\rho}\}$,
we define
\begin{equation}
\begin{array}{lcl}
|\nabla^{\M_\rho}|_{\M_\rho}&:=&
\sup_{m\in\M_\rho}\frac{|\nabla^{\M_\rho}(m)|_{\M_\rho}}{|m|_{\M_\rho}}\;,\bigskip\\
|\nabla^{\M_\rho}|_{\M_\rho,\mathrm{Sp}}&:=&
\limsup_{n\to\infty}|(\nabla^{\M_\rho})^n|_{\M_\rho}^{1/n}\;.
\end{array}
\end{equation}
The definition of $|\nabla^{\M_\rho}|_{\M_\rho,\mathrm{Sp}}$ does
not depend on the chosen norm $|\cdot|_{\M_\rho}$, but only on the
norm $|\cdot|_{\rho}$ of $\mathcal{F}_{L,\rho}$ (cf. \cite[
1.1.7]{Ked-Swan}).
\end{definition}

\begin{definition}[\protect{(\cite[
2.4.6]{Ked-Swan})}]\label{definition of T(M,rho)} Let $\M$ be a
$\nabla$-module over $\a_L(I)$. We define the \emph{(toric)
generic radius of convergence} of $\M$ at $\rho$ as (cf. Lemma
\ref{False lemma}):
\begin{equation}\label{S(S_i,rho)}
T(\M,\rho) :=
\min\Bigl(\;\frac{|d/dT|_{\mathcal{F}_{L,\rho},\mathrm{Sp}}}{
|\nabla_T^{\M_\rho}|_{\M_\rho,\mathrm{Sp}}}\;,\;
\frac{|d/du_1|_{\mathcal{F}_{L,\rho},\mathrm{Sp}}}{|\nabla_{u_1}^{\M_\rho}|_{\M_\rho,\mathrm{Sp}}}\;,\;\ldots\;,\;
\frac{|d/du_r|_{\mathcal{F}_{L,\rho},\mathrm{Sp}}}{|\nabla_{u_r}^{\M_\rho}|_{\M_\rho,\mathrm{Sp}}}\;\Bigr)\;.
\end{equation}
\end{definition}

This definition is a generalization to the case of non perfect residual field of the
notion of radius of convergence as given at  the end of section  \ref{Radius of convergence and irregularities: the perfect residue field case}.

\begin{lemma}\label{expliciting the definition of the Swan conductor}
Let $\M_\rho$ be a $\nabla$-module over $\mathcal{F}_{L,\rho}$.
For $i=0,\ldots,r$ let $G_n^{i} \in M_d(\mathcal{F}_{L,\rho})$ be
the matrix of $(\nabla_{u_i}^{\M_\rho})^n$ (resp. if $i=0$,
$G^0_n$ is the matrix of $(\nabla_T^{\M_\rho})^n$). Then
\begin{equation}
|\nabla_{u_i}^{\M_{\rho}}|_{\M_\rho,\mathrm{Sp}}=\max(\;\omega\;,
\;\limsup_{n\to\infty}|G_n^i|^{1/n}_\rho\;)\;,
\end{equation}
where $\omega:=|p|^{\frac{1}{p-1}}$ (cf. Lemma \ref{False lemma}),
and
\begin{equation}
|\nabla_{T}^{\M_\rho}|_{\M_\rho,\mathrm{Sp}}=\max(\;\omega\cdot\rho^{-1}\;,
\;\limsup_{n\to\infty}|G_n^0|^{1/n}_\rho\;)\;.
\end{equation}
Hence, by Definition \ref{definition of T(M,rho)}, one has
\begin{equation}\label{S(S_i, rho) simpler basis}
T(\M_\rho,\rho)=\min_{i=1,\ldots,r}\Bigl(\;1\;,\;\omega\cdot\rho^{-1}\cdot
\Bigl[\liminf_{n\to\infty}|G_n^0|^{-1/n}_\rho\Bigr]\;,\;
\omega\cdot
\Bigl[\liminf_{n\to\infty}|G_n^i|^{-1/n}_\rho\Bigr]\;\Bigr)\;.
\end{equation}
\end{lemma}
\dem This follows directly from the definition \ref{definition of
T(M,rho)} (cf. \cite[(1.1.7.1)]{Ked-Swan}). \CVD
\medskip

The following definition generalizes that one given in section \ref{Radius/spectral} to the case of general residue field with finite $p$-basis.

\begin{definition}[\protect{(\cite[2.5.1]{Ked-Swan}})]
Let $\M$ be a $\nabla$-module over $\a_L(]\varepsilon,1[)$, with
$0\leq\varepsilon<1$. We will say that $\M$ is \emph{solvable} if
\begin{equation}
\lim_{\rho\to 1^{-}}T(\M,\rho)=1\;.
\end{equation}
\end{definition}
This is the case if the module has a Frobenius structure.

\subsubsection{}\label{log-definition} We recall that if $f :
\mathbb{R}_{> 0}\to\mathbb{R}_{> 0}$ is a function, the
$\log$-function $\widetilde{f} : \mathbb{R}\to \mathbb{R} $
attached to $f$ is defined as $\widetilde{f}(r) :=
\log(f(\exp(r)))$. Note that if $f(\rho)=a\rho^b$ (for
$\rho\in]\rho_1,\rho_2[$), then $\widetilde{f}(r)=\log(a)+b r$
(for $r\in]\log(\rho_1),\log(\rho_2)[$).

\begin{notation}
We will say that $f$ has  a given property \emph{logarithmically}
if $\widetilde{f}$ has that property.
\end{notation}

\begin{definition}[(\protect{\cite[ 2.5.6]{Ked-Swan}})]\label{Definition of Diff-Swan}
Let $\M\in \nabla-\Mod(\R_L/K)$ be a solvable $\nabla$-module. The
\emph{slope} of $\M$ (at $1^-$) is the log-slope of the function
$\rho\mapsto T(\M,\rho)$, for $\rho<1$ sufficiently close to
$1^{-}$. We denoted it by
\begin{equation}
\mathrm{slope}(\M,1^-)\;.
\end{equation}
\end{definition}

Let $\M\in\nabla-\Mod(\a_L(I)/K)$. In \cite{Ked-Swan} K.Kedlaya
proves that the function $\rho\mapsto T(\M,\rho)$ is
\emph{continuous}, \emph{$\log$-concave}, \emph{piecewise
$\log$-affine}\footnote{i.e. the function $\rho\mapsto T(\M,\rho)$
is locally of the type $a\rho^b$ in a partition of $I$.} and 
moreover, if $M$ is solvable,   that the Definition \ref{Definition of Diff-Swan} has a
meaning, since there exist a $\varepsilon<1$, and a
number $\beta$ such that $T(\M,\rho)=\rho^{\beta}$, for all
$\rho\in]\varepsilon,1[$. In this case the slope of $\M$ is (cf.
\cite[2.5.5]{Ked-Swan}):
\begin{equation}
\mathrm{slope}(\M,1^{-})=\beta.
\end{equation}
The $\log$-function $\log(\rho)\mapsto \log(T(\M,\rho))$ draws a
graphic as the following:
\begin{equation}
\begin{picture}(200,90) %
\put(100,0){\vector(0,1){90}} %
\put(0,65){\vector(1,0){200}} %

\put(40,25){\line(-2,-5){10}} %
\put(40,25){\line(1,1){30}} %
\put(70,55){\line(3,1){30}} %

\put(85,60){\circle{10}} %
\put(85,55){\line(2,-1){30}} %
\put(120,35){$\mathrm{slope}(\M,1^{-})$}

\put(180,70){$\log(\rho)$} %
\put(105,80){$\log(T(M,\rho))$}
\end{picture}
\end{equation}

Following \cite[Definitions 2.4.6, 1.2.3,
2.7.1]{Ked-Swan}\label{S(M,rho)} let $\M$ be a \emph{solvable}
$\nabla$-module over $\O_{\mathcal{E}_L^{\dag}}$. Assume that $\M$
is defined over $]\epsilon,1[$. Let
$\M_{\rho,1},\M_{\rho,2},\ldots,\M_{\rho,n}$ be the Jordan-H\"older factors
of $\M_{\rho}$, for $\rho\in ]\epsilon,1[$. We define $S(\M,\rho)$
as the multi-set whose elements are
$T(\M_{\rho,1},\rho),\ldots,T(\M_{\rho,n},\rho)$ with multiplicity
$\mathrm{dim}_{\mathcal{F}_\rho}\M_{\rho,1},\ldots,\mathrm{dim}_{\mathcal{F}_\rho}\M_{\rho,n}$
respectively. We will say that $\M$ has \emph{uniform break}
$\beta$ if there exists $\epsilon\leq\epsilon'<1$ such that, for
all $\rho\in ]\epsilon',1[$, $S(\M,\rho)$ consists in a single
element $\rho^\beta$ with multiplicity
$\mathrm{rank}_{\O_{\mathcal{E}^{\dag}_L}}\M$.

\begin{theorem}[(\protect{\cite[2.7.2]{Ked-Swan}})]\label{deco} We maintain the notations
of  section \ref{S(M,rho)}. Every indecomposable solvable  $\nabla$-module
over $\O_{\Ed_L}$ has a uniform break. In general, for $M$ a solvable $\nabla$-module  we have a
decomposition
\begin{equation}
\M=\bigoplus_{\beta \in \mathbb{Q}_{\geq 0}} \M_{\beta}\;,
\end{equation}
where $\M_{\beta}$ is a solvable  $\nabla$-module over
$\O_{\mathcal{E}^{\dag}_L}$-module, with \emph{uniform break}
$\beta$.\hfill\CVD
\end{theorem}

\begin{definition}[(differential Swan conductor \protect{\cite[2.8.1]{Ked-Swan}})]\label{definition of diff Swan conductor}
Let $\M$ be a \emph{solvable} $\nabla$-module over
$\O_{\mathcal{E}_L^{\dag}}$. Let $\M=\oplus_{\beta\in
\mathbb{Q}_{\geq 0}} \M_{\beta}$, be the decomposition of Theorem
\ref{deco}. We define the \emph{differential Swan conductor} of
$\M$ as
\begin{equation}
\mathrm{sw}^{\nabla}(\M)\;:=\;\sum_{\beta \in \mathbb{Q}_{\geq
0}}\beta
\cdot\mathrm{rank}_{\O_{\mathcal{E}^{\dag}_L}}\M_{\beta}\;.
\end{equation}
Moreover, for all
$\V\in\mathrm{Rep}^{\mathrm{fin}}_{\O_K}(\G_{\E})$, we set:
\begin{equation}
\mathrm{sw}^{\nabla}(\V)\;:=\;\sw^{\nabla}(\D^{\dag}(\V))\;.
\end{equation}
\end{definition}

\begin{remark}\label{RK1}
If $\mathrm{rank}_{\O_{\mathcal{E}^{\dag}_L}}\M=1$, then
$\mathrm{sw}^{\nabla}(\M)=\mathrm{slope}(\M,1^{-})$.
\end{remark}

\subsubsection{}\label{M_1 otimes M_2} The Toric Generic Radius of
convergence $T(\M,\rho)$ can be seen as a \emph{concrete} Radius
of convergence of  certain Taylor solutions (cf. \cite[Section
2.2]{Ked-Swan}). Hence, by the usual properties of the Radius of
convergence of  solutions of a differential equation one has
$\mathrm{sw}^{\nabla}(\M_1\otimes\M_2)\leq\max(\mathrm{sw}^{\nabla}(\M_1),\mathrm{sw}^{\nabla}(\M_2))$
(resp. $T(\M_1\otimes\M_2,\rho)\geq
\min(T(\M_1,\rho),T(\M_2,\rho))$), and moreover equality holds if
$\mathrm{sw}^{\nabla}(\M_1)\neq \mathrm{sw}^{\nabla}(\M_2)$ (resp.
$T(\M_1,\rho)\neq T(\M_2,\rho)$).

\subsubsection{Ramification filtration.}
The group $\G_{\E}$ is canonically imbedded into the Tannakian
group of the category
$\mathrm{Rep}^{\mathrm{fin}}_{\O_K}(\G_{\E})$. The definition of
the differential Swan conductor, together with Theorem \ref{deco},
define a \emph{ramification filtration} on $\G_{\E}$. Indeed, by
Theorem \ref{deco}, we apply the formalism introduced in \cite{An}
to define a filtration on the Tannakian group of
$\mathrm{Rep}^{\mathrm{fin}}_{\O_K}(\G_{\E})$. Hence $\G_{\E}$
inherits the filtration.

\section{Arithmetic Swan conductor for rank one representations}
\label{Arithmetic Swan conductor for rank one representations with
finite local monodromy}

In rank one case the arithmetic Swan conductor as defined by
K.Kato (cf. \cite{Kato-Swan} and Def. \ref{Kato definition of Swan
conductor}) coincides with that one of  A.Abbes and T.Saito
(cf. \cite{Ab-Sa-Ramification-groups}, \cite[9.10]{Ab-Sa}). In
this section we recall Kato's definition of the Swan conductor of
a rank one representation $\alpha : \G_{\E} \rightarrow
\mathbb{Q}/\mathbb{Z}$ and we describe completely Kato's
filtration on $\mathrm{H}^1(\G_{\E},\mathbb{Q}/\mathbb{Z})$: we
will use this later in our study of  rank one $p$-adic
representations with finite local monodromy. Indeed, in section
\ref{decomposition of G_E^ab}, we obtain the decomposition
$\G_{\E}^{\mathrm{ab}}=\mathcal{I}_{\G_{\E}^{\mathrm{ab}}}\oplus\G_{k}^{\mathrm{ab}}$.
This fact will allow us to \emph{define} the arithmetic Swan
conductor for rank one representations with finite local monodromy
(cf. Def. \ref{Definition of Swan for finite char}).

\if{ Let $\alpha:\G_{\E}\to\mathbb{Q}/\mathbb{Z}$ be a continuous
character. In the rank one case, the \emph{arithmetic Swan
conductor} of $\alpha$ defined by A.Abbes and T.Saito in
\cite{Ab-Sa-Ramification-groups} coincides with the Kato's Swan
conductor (cf. \cite{Kato-Swan}, Def. \ref{Kato definition of Swan
conductor}, cf. \cite[Th.9.10]{Ab-Sa}). In this section we recall
the definition of K.Kato, and describe completely the 's
filtration on $\mathrm{H}^1(\G_{\E},\mathbb{Q}/\mathbb{Z})$. In
section \ref{decomposition of G_E^ab} we obtain the decomposition
$\G_{\E}^{\mathrm{ab}}=\mathcal{I}_{\G_{\E}^{\mathrm{ab}}}\oplus\G_k$.
This fact will allow us to \emph{define} the arithmetic Swan
conductor for rank one representations with finite local monodromy
(cf. Def. \ref{Definition of Swan for finite char}). }\fi 

\begin{remark}
In all the section \ref{Arithmetic Swan conductor for rank one
representations with finite local monodromy}, $k$ is an arbitrary
field of characteristic $p$ (we do not assume that $k$ has a
finite $p$-basis).
\end{remark}

\subsection{Kato's arithmetic Swan conductor
for rank one representations}

In \cite{Kato-Swan} K.Kato defined the Swan conductor of a
character in
$\mathrm{H}^1(\G_{\E},\mathbb{Q}/\mathbb{Z})=\mathrm{Hom}^{\mathrm{cont}}(\G_{\E},\mathbb{Q}/\mathbb{Z})$.
He gave this definition by introducing a filtration on
$\mathrm{H}^1(\G_{\E},\mathbb{Q}_p/\mathbb{Z}_p)$ (cf. section
\ref{meaning of continuous characters}). This filtration is
defined as the image (via the Artin-Schreier-Witt sequence) of a
filtration of $\CW(\E)$. The filtration on $\CW(\E)$ is nothing
but that introduced in Section \ref{Witt co-vectors of a filtered
ring...li}, corresponding to the canonical filtration of $\E$
(actually a graduation) induced by the valuation (cf. Section
\ref{filtration of a valuation}). In this section we describe
completely these filtrations, without any assumption on $k$ (no
finiteness of the $p$-base is required),  and without using Kato's
logarithmic differentials and differential Swan conductors. Almost
all material of this section comes from \cite{Rk1}, we reproduce it
here for the convenience of the reader. What is new here is the
relations of the results of \cite{Rk1} with Kato's definitions. In
section \ref{Kato's refined Swan conductor} one finds a link with
the constructions of K.Kato.

\subsubsection{Kato's Filtration on $\W_m(\E)$ and $\CW(\E)$.}
\label{valuation on W_m} Let $R$ be a ring, and let $v:R \to
\mathbb{R}\cup\{\infty\}$ be a valuation on $R$ (i.e. satisfying
$v(0)=\infty$, $v(\lambda_1+\lambda_2)\geq \min(
v(\lambda_1),v(\lambda_2))$, and $v(\lambda_1\cdot\lambda_2)\geq
v(\lambda_1)+v(\lambda_2)$\;). The valuation $v$ of $R$ extends to
a valuation, denoted again by $v$, on the ring of Witt vectors
$\W_s(R)$ as follows:
\begin{equation}\label{definition of v of KAto}
v(\lambda_0, \ldots,\lambda_s)\;:=\; \min(\;p^s
v(\lambda_0)\;,\;p^{s-1}
v(\lambda_1)\;,\;\ldots\;,\;v(\lambda_s)\;)\;.
\end{equation}
\label{vV(lb)=v(lb)} The function $v:\W_s(R)\to \mathbb{R}$
verifies $v(\bs{0}) = +\infty$, $v(\lb_1+\lb_2)  \geq
\min(v(\lb_1),v(\lb_2))$, and $v(\lb_1\cdot\lb_2) \geq
v(\lb_1)+v(\lb_2)$, for all $\lb_{1},\lb_2\in\W_s(R)$. One has
also $v(\V(\lb))=v(\lb)$, and, if $R$ is an $\mathbb{F}_p$-ring,
then $v(\Fb(\lb))=p\cdot v(\lb)$. Hence $v$ extends to a valuation
\begin{equation}
v:\bs{\mathrm{CW}}(R)\to \mathbb{R}\cup\{+\infty\}\;.
\end{equation}

\begin{definition}[(Kato's filtration on $\W_s(\E)$)] Let
as usual $\E:=k(\!(t)\!)$, and let $v=v_t$ be the $t$-adic
valuation. The \emph{Kato's filtration on $\W_s(\E)$} is defined
as $\mathrm{Fil}_{-1}(\W_s(\E)):=0$, and, for all $d\geq 0$, as:
\begin{equation}
\mathrm{Fil}_d(\W_s(\E)):= \{\lb\in\W_s(\E)\;|\;v(\lb)\geq -d\}\;.
\end{equation}
\end{definition}

One has
\begin{equation}
\V(\;\mathrm{Fil}_d(\W_s(\E))\;)\;\subset\;\mathrm{Fil}_d(\W_{s+1}(\E))\;,
\qquad\Fb(\;\mathrm{Fil}_d(\W_s(\E))\;)\;\subseteq\;\mathrm{Fil}_{pd}(\W_s(\E))\;.
\end{equation}
Hence  Kato's filtration on $\W_s(\E)$ passes to the limit and
defines a filtration on $\bs{\mathrm{CW}}(\E)$. One has
$\mathrm{Fil}_{-1}(\CW(\E)):=0$ and, for all $d\geq 0$, one has
\begin{equation}
\mathrm{Fil}_d(\bs{\mathrm{CW}}(\E))\;\;:=\;\;
\{\lb\in\bs{\mathrm{CW}}(\E)\;|\;v(\lb)\geq
-d\}\;\;=\;\;\bigcup_{s\geq 0}\;\mathrm{Fil}_d(\W_s(\E))\;.
\end{equation}

\subsubsection{Kato's Filtration on
$\mathrm{H}^1(\G_{\E},\mathbb{Q}_p/\mathbb{Z}_p)$ and
$\mathrm{H}^1(\G_{\E},\mathbb{Q}/\mathbb{Z})$.}\label{meaning of continuous
characters} Let $\mathbb{T}:=\mathbb{R}/\mathbb{Z}$. The
\emph{Pontriagyn dual} of $\G_{\E}$ is by definition the discrete
abelian group $\Hom^{\mathrm{cont}}(\G_{\E},\mathbb{T})$. Every
proper closed subgroup of $\mathbb{T}$ is finite (cf. \cite[
2.9.1]{Ribes-Zalesskii}). The image of a continuous morphism
$\alpha:\G_{\E}\to \mathbb{T}$ is then finite, and hence
\begin{equation}
\Hom^{\mathrm{cont}}(\G_{\E},\mathbb{T})=
\Hom^{\mathrm{cont}}(\G_{\E},\mathbb{Q}/\mathbb{Z})=\bigcup_{n\geq
1}\Hom(\G_{\E},(n^{-1}\mathbb{Z})/\mathbb{Z})\;,
\end{equation}
where $\mathbb{Q}/\mathbb{Z}\subseteq \mathbb{T}$ is considered
with the discrete topology, so that every continuous
character $\alpha \in
\Hom^{\mathrm{cont}}(\G_{\E},\mathbb{Q}/\mathbb{Z})$ has finite
image in $\mathbb{Q}/\mathbb{Z}$. On the other hand, we recall
that if $\mathrm{A}$ is a finite abelian group with trivial
action of $\G_{\E}$, then
$\mathrm{H}^1(\G_{\E},\mathrm{A})=\Hom(\G_{\E},\mathrm{A})$ (cf.
\cite[VII, $\S 3$]{Se}), so we have
$\mathrm{H}^1(\G_{\E},\mathbb{Q}/\mathbb{Z})=\Hom^{\mathrm{cont}}(\G_{\E},\mathbb{Q}/\mathbb{Z})$.

\begin{definition}\label{definition of Fil_d(H^1(G_E,Q/Z))}
The \emph{Kato's filtration on
$\mathrm{H}^1(\G_{\E},\mathbb{Q}_p/\mathbb{Z}_p)=
\mathrm{Hom}^{\mathrm{cont}}(\G_{\E},\mathbb{Q}_p/\mathbb{Z}_p)$}
is defined as the image, under the morphism $\delta$ of the
Sequence \eqref{artin-schreier-diagram-covectors}, of the
filtration on $\bs{\mathrm{CW}}(\E)$:
\begin{equation}\label{Fil_d(H^1)}
\mathrm{Fil}_d(\mathrm{H}^1(\G_{\E},\mathbb{Q}_p/\mathbb{Z}_p))\;\;
:= \;\; \delta(\mathrm{Fil}_d(\bs{\mathrm{CW}}(\E)))\;.
\end{equation}
\end{definition}
\subsubsection{Arithmetic Swan conductor.}

\if{Here we define the arithmetic Swan conductor of a character of
$\G_{\E}$ with values in $\mathbb{Q}/\mathbb{Z}$. We will define
in section \ref{Swan cond aritm for rk1 repre with loc mon} the
arithmetic Swan conductor of a character with values in
$\O_K^{\times}$ with finite local monodromy (not necessarily with
finite image).}\fi

\begin{definition}[(Arithmetic Swan conductor)]\label{Kato definition of Swan conductor}
Let
$\alpha\in\mathrm{Hom}^{\mathrm{cont}}(\G_{\E},\mathbb{Q}/\mathbb{Z})$,
the \emph{arithmetic Swan conductor} of $\alpha$ (with respect to
$v=v_t$) is defined as
\begin{equation}
\mathrm{sw}(\alpha)= \min\{\; d\geq 0 \;\;|\;\; \alpha_p
\in\mathrm{Fil}_{d}(\mathrm{H}^{1}(\mathrm{G}_\E,\mathbb{Q}_p/\mathbb{Z}_p))\;\}\;,
\end{equation}
where $\alpha_p$ is the image of $\alpha$ under the projection
\begin{equation}\label{Q/Z--->Q_p/Z_p}
\mathrm{Hom}^{\mathrm{cont}}(\G_{\E},\mathbb{Q}/\mathbb{Z})=
\bigoplus_{\ell=\textrm{prime}}\mathrm{Hom}^{\mathrm{cont}}(\G_{\E},\mathbb{Q}_\ell/\mathbb{Z}_\ell)
\longrightarrow\mathrm{Hom}^{\mathrm{cont}}(\G_{\E},\mathbb{Q}_p/\mathbb{Z}_p)\;,
\end{equation}
where the word ``$\mathrm{cont}$'' means, as usual, that the images
of the homomorphisms are finite (cf. Section \ref{meaning of
continuous characters}). We then define  \emph{Kato's
filtration on
$\mathrm{H}^1(\G_{\E},\mathbb{Q}/\mathbb{Z})=\mathrm{Hom}^{\mathrm{cont}}(\G_{\E},\mathbb{Q}/\mathbb{Z})$},
by taking
$\mathrm{Fil}_d(\mathrm{H}^1(\G_{\E},\mathbb{Q}/\mathbb{Z}))$ to be 
the inverse image of
$\mathrm{Fil}_d(\mathrm{H}^1(\G_{\E},\mathbb{Q}_p/\mathbb{Z}_p))$,
via the morphism \eqref{Q/Z--->Q_p/Z_p}.
\end{definition}

\begin{remark}
In section \ref{decomposition of G_E^ab} we will generalize this
definition to all rank one representations with \emph{finite local
monodromy} of $\mathrm{G}_{\E}$, i.e. characters
$\alpha:\G_{\E}\to \O_K^{\times}$ such that
$\alpha(\mathcal{I}_{\G_{\E}})$ is finite.
\end{remark}

\subsection{Description of \protect{$\mathrm{H}^1(\G_{\E},\mathbb{Q}_p/\mathbb{Z}_p)$}
and computation of
$\mathrm{Fil}_d(\mathrm{H}^1(\G_{\E},\mathbb{Q}/\mathbb{Z}))$ }
\label{Description of H^1 and Fil_d}

In this section we recall an explicit description of
$\mathrm{H}^1(\G_{\E},\mathbb{Q}_p/\mathbb{Z}_p)$ that was
obtained in \cite[Section 3.2, and Lemma 4.1]{Rk1}. The proofs of
the statements are in \cite{Rk1}. In order to study Swan
conductor, we are interested in
$\mathrm{Hom}^{\mathrm{cont}}(\G_{\E},\mathbb{Q}_p/\mathbb{Z}_p)$.
The Artin-Schreier-Witt sequence
\eqref{artin-schreier-diagram-covectors} describes the latter as
\begin{equation}
\mathrm{Hom}^{\mathrm{cont}}(\G_{\E},\mathbb{Q}_p/\mathbb{Z}_p)=
\frac{\bs{\mathrm{CW}}(\E)}{(\Fb-1)\bs{\mathrm{CW}}(\E)}\;.
\end{equation}
Now the $\Fb$-module $\bs{\mathrm{CW}}(\E)$ admits the following
decomposition in $\Fb$-sub-$\mathbb{Z}$-modules (cf. \cite[Lemma
3.4]{Rk1}, cf. Section \ref{CW(I)}):
\begin{equation}\label{decomposition of CW(E) as - 0 +}
\bs{\mathrm{CW}}(k(\!(t)\!)) = \bs{\mathrm{CW}}(t^{-1}k[t^{-1}])
\oplus \bs{\mathrm{CW}}(k) \oplus \bs{\mathrm{CW}}(tk[[t]])\;.
\end{equation}
Moreover, one proves that
$\bs{\mathrm{CW}}(tk[[t]])/(\Fb-1)(\bs{\mathrm{CW}}(tk[[t]]))=0$
(cf. \cite[Prop.3.1]{Rk1}). Hence
\begin{equation}
\mathrm{H}^1(\G_{\E},\mathbb{Q}_p/\mathbb{Z}_p)=
\frac{\bs{\mathrm{CW}}(t^{-1}k[t^{-1}])}{(\Fb-1)(\bs{\mathrm{CW}}(t^{-1}k[t^{-1}]))}
\oplus \frac{\bs{\mathrm{CW}}(k)}{(\Fb-1)(\bs{\mathrm{CW}}(k))}\;.
\end{equation}
We will see that this decomposition corresponds (via the
Pontriagyn duality) to a decomposition of the $p$-primary part of
$\G_{\E}^{\mathrm{ab}}$ into the wild inertia
$\mathcal{P}_{\mathrm{G}^{\mathrm{ab}}_{\E}}$ and the $p$-primary
part of
$\G_k^{\mathrm{ab}}=\mathrm{Gal}(k^{\mathrm{sep}}/k)^{\mathrm{ab}}$
(cf. Proof of Proposition \ref{Decomposition of G_E^ab prop}).

Since the Swan conductor of an element of
$\frac{\bs{\mathrm{CW}}(k)}{(\Fb-1)(\bs{\mathrm{CW}}(k))}$ is $0$,
 we are led to study
$\bs{\mathrm{CW}}(t^{-1}k[t^{-1}])$ (corresponding to the wild
inertia of $\mathrm{G}^{\mathrm{ab}}_{\E}$, cf. Formula
\eqref{splitting p-primary}).

\subsubsection{Description of $\CW(t^{-1}R[t^{-1}])$.}
In this sub-section  $R$ denotes an arbitrary ring, $v_p(d)$
denotes the $p$-adic valuation of $d$, and $v=v_t$ denotes the
$t$-adic valuation  on $R[[t]][t^{-1}]$. The decomposition
\eqref{decomposition of CW(E) as - 0 +} holds also for
$R[[t]][t^{-1}]$.

\begin{definition}[(cf. \protect{\cite[Def.3.3]{Rk1}})]
Let $d,n,m\in\mathbb{N}$ be such that $d=np^m>0$, with $(n,p)=1$.
Let $\lb:=(\lambda_0,\ldots,\lambda_m)\in \W_m(R)$. We set
\begin{equation}
\lb\cdot t^{-d} :=
(\cdots,0,0,0,\lambda_0t^{-n},\lambda_1t^{-np},\ldots,\lambda_mt^{-d})\;\in\;
\bs{\mathrm{CW}}(t^{-1}R[t^{-1}])\;.
\end{equation}
We call $\lb t^{-d}$ the \emph{co-monomial of degree $-d$ relative
to $\lb$}. We denote by $\bs{\mathrm{CW}}^{(-d)}(R)$ the subgroup
of $\bs{\mathrm{CW}}(t^{-1}R[t^{-1}])$ formed by co-monomials of
degree $-d$.
\end{definition}

With the notation of Sections \ref{Witt (co-)vectors of a
graded ring} and \ref{filtration of a valuation}, if
$A=R[T^{-1}]$, then one has for all $d>0$
\begin{equation}
\CW^{(-d)}(R)\;=\;\CW^{(d)}(A)\;.
\end{equation}

\begin{proposition}[(\protect{\cite[Remark 3.3, Lemma 3.4]{Rk1}})]
Let $d=np^m>0$, $(n,p)=1$. The map
\begin{equation}\label{isomorphism between W_m and CW^(-d)}
\W_m(R)\xrightarrow[]{\;\;\sim\;\;}\bs{\mathrm{CW}}^{(-d)}(R)
\end{equation}
sending $\lb\in\W_m(R)$ into $\lb t^{-d}\in\CW(R)^{(-d)}$ is an
isomorphism of groups. Moreover, one has
\begin{equation}
\CW(t^{-1}R[t^{-1}])=\oplus_{d>0}\CW^{(-d)}(R)\;\;\cong\;\;
\oplus_{d>0}\W_{v_p(d)}(R)\;,
\end{equation}
where $v_p(d)$ is the $p$-adic valuation of $d$. In other words,
every co-vector $\bs{f}^{-}(t)\in\CW(t^{-1}R[t^{-1}])$ can be
uniquely written as a finite sum
\begin{equation}
\bs{f}^{-}(t)=\sum_{d>0}\lb_{-d}t^{-d}\;,
\end{equation}
with  $\lb_{-d}\in\W_{v_p(d)}(R)$, for all $d\geq 0$.\hfill\CVD
\end{proposition}

This decomposition extends to $\CW(t^{-1}R[t^{-1}])$ the trivial
decomposition $t^{-1}R[t^{-1}]=\oplus_{d>0} R \cdot t^{-d}$.
Moreover a co-vector $\bs{f}(t)\in\CW(R[[t]][t^{-1}])$ can be
uniquely written as a finite sum
\begin{equation}
\bs{f}(t) = \sum_{d>0}\lb_{-d}t^{-d} + \bs{f}_0 + \bs{f}^+(t) \;,
\end{equation}
with $\bs{f}^+(t)\in\CW(t R[[t]])$, $\bs{f}_0\in\CW(R)$,
$\lb_{-d}\in\W_{v_p(d)}(R)$, for all $d>0$. We denote by
$\bs{f}^{-}(t)$ the co-vector $\sum_{d>0}\lb_{-d}t^{-d}
\in\CW(t^{-1}R[t^{-1}])$.

\begin{corollary}\label{ Fil CW}
The module $\CW(R[[t]][t^{-1}])$ is graded. Moreover for all
$d\geq 0$, one has:
\begin{eqnarray*}
\mathrm{Fil}_{d}(\CW(R[[t]][t^{-1}]))\quad &=&\quad \oplus_{1\leq
d'\leq
d} \CW^{(-d')}(R) \oplus \CW(R) \oplus \CW(tR[[t]])\;,\\
&&\\
\mathrm{Gr}_{d}(\CW(R[[t]][t^{-1}]))\quad&=&\quad\mathrm{Fil}_{d}/\mathrm{Fil}_{d-1}=
\left\{\begin{array}{lcl} \CW^{(-d)}(R)&\textrm{if}&d>0\;,\\&&\\
\CW(R[[t]])&\textrm{if}& d = 0\;.
\end{array}\right.
\end{eqnarray*}
\end{corollary}
\begin{proof} The proof results from Lemma \ref{witt vector of a graduete lemma}.
For all $\lb\in\W_{v_p(d)}(R)$, $\lb\neq \bs{0}$, one has (cf.
Def. \eqref{definition of v of KAto}) $v_t(\lb t^{-d})=-d$. Hence
a co-vector
$\bs{f}(t)=\sum_{d>0}\lb_{-d}t^{-d}+\bs{f}_0+\bs{f}^+(t) \in
\CW(R[[t]][t^{-1}])$ lies in
$\mathrm{Fil}_{d}(\CW(R[[t]][t^{-1}]))$, $d\geq 0$, if and only if
$\lb_{-d'}=\bs{0}$, for all $d' < -d$.
\end{proof}

\subsubsection{Action of Frobenius, and description of $\mathrm{H}^1(\G_{\E},\mathbb{Q}_p/\mathbb{Z}_p)$.}
\label{action on CW^(-d)} We want now to study the action of $\Fb$
on $\CW(\E)$, in order to describe $\CW(\E)/(\Fb-1)(\CW(\E))$.

One sees that $\Fb(\lb\cdot t^{-d}) = \V\Fb(\lb)\cdot t^{-pd}$,
then one has the following commutative diagram, where the
horizontal arrows are the isomorphisms \eqref{isomorphism between
W_m and CW^(-d)}:
\begin{equation}
\xymatrix{\W_{v_p(d)}(k)\ar[r]^-{\sim}\ar[d]_{p=\V\Fb}\ar@{}[dr]|{\odot}&\ar@{}[dr]|{\odot}
\CW^{(-d)}(k)\ar[d]^{\Fb}\ar@{}[r]|-{\subset}& \CW(t^{-1}k[t^{-1}])\ar[d]^{\Fb}\phantom{\;.}\\
\W_{v_p(d)+1}(k)\ar[r]^-{\sim}&\CW^{(-pd)}(k)\ar@{}[r]|-{\subset}&\CW(t^{-1}k[t^{-1}])\;.\\}
\end{equation}
The Frobenius $\Fb$ acts then on the family
$\{\CW^{(-d)}(k)\}_{d\geq 1}$, as indicated in the following
picture:
\begin{equation}
\begin{scriptsize}
\begin{picture}(120,60)
\put(0,5){\vector(0,1){45}} \put(0,5){\vector(1,0){120}}
\put(-10,45){$m$}\put(120,0){$n$}
\put(125,43){$d=np^m$}

\put(17.5,3.25){$\bullet$} \put(37.5,3.25){$\bullet$}
\put(57.5,3.25){$\bullet$} \put(77.5,3.25){$\bullet$}
\put(97.5,3.25){$\bullet$}
\put(17.5,22.5){$\bullet$} \put(37.5,22.5){$\bullet$}
\put(57.5,22.5){$\bullet$} \put(77.5,22.5){$\bullet$}
\put(97.5,22.5){$\bullet$}
\put(17.5,42.5){$\bullet$} \put(37.5,42.5){$\bullet$}
\put(57.5,42.5){$\bullet$} \put(77.5,42.5){$\bullet$}
\put(97.5,42.5){$\bullet$}
%

\put(17.5,12.5){$\uparrow$}\put(37.5,12.5){$\uparrow$}
\put(57.5,12.5){$\uparrow$}\put(77.5,12.5){$\uparrow$}
\put(97.5,12.5){$\uparrow$}
\put(17.5,32.5){$\uparrow$}\put(37.5,32.5){$\uparrow$}
\put(57.5,32.5){$\uparrow$}\put(77.5,32.5){$\uparrow$}
\put(97.5,32.5){$\uparrow$}
\put(17.5,52.5){$\uparrow$}\put(37.5,52.5){$\uparrow$}
\put(57.5,52.5){$\uparrow$}\put(77.5,52.5){$\uparrow$}
\put(97.5,52.5){$\uparrow$}
\put(102,31){\begin{tiny}$\Fb\V$\end{tiny}}
\put(102,11){\begin{tiny}$\Fb\V$\end{tiny}}
\put(102,51){\begin{tiny}$\Fb\V$\end{tiny}}
\end{picture}
\end{scriptsize}
\end{equation}
where $d=np^m\geq 1$, with $(n,p)=1$, and $m=v_p(d)$. Hence for
all $n\geq 1$, $(n,p)=1$,  the subgroup
\begin{equation}
\bs{\C}_n(k):=\oplus_{m\geq 0}\CW^{(-np^m)}(k)
\end{equation}
is an sub-$\Fb$-module of $\CW(\E)$, and
\begin{equation}
\frac{\CW(t^{-1}k[t^{-1}])}{(\Fb-1)(\CW(t^{-1}k[t^{-1}]))}=
\bigoplus_{n\in\J}\frac{\bs{\C}_n(k)}{(\Fb-1)(\bs{\C}_n(k))}\;,
\end{equation}
where
\begin{equation}
\J:=\{n\in\mathbb{N}\;|\; (n,p)=1, n\geq 1\}\;.
\end{equation}
One sees that
\begin{eqnarray}
\qquad\frac{\bs{\C}_n(k)}{(\Fb-1)(\bs{\C}_n(k))}&=&\varinjlim_{m\geq
0}(\CW^{(-np^m)}(k)\xrightarrow[]{\Fb}\CW^{(-np^{m+1})}(k)\xrightarrow[]{\Fb}\cdots)\\
&\stackrel{(*)}{\cong}&\varinjlim_{m\geq
0}(\W_m(k)\xrightarrow[]{p}\W_{m+1}(k)\xrightarrow[]{p}\cdots)=\widetilde{\CW}(k)\;,
\end{eqnarray}
where the isomorphism $(*)$ is deduced by the isomorphism
\eqref{isomorphism between W_m and CW^(-d)}.

\begin{theorem}\label{description of H^1} The following statements hold:
\begin{enumerate}\item One has:
\begin{equation}
\mathrm{H}^1(\G_{\E},\mathbb{Q}_p/\mathbb{Z}_p)\cong
\frac{\CW(\E)}{(\Fb-1)(\CW(\E))}\;\;\cong\;\;
\widetilde{\CW}(k)^{(\J)}\oplus \frac{\CW(k)}{(\Fb-1)(\CW(k))}\;,
\end{equation}
where $\widetilde{\CW}(k)^{(\J)}$ is the direct sum of copies
of $\widetilde{\CW}(k)$, indexed by $\J$.%
\item  For $d=0$ one has
$\mathrm{Fil}_0(\mathrm{H}^1(\mathrm{G}_{\mathrm{E}},\mathbb{Q}_p/\mathbb{Z}_p))=\boldsymbol{\mathrm{CW}}(k)/(\bar{\mathrm{F}}-1)\boldsymbol{\mathrm{CW}}(k)$,
and for $d\geq 1$ one has  (cf. \eqref{Natural filtration of
CW and CW tilde})
\begin{eqnarray*}
\mathrm{Fil}_{d}(\mathrm{H}^{1}(\G_{\E},\mathbb{Q}_p/\mathbb{Z}_p))&=&
\oplus_{n\in\J}\Bigl(\mathrm{Fil}_{m_{n,d}}(\widetilde{\CW}(k))
\Bigr) \oplus \frac{\CW(k)}{(\Fb-1)(\CW(k))}\;,\\
&&\\
\mathrm{Gr}_{d}(\mathrm{H}^{1}(\G_{\E},\mathbb{Q}_p/\mathbb{Z}_p))&=&\left\{
\begin{array}{lcl}
\W_{v_p(d)}(k)/p\W_{v_p(d)}(k)&\textrm{if}&d>0\;,\\
&&\\
\frac{\CW(k)}{(\Fb-1)(\CW(k))}&\textrm{if}&d=0\;,
\end{array}\right.
\end{eqnarray*}
where $m_{n,d}:=\max\{m\geq 0\;|\; np^m\leq d\}$, and  $v_p(d)$ is
the $p$-adic valuation of $d$.\footnote{We recall that
$\mathrm{Fil}_m(\widetilde{\CW}(k))\cong\W_m(k)$ (cf. section
\ref{CAN}).}

\item The epimorphism
$\mathrm{Proj}_d:\mathrm{Gr}_{d}(\CW(\E))\to\mathrm{Gr}_{d}(\mathrm{H}^1(\G_{\E},\mathbb{Q}_p/\mathbb{Z}_p))$
corresponds via the isomorphism \eqref{isomorphism between W_m and
CW^(-d)} and Corollary \ref{ Fil CW} to the following:
\begin{equation}\mathrm{Proj}_d=\left\{
\begin{array}{lcl}
\W_{v_p(d)}(k)\;\longrightarrow\;
\W_{v_p(d)}(k)/p\W_{v_p(d)}(k)&\textrm{if}&d>0\;,\\
&&\\
\CW(k[[t]])\;\xrightarrow[t\mapsto
0]{}\;\frac{\CW(k)}{(\Fb-1)(\CW(k))}&\textrm{if}&d=0\;.
\end{array}\right.
\end{equation}
\item One has
\begin{equation}
\mathrm{Fil}_d(\mathrm{H}^1(\G_{\E},\mathbb{Q}/\mathbb{Z}))=
\mathrm{Fil}_d(\mathrm{H}^1(\G_{\E},\mathbb{Q}_p/\mathbb{Z}_p))\oplus
\Bigl(\oplus_{\ell\neq
p}\;\mathrm{H}^{1}(\G_{\E},\mathbb{Q}_\ell/\mathbb{Z}_\ell)\Bigr)\;,
\end{equation}
and
\begin{equation}
\mathrm{Gr}_{d}(\mathrm{H}^{1}(\G_{\E},\mathbb{Q}/\mathbb{Z}))=\left\{
\begin{array}{lcl}
\mathrm{Gr}_{d}(\mathrm{H}^{1}(\G_{\E},\mathbb{Q}_p/\mathbb{Z}_p))& \mathrm{if}&d>0\;,\\&&\\
\mathrm{Fil}_0(\mathrm{H}^1(\G_{\E},\mathbb{Q}_p/\mathbb{Z}_p))\oplus
\Bigl(\oplus_{\ell\neq
p}\mathrm{H}^{1}(\G_{\E},\mathbb{Q}_\ell/\mathbb{Z}_\ell)\Bigr)&\mathrm{if}&d=0\;.\\
\end{array}
\right.
\end{equation}
\end{enumerate}
\end{theorem}
\begin{proof} The theorem follows immediately from the above computations. In particular from Corollary \ref{ Fil CW},
 Expression \eqref{Natural filtration of CW and CW tilde}, Section \ref{action
on CW^(-d)}, Equation \eqref{Fil_d(H^1)}, and Definition
\ref{definition of Fil_d(H^1(G_E,Q/Z))}. We observe that
$m_{n,d}=m_{n,d-1}+1$ if and only if $d=np^{m_{n,d}}$, $n \in \J$.
\end{proof}

\begin{corollary}\label{swan conductor of a comonomial}
Let $\lb\in\W_m(k)$, and let $n\in\J$. Let
$\alpha^{-}:=\delta(\lb\cdot t^{-np^m})$, where $\delta$ is the
Artin-Schreier-Witt morphism (cf.
\eqref{artin-schreier-diagram-covectors}). If $\lb\in
p^{k}\W_m(k)-p^{k+1}\W_m(k)$, then
$\mathrm{sw}(\alpha^-)=np^{m-k}$.\hfill$\Box$
\end{corollary}

\subsubsection{Minimal Lifting.}\label{minimal lifting} Let
$\alpha\in \mathrm{Hom}^{\mathrm{cont}}(\G_{\E},
\mathbb{Q}_p/\mathbb{Z}_p)$. Let $\alpha:=\alpha^{-}\cdot\alpha_0$
be the decomposition given by Theorem \ref{description of H^1}.
Then one can choose a lifting $\bs{f}(t)$ of $\alpha$ in
$\CW(k[t^{-1}])$ (i.e. $\delta(\bs{f})=\alpha$, cf. Sequence
\eqref{artin-schreier-diagram-covectors}) of the form (\cite[Def.
3.6]{Rk1}):
\begin{equation}\label{jjj}
\bs{f}(t) = \sum_{n\in\J}\lb_{-n}t^{-np^{m(n)}} + \bs{f}_0\;,
\end{equation}
with $\bs{f}_0\in\CW(k)$, and with $\lb_{-n} \in
\W_{m(n)}(k)-p\W_{m(n)}(k)$, for every $n\in\J$ such that
$\lb_{-n}\neq 0$. Such a lifting of $\alpha$ will be called a
\emph{minimal lifting}, and a Witt co-vector of the form
\eqref{jjj} will be called \emph{pure}. In this case one has
\begin{equation}
\mathrm{sw}(\alpha)=\max(\;0\;,\;\max\{\;np^{m(n)}>0\;|\;\lb_{-n}\neq
0\;\}\;)\;.
\end{equation}

\subsubsection{Representation of  $\W_m(k)/p\W_m(k)$.}\label{representation of W_m/p..}
Let $\{\bar{u}_\gamma\}_{\gamma\in \Gamma}$ be a (not necessarily
finite) $p$-basis of $k$ (over $k^p$). We can write every element
$\lambda\in k$ as $\lambda = \sum_{\underline{s}\in I_{\Gamma}}
\lambda_{\underline{s}}\bar{u}^{\underline{s}}$, where
$I_{\Gamma}:=\{\underline{s}=(s_{\gamma})_{\gamma}\in
[0,p-1]^{\Gamma}\textrm{ such that }s_\gamma\neq 0\textrm{ for
finitely many values of }\gamma\}$. Let
$I_{\Gamma}':=I_{\Gamma}-\{(0,\ldots,0)\}$. We denote by $k'$ the
sub-$k^p$-vector space of $k$ with basis
$\{\bar{u}^{\underline{s}}\}_{\underline{s}\in I_{\Gamma}'}$, that
is the set of elements of $k$ satisfying
$\lambda_{\underline{0}}=0$. Then an element of $\W_m(k)/p\W_m(k)$
admits a unique lifting in $\W_m(k)$ of the form
$(\lambda_0,\lambda_1,\ldots,\lambda_m)$, satisfying $\lambda_0\in
k$, $\lambda_1,\ldots,\lambda_m\in k'$. Hence an element of
$\W_m(k)/p\W_m(k)$ can be uniquely represented by a Witt vector in
$\W_m(k)$ of this form.

\subsection{Decomposition of \protect{$\G_{\E}^{\mathrm{ab}}$}}
\label{decomposition of G_E^ab}

The object of this section is  the proof of Proposition \ref{Decomposition of
G_E^ab prop} below. This fact was proved  in  
\cite{Mel-Sha}. The result is stated also in \cite[Remark
4.18]{Rk1} without any proof: here we will give a new proof using the framework we have just introduced.  The reader must be cautious in reading \cite{Rk1}
to the fact that there is sometime  a confusion between
$\mathcal{I}^{\mathrm{ab}}_{\G_{E}}$ and
$\mathcal{I}_{\G^{\mathrm{ab}}_{\E}}$. Every statement of
\cite{Rk1} is correct, if one consider
$\mathcal{I}_{\G^{\mathrm{ab}}_{\E}}$ and
$\mathcal{P}_{\G^{\mathrm{ab}}_{\E}}$, instead of
$\mathcal{I}_{\G_{\E}}^{\mathrm{ab}}$ and
$\mathcal{P}^{\mathrm{ab}}_{\G_{\E}}$(cf. Remark \ref{inertia ab
different from ab inertia}). The main tool to prove Proposition
\ref{Decomposition of G_E^ab prop} is Theorem \ref{description of
H^1} and its proof.

\begin{proposition}\label{Decomposition of G_E^ab prop}
Let as usual $\E\cong k(\!(t)\!)$. One has:
\begin{equation}
\G_{\E}^{\mathrm{ab}}=\mathcal{I}_{\G_{\E}^{\mathrm{ab}}}\oplus
\G_{k}^{\mathrm{ab}}\;.
\end{equation}
\end{proposition}
\begin{proof}

We can replace $k$ with $k^{\mathrm{perf}}$. Indeed we have a
canonical isomorphism
\begin{equation}\label{G_E simto G_E^perf}
\Gal(k^{\mathrm{perf}}(\!(t)\!)^{\mathrm{sep}}/k^{\mathrm{perf}}(\!(t)\!)) \;
\xrightarrow[]{\;\;\sim\;\;} \;
\Gal(k(\!(t)\!)^{\mathrm{sep}}/k(\!(t)\!)),
\end{equation}
identifying $\mathcal{I}_{\G_{k(\!(t)\!)}}$ and
$\mathcal{P}_{\G_{k(\!(t)\!)}}$ with
$\mathcal{I}_{\G_{k^{\mathrm{perf}}(\!(t)\!)}}$ and
$\mathcal{P}_{\G_{k^{\mathrm{perf}}(\!(t)\!)}}$ respectively. This
is proved by including both $k(\!(t)\!)^{\mathrm{sep}}$ and
$k^{\mathrm{perf}}(\!(t)\!)^{\mathrm{sep}}$ in the $t$-adic
completion $\widehat{k(\!(t)\!)^{\mathrm{alg}}}$ of
$k(\!(t)\!)^{\mathrm{alg}}$. On the right hand side of \eqref{G_E
simto G_E^perf} there is the logarithmic Abbes-Sato's filtration,
while on the left hand side we have the classical filtration, as
presented in \cite{Se}. \emph{These two filtrations are not
preserved by the above isomorphism,} as proved by Theorem
\ref{description of H^1}. We need now some lemmas (and we make a remark).

\begin{remark}\label{inertia ab different from ab inertia}
We recall that if $K_1\subset K_2\subset K_3$ are Galois
extensions of ultrametric complete valued fields with
\emph{perfect} residue fields $k_1\subset k_2\subset k_3$, then
the map $\mathcal{I}_{\Gal(K_3/K_1)}\to
\mathcal{I}_{\Gal(K_2/K_1)}$ is always surjective. We recall also
that if $K_1^{\mathrm{ab}}$ (resp $k_1^{\mathrm{ab}}$) is the
maximal abelian extension of $K_1$ in $K_2$ (resp. of $k_1$ in
$k_2$), then the residual field of $K_1^{\mathrm{ab}}$ is
$k_1^{\mathrm{ab}}$. Hence one has a surjective map
$\mathcal{I}_{\G_{\E}}^{\mathrm{ab}} \longrightarrow
\mathcal{I}_{\G_{\E}^{\mathrm{ab}}}$, which is usually not an
isomorphism.
\end{remark}

\begin{lemma}\label{passage to abelian in the semidirect product}
Let $\mathcal{I},H$ be two subgroups of a given group $\G$. Assume
that $\mathcal{I}$ is normal in $\G$, and that $\G$ is a
semidirect product of $\mathcal{I}$ and $H$. Then
$\G^{\mathrm{ab}}$ is a direct product of a quotient
$\mathcal{I}^{\mathrm{ab}}/N$ of $\mathcal{I}^{\mathrm{ab}}$ with
$H^{\mathrm{ab}}$.
\end{lemma}
\begin{proof}
Straightforward.
\end{proof}

\begin{lemma}\label{semi-direct product over E}
Assume that $k$ is perfect. Let $\E^{\mathrm{tame}}$ be the
maximal tamely ramified extension of $\E$ in $\E^{\mathrm{sep}}$.
Then $\Gal(\E^{\mathrm{tame}}/\E)$ is a semi-direct product of
$\mathcal{I}_{\Gal(\E^{\mathrm{tame}}/\E)}$, with
$\G_k=\Gal(k^{\mathrm{sep}}/k)$.
\end{lemma}
\begin{proof}
Let $\E^{\mathrm{unr}}\subset\E^{\mathrm{tame}}$ be the maximal
unramified extension of $\E$ in $\E^{\mathrm{sep}}$. It is known
that $\E^{\mathrm{tame}} = \cup_{(N,p)=1}
\E^{\mathrm{unr}}(t^{1/N})$. In other words, every element $x\in
\E^{\mathrm{tame}}$ can be written as $x=\sum_{i=0}^{N-1} a_it^{i/N}$,
for some $N$-th root $t^{1/N}$ of $t$, with $a_1,\ldots,a_{N-1}\in
k'$, where $k'/k$ is some finite Galois extension of $k$. Then via the isomorphism 
$\Gal(\E^{\mathrm{unr}}/\E) \simto\G_{k}$,  $\G_{k}$ acts on
$\E^{\mathrm{tame}}$ by $\sigma(x):=\sum_{i=0}^n
\sigma(a_i)t^{i/N}$. Hence the sequence
\begin{equation}
1\to\mathcal{I}_{\Gal(\E^{\mathrm{tame}}/\E)}\to
\Gal(\E^{\mathrm{tame}}/\E) \to\G_k\to 1
\end{equation}
is splitting.
\end{proof}

\begin{corollary}\label{decomposition of G_E}
Assume that $k$ is perfect. Let
$\E^{\mathrm{tame,ab}}:=\E^{\mathrm{tame}}\cap\E^{\mathrm{ab}}$.
Then
\begin{equation}
\Gal({\E}^{\mathrm{tame,ab}}/\E)=
\mathcal{I}_{\Gal({\E}^{\mathrm{tame,ab}}/\E)} \times
\G_{k}^{\mathrm{ab}}\;.
\end{equation}
\end{corollary}
\begin{proof}
Apply Lemmas \ref{passage to abelian in the semidirect product}
and \ref{semi-direct product over E}.
\end{proof}

\emph{Continuation of the proof of Proposition \ref{Decomposition
of G_E^ab prop}:} Let $\mathrm{P}$ (resp. $\mathrm{P}_k$) be the
$p$-primary subgroup of $\G_{\E}^{\mathrm{ab}}$ (resp.
$\G_k^{\mathrm{ab}}$). By the classical properties of $p$-primary
subgroups, $\mathrm{P}$ and $\mathrm{P}_k$ are direct factors of
$\G_{\E}^{\mathrm{ab}}$ and $\mathrm{G}_k^{\mathrm{ab}}$. One has
the exact sequence:
\begin{equation}\label{splitting p-primary}
0\to \mathcal{P}_{\G_{\E}^{\mathrm{ab}}} \to \mathrm{P} \to
\mathrm{P}_k \to 0\;.
\end{equation}
By the Artin-Schreier-Witt theory, one has canonical
identifications of
$\Hom^{\mathrm{cont}}(\mathrm{P},\mathbb{Q}_p/\mathbb{Z}_p)$ with $\CW(\E)/(\Fb-1)(\CW(\E))$,
and
$\Hom^{\mathrm{cont}}(\mathrm{P}_k,\mathbb{Q}_p/\mathbb{Z}_p)=\CW(k)/(\Fb-1)(\CW(k))$.
On the other hand, one has the exact sequence
\begin{equation}\label{splitting Artin-Schreier ...}
0\leftarrow
\frac{\CW(t^{-1}k[t^{-1}])}{(\Fb-1)(\CW(t^{-1}k[t^{-1}]))}
\leftarrow\frac{\CW(\E)}{(\Fb-1)(\CW(\E))}\leftarrow
\frac{\CW(k)}{(\Fb-1)(\CW(k))}\leftarrow 0\;.
\end{equation}
Hence, by Pontriagyn duality,
$\Hom^{\mathrm{cont}}(\mathcal{P}_{\G_{\E}^{\mathrm{ab}}},\mathbb{Q}_p/\mathbb{Z}_p)$
is canonically identified to
$\frac{\CW(t^{-1}k[t^{-1}])}{(\Fb-1)(\CW(t^{-1}k[t^{-1}]))}$.
Since the sequence \eqref{splitting Artin-Schreier ...} splits,
then the sequence \eqref{splitting p-primary} splits too.

Since $\mathrm{P}$ is the $p$-primary part of
$\G_\E^{\mathrm{ab}}$, then $\G_{\E}^{\mathrm{ab}}\cong
\mathrm{P}\oplus (\G_{\E}^{\mathrm{ab}}/\mathrm{P})$. By Corollary
\ref{decomposition of G_E} one finds
$\G_{\E}^{\mathrm{ab}}/\mathrm{P}=\mathrm{Gal}(\E^{\mathrm{tame,ab}}/\E)/
(\mathrm{P}/\mathcal{P}_{\mathrm{G}_{\E}^{\mathrm{ab}}}) \! \cong \! 
(\mathcal{I}_{\Gal(\E^{\mathrm{tame,ab}}/\E)}\times\G_k^{\mathrm{ab}})/
(\mathrm{P}/\mathcal{P}_{\mathrm{G}_{\E}^{\mathrm{ab}}}) \! \cong \!
\mathcal{I}_{\Gal(\E^{\mathrm{tame,ab}}/\E)}\times(\G_k^{\mathrm{ab}}/\mathrm{P}_k)$.
Indeed
$\mathcal{I}_{\Gal(\E^{\mathrm{tame,ab}}/\E)}=(\mathcal{I}_{\G_{\E}^{\mathrm{ab}}}/\mathcal{P}_{\G_{\E}^{\mathrm{ab}}})$.
This shows that $\G_{\E}^{\mathrm{ab}}\cong \mathrm{P}_k \oplus
\mathcal{P}_{\mathrm{G}^{\mathrm{ab}}_{\E}} \oplus
(\mathcal{I}_{\G_{\E}^{\mathrm{ab}}}/\mathcal{P}_{\G_{\E}^{\mathrm{ab}}})\oplus
(\G_k^{\mathrm{ab}} / \mathrm{P}_k)$.\!
\end{proof}

\subsection{Definition of arithmetic Swan conductor for rank one
representations with finite local monodromy} \label{Swan cond
aritm for rk1 repre with loc mon}

Since
\begin{equation}\label{decomposition of G_E = P + I/P + G_k}
\G_{\E}^{\mathrm{ab}} \;\; \cong \;\;
\mathcal{P}_{\G_{\E}^{\mathrm{ab}}}\oplus
(\mathcal{I}_{\G_{\E}^{\mathrm{ab}}}/\mathcal{P}_{\G_{\E}^{\mathrm{ab}}})
\oplus \G_k^{\mathrm{ab}}\;,
\end{equation}
then every rank one representation with finite local monodromy
$\alpha:\G_{\E}\to \O_K^{\times}$,
$\V(\alpha)\in\mathrm{Rep}^{\mathrm{fin}}_{\O_K}(\G_{\E})$ is
a product of three characters:
\begin{equation}\label{alpha_k,tame,wild}
\alpha=\alpha_{\mathrm{wild}}\cdot \alpha_{\mathrm{tame}}\cdot
\alpha_{k}\;,
\end{equation}
where $\alpha_{k}$ (resp. $\alpha_{\mathrm{tame}}$,
$\alpha_{\mathrm{wild}}$) is equal to $1$ on
$\mathcal{I}_{\G_{\E}^{\mathrm{ab}}}$ (resp.
$\mathcal{P}_{\G_{\E}^{\mathrm{ab}}} \oplus \G_k^{\mathrm{ab}}$,
$(\mathcal{I}_{\G_{\E}^{\mathrm{ab}}}/\mathcal{P}_{\G_{\E}^{\mathrm{ab}}})
\oplus \G_k^{\mathrm{ab}}$). In term of representations, this is
equivalent to the expression of  $\V(\alpha)$ as a tensor product:
\begin{equation}
\V(\alpha)=\V(\alpha_{\mathrm{wild}})\otimes\V(\alpha_{\mathrm{tame}})\otimes\V(\alpha_k)\;.
\end{equation}

\begin{definition}[(Swan conductor of a rank one representation
with finite local monodromy)]\label{Definition of Swan for finite
char} Let
$\V(\alpha)\in\mathrm{Rep}_{\O_K}^{\mathrm{fin}}(\G_{\E})$ be a
rank one representation with finite local monodromy. Let $n$ be
the greatest number such that
$\bs{\mu}_{n}(\O_K)=\bs{\mu}_{n}(K^{\mathrm{alg}})$, and let
\begin{equation}
\psi:\mathbb{Z}/n\mathbb{Z}\xrightarrow[]{\;\;\sim\;\;}
\bs{\mu}_{n}(\O_K)
\end{equation}
be a fixed identification. We define the \emph{Swan conductor},  $\mathrm{sw}(\V(\alpha))$,   of $\alpha$ 
 as
\begin{equation}
\mathrm{sw}(\V(\alpha))=\mathrm{sw}(\psi^{-1}\circ(\alpha_{\mathrm{wild}}\cdot\alpha_{\mathrm{tame}}))\;,
\end{equation}
where, in the right hand side, we mean the Kato's definition of
the character
$\psi^{-1}\circ(\alpha_{\mathrm{wild}}\cdot\alpha_{\mathrm{tame}}):\G_{\E}\longrightarrow
\mathbb{Z}/n\mathbb{Z}\subset\mathbb{Q}/\mathbb{Z}$  (cf. Def.
\ref{Kato definition of Swan conductor}).
\end{definition}

The above definition does not depend on the choice of $\psi$.
Indeed if $\psi':\mathbb{Z}/n\mathbb{Z}\simto \bs{\mu}_{n}(\O_K)$
is another choice, there exists $N\in\mathbb{Z}$, with $(N,n)=1$,
such that $\psi'=(N\cdot)\circ\psi$, where
$(N\cdot):\mathbb{Z}/n\mathbb{Z}\simto\mathbb{Z}/n\mathbb{Z}$ is
the multiplication by $N$. Let
$\beta:=\psi^{-1}\circ(\alpha_{\mathrm{wild}}\cdot\alpha_{\mathrm{tame}})$,
and $\beta':=(N\cdot)\circ\beta
=\psi'^{-1}\circ(\alpha_{\mathrm{wild}}\cdot\alpha_{\mathrm{tame}})$.
If $n$ is prime to $p$, then
$\beta(\mathcal{P}_{\G_{\E}^{\mathrm{ab}}})=
\beta'(\mathcal{P}_{\G_{\E}^{\mathrm{ab}}})=\{0\}$ in
$\mathbb{Z}/n\mathbb{Z}$. Hence the Swan conductor is equal to $0$
in both cases (cf. Section \ref{vanishing of tame and residual
katos conductors} below). If $n$ is not prime to $p$, let
$\mathbb{Z}/p^{m+1}\mathbb{Z}$ be the $p$-primary part of
$\mathbb{Z}/n\mathbb{Z}$.  In this case  $N$ is prime to
$p$: then  the multiplication by $N$ preserves Kato's filtration, it implies that    the Swan conductors of $\beta$ and
$\beta'=(N\cdot)\circ\beta$ are equal.

\subsubsection{Vanishing of residual and tame arithmetic Swan
conductors.}\label{vanishing of tame and residual katos
conductors} Definition \ref{Definition of Swan for finite char}
agrees with Definition \ref{Kato definition of Swan conductor},
and more precisely the Swan conductors of $\alpha_k$ and
$\alpha_{\mathrm{tame}}$ are always equal to $0$. Indeed Def.
\ref{Definition of Swan for finite char} and Def. \ref{Kato
definition of Swan conductor} coincide for
$\alpha_{\mathrm{tame}}$, and, by Theorem \ref{description of
H^1}, $\mathrm{sw}(\alpha_{\mathrm{tame}})=0$. Moreover, if the
image of $\alpha_k$ is finite in $\O_K^{\times}$, then, by Theorem
\ref{description of H^1}, one has $\mathrm{sw}(\alpha_k)=0$, and
this agrees with Definition \ref{Definition of Swan for finite
char}. For all characters $\alpha:\G_{\E}\to\O_K^{\times}$, with
$\V(\alpha)\in\mathrm{Rep}_{\O_K}^{\mathrm{fin}}(\G_{\E})$, one
has
\begin{equation}
\mathrm{sw}(\alpha)=\mathrm{sw}(\alpha_{\mathrm{wild}})\;.
\end{equation}

\begin{remark}\label{V_1 otimes V_2}
One sees that $\mathrm{sw}(\V_1\otimes\V_2)\leq
\max(\mathrm{sw}(\V_1),\mathrm{sw}(\V_2))$, for all
$\V_1,\V_2\in\mathrm{Rep}^{\mathrm{fin}}_{\O_K}(\G_{\E})$.
Moreover equality holds if
$\mathrm{sw}(\V_1)\neq\mathrm{sw}(\V_2)$
\end{remark}

\section{Kato's refined Swan conductor} \label{Kato's refined
Swan conductor}

 Let
$\E$ be, as usual, a complete discrete valued field of
characteristic $p$, with residue field $k$. We fix a uniformizer
$t\in\O_{\E}$ and an isomorphism $\O_{\E}\cong k[\![t]\!]$. We
identify $k$ with its image in $\O_{\E}$ via that isomorphism. As
usual for all $\mathbb{F}_p$-ring $R$ we set
$\Omega^1_R:=\Omega^1_{R/R^p}$.

In \cite{Kato-Swan} K.Kato was able to introduce a filtration on
$\Omega^1_{\E}$ and then a family of submodules of the $d$-th
graded $\B \mathrm{Gr}_d\Omega^{1}_{\E}\subset
\mathrm{Gr}_d\Omega^{1}_{\E}$, $d\geq 0$. In such a way he was
able to define an isomorphism, for all $d\geq 0$:
\begin{equation}
\psi_d\; :
\;\mathrm{Gr}_d\mathrm{H}^1(\G_{\E},\mathbb{Q}/\mathbb{Z})
\xrightarrow[]{\;\;\sim\;\;} \B\mathrm{Gr}_d\Omega^1_{\E}\;
\subset \;\mathrm{Gr}_d\Omega^1_{\E}\;,
\end{equation}
associating in this way to a character a $1$-differential class.
Whereas the arithmetic Swan conductor $\mathrm{sw}(\alpha)$ of a
character $\alpha\in \mathrm{H}^1(\G_{\E},\mathbb{Q}/\mathbb{Z})$
is the smallest integer $d\geq 0$, such that
$\alpha\in\mathrm{Fil}_d\mathrm{H}^1(\G_{\E},\mathbb{Q}/\mathbb{Z})$
(cf. Def. \ref{Kato definition of Swan conductor}), the so called
``\emph{refined Swan conductor}'' of $\alpha$ is the image of the
class of $\alpha$ by the morphism:
\begin{equation}
\psi_{\mathrm{sw}(\alpha)} :
\mathrm{Gr}_{\mathrm{sw}(\alpha)}\mathrm{H}^1(\G_{\E},\mathbb{Q}/\mathbb{Z})
\xrightarrow[]{\;\;\;\;\;}
\mathrm{Gr}_{\mathrm{sw}(\alpha)}\Omega^1_{\E}\;.
\end{equation}
The refined Swan conductor of $\alpha$ is defined only if
$\mathrm{sw}(\alpha)
> 0$, we denote it by
$\mathrm{rsw}(\alpha)$.

In this section we interpret the refined Swan conductor, and
the isomorphisms $\psi_d$'s, in term of our isomorphism
$\mathrm{Gr}_d\mathrm{H}^1(\G_{\E},\mathbb{Q}/\mathbb{Z}) \cong
\W_{v_p(d)}(k) / p\cdot \W_{v_p(d)}(k)$, if $d>0$, (cf. Theorem
\ref{description of H^1}): hence we explicitly associate a differential  to a Witt
vector on $\W_{v_p(d)}(k)$.

We improve the explicit description recently obtained by A.Abbes
and T.Saito. We first recall the definition of $\psi_d$ and the
work of Abbes-Saito for the convenience of the reader (cf. Section
\ref{sectionhhhf jhl k}). Then we apply our description to that
context (cf. Section \ref{Explicit description of psi_d in terms of Witt
co-vectors}). Notations and settings come from
\cite{Ab-Sa}. In this section, according to \cite{Ab-Sa}, we
assume that $k$ has a finite $p$-basis
$\{\bar{u}_1,\ldots,\bar{u}_r\}$.

\subsection{Definition of Kato's refined Swan conductor and Abbes-Saito's computations.}
\label{sectionhhhf jhl k}
\subsubsection{Kato's Filtration of $\Omega^1_{\E}$.}
We refer to \cite[5.4]{Ab-Sa} for the formal definition of
$\Omega^1_{k[\![t]\!]}(\log)$.  Considering the trivialization
$\E\cong k(\!(t)\!)$, this is nothing but
\begin{equation}
\Omega^1_{k[\![t]\!]}(\log)\; \cong \; \Bigl(\oplus_{i=1}^r
k[\![t]\!]\cdot d\bar{u}_i\Bigr)\oplus k[\![t]\!]\cdot d\log(t)
\;\; \subset \;\; \Bigl(\oplus_{i=1}^r \E \cdot
d\bar{u}_i\Bigr)\oplus \E\cdot d\log(t) \; \cong \; \Omega^1_{\E}
\;,
\end{equation}
where $d\log(t):= dt/t \in\Omega^1_{\E}$. For all $d\geq 0$, one
sets
\begin{equation}
\mathrm{Fil}_d\Omega^1_{\E}:=
t^{-d}\cdot\Omega^1_{k[\![t]\!]}(\log) \quad,\qquad
\Omega^1_k(\log):=\Omega^1_{k[\![t]\!]}(\log)\otimes_{k[\![t]\!]}
k \;.
\end{equation}
For $d>0$, the graded admits then the following trivialization
\begin{equation}\label{isom Gr_d Omega^1_E}
\mathrm{Gr}_d \Omega^1_{\E}  \cong  \Bigl(\oplus_{i=1}^r k\cdot
t^{-d}\cdot d\bar{u}_i \Bigr)\oplus k\cdot t^{-d}\cdot
d\log(t)\;=\; t^{-d}\cdot \Omega^1_k(\log)\;.
\end{equation}
In particular $\Omega^1_{\E}$ is graded:
$\Omega^1_{\E}=\oplus_{d\geq 0}\mathrm{Gr}_d(\Omega^1_{\E})$.

\subsubsection{Kato's isomorphism $\psi_d$.}
We recall that, by \cite[10.7]{Ab-Sa}, for all $s\geq 0$, $d>0$,
there exists a unique group morphism $\psi_{s,d}$ making the following 
  diagram commutative
\begin{equation}\label{diagram BGr_d, psi_d}
\xymatrix{\mathrm{Gr}_d\W_s(\E)\ar[rr]^-{-\mathrm{gr}_d(\F^sd)}\ar[d]_-{\mathrm{gr}_d(\delta)}&&
\mathrm{Gr}_d(\Omega^1_{\E})\\
\mathrm{Gr}_d(\mathrm{H}^1(\mathrm{G}_{\E},\mathbb{Z}/p^{s+1}\mathbb{Z}))\ar[rru]_-{\psi_{s,d}}&&}
\end{equation}
where
$\delta:\W_s(\E)\to\mathrm{H}^1(\mathrm{G}_{\E},\mathbb{Z}/p^{s+1}\mathbb{Z})$
is the Artin-Schreier-Witt morphism (cf.
\eqref{artin-screier-diagram}), and
$\F^sd:\W_s(\E)\to\Omega^{1}_{\E}$ is given by
\begin{equation}\label{definition of F^sd...lj....}
\F^sd(\bar{f}_0,\ldots,\bar{f}_s)=\sum_{i=0}^{s}\bar{f}_i^{p^{s-i}}d\log(f_i)\;.
\end{equation}
By \cite[10.8]{Ab-Sa}, the family of maps $\{\psi_{s,d}\}_{s\geq
0}$ is compatible with the inclusions
$\jmath:\mathrm{H}^1(\mathrm{G}_{\E},\mathbb{Z}/p^{s}\mathbb{Z})\to
\mathrm{H}^1(\mathrm{G}_{\E},\mathbb{Z}/p^{s+1}\mathbb{Z})$ (cf.
\eqref{artin-screier-diagram}). We have hence a map:
\begin{equation}
\psi_d:\mathrm{Gr}_d(\mathrm{H}^1(\mathrm{G}_{\E},\mathbb{Q}_p/\mathbb{Z}_p))\to\mathrm{Gr}_d(\Omega^1_{\E})\;.
\end{equation}

\subsubsection{The groups $\B_m\Omega^1_{k}$ and $\B\mathrm{Gr}_d(\Omega^1_{\E})$.}
For each $q\geq 0$ we denote by
\begin{equation}
\mathrm{Z}^q\Omega^{\bullet}_{k}=\mathrm{Ker}(d:\Omega^q_k\to\Omega^{q+1}_k)
\quad,\quad
\B^q\Omega^{\bullet}_{k}=\mathrm{Im}(d:\Omega^{q-1}_k\to\Omega^{q}_k)\quad,\quad
\mathrm{H}^q(\Omega^{\bullet}_k)=\mathrm{Z}^q\Omega^{\bullet}_{k}/\mathrm{B}^q\Omega^{\bullet}_{k}
\end{equation}
We denote the inverse Cartier isomorphism (cf. \cite[Ch2, Section
6]{Cart}) by
\begin{equation}
\C^{-1}\;\;:\;\;\Omega^q_k\;\;\xrightarrow[]{\;\;\sim\;\;}\;\;\mathrm{H}^q(\Omega^{\bullet}_k)
\end{equation}
where $\C:\mathrm{Z}^q\Omega^{\bullet}_k\to\Omega^q_k$ is the
Cartier operation. For $r\geq 0$ we introduce subgroups (cf.
\cite[2.2.2]{Ill})
\begin{equation}
\B_m\Omega^q_k\;\;\subset\;\;
\mathrm{Z}_m\Omega^q_k\;\;\subset\;\; \Omega^q_k
\end{equation}
where $\B_0\Omega^q_k=0$, $\B_1\Omega^q_k=\B^q\Omega^{\bullet}_k$
(resp. $\mathrm{Z}_0\Omega^q_k=\Omega^q_k$,
$\mathrm{Z}_1\Omega^q_k=\mathrm{Z}^q\Omega^{\bullet}_k$), and
inductively $\B_{m+1}\Omega^q_k$ (resp.
$\mathrm{Z}_{m+1}\Omega^q_k$) is the inverse image in
$\mathrm{Z}^q\Omega^{\bullet}_k$, under the canonical projection
$\mathrm{Z}^q\Omega^{\bullet}_k\to\mathrm{H}^q(\Omega^{\bullet}_k)$,
of
$\C^{-1}(\B_m\Omega^q_k)\subset\mathrm{H}^q(\Omega^{\bullet}_k)$
(resp.
$\C^{-1}(\mathrm{Z}_m\Omega^q_k)\subset\mathrm{H}^q(\Omega^{\bullet}_k)$).
They respect the following inclusions
\begin{equation}
0=\B_0\Omega^q_k\subset\cdots\subset\B_{m}\Omega_k^q\subset\B_{m+1}\Omega^q_k\subset\cdots
\subset\mathrm{Z}_{m+1}\Omega^q_k\subset\mathrm{Z}_{m}\Omega^q_k\subset\cdots\subset\mathrm{Z}_{0}\Omega^q_k=\Omega^q_k\;.
\end{equation}
By definition
$\C(\mathrm{Z}_{m+1}\Omega^{q}_k)=\mathrm{Z}_{m}\Omega^{q}_k$ and
$\C(\mathrm{B}_{m+1}\Omega^{q}_k)=\mathrm{B}_{m}\Omega^{q}_k$, we
denote by $\C^m:\mathrm{Z}_{m}\Omega^{q}_k\to\Omega^{q}_k$ the
$m$-th iteration of the Cartier operation
($\C^0:=\mathrm{Id}_{\Omega^q_k}$).

\begin{definition}[\protect{(\cite[10.11]{Ab-Sa})}]
Let as usual $d=np^m>0$, $(n,p)=1$, $m=v_p(d)$. We denote by
$\B\mathrm{Gr}_d\Omega^1_{\E}\subset \mathrm{Gr}_d\Omega^1_{\E}$
the subgroup formed by elements of the form $t^{-d}(\alpha+\beta
\cdot d\log(t))$ (cf. Formula  \eqref{isom Gr_d Omega^1_E}), with
$\alpha\in \B_{m+1}\Omega^1_k$,
$\beta\in\mathrm{Z}_m\Omega^0_k(=\mathrm{Z}_mk=k^{p^m})$
satisfying $n\C^m(\alpha)+d\circ\C^m(\beta)=0$.
\end{definition}

\begin{remark}\label{BGrd for m=0} Notice that if $\beta=a^{p^m}$, $a\in k$, then
$\C^m(\beta)=\beta^{p^{-m}}=a$. In particular, for $m=0$, and
$(n,p)=1$, one finds $\B\mathrm{Gr}_n\Omega^1_{\E} = \{
t^{-n}(d(x)-n\cdot x\cdot d\log(t))\;|\; x\in k\}\subset
t^{-n}\cdot\Omega^1_k(\log)\simto \mathrm{Gr}_n\Omega^1_{\E}$.
\end{remark}

\subsubsection{Abbes-Saito's explicit description of $\B\mathrm{Gr}_d\Omega_{\E}^1$.}
Let $\{\bar{u}_1,\ldots,\bar{u}_r\}$ be a $p$-basis of $k$ (over
$\mathbb{F}_p$). Let $I_r:=[0,p-1]^r\subset\mathbb{N}^r$, and
$I_r':=I_r-\{(0,\ldots,0)\}$. For
$\underline{s}:=(s_1,\ldots,s_r)\in I_r$, we set
$\bar{u}^{\underline{s}}:=\bar{u}_1^{s_1}\cdots \bar{u}_r^{s_r}$.
Every $\lambda$ of $k$ can be uniquely written as
$\lambda=\sum_{\underline{s}\in I_r}\lambda_{\underline{s}}\cdot
\bar{u}^{\underline{s}}$, with $\lambda_{\underline{s}}\in k^p$.
Hence $d(\lambda)=\sum_{\underline{s}\in
I_r'}\lambda_{\underline{s}}\cdot \bar{u}^{\underline{s}} \cdot
d\log(\bar{u}^{\underline{s}})\in\Omega^1_k $, where, as usual,
$d\log(\bar{u}^{\underline{s}})=\sum_{i=1}^rs_i\cdot
d\bar{u}_i/\bar{u}_i$. For $j\geq 0$, we set
\begin{equation}
d_{\F^j}(\lambda) :=\sum_{\underline{s}\in
I_r'}\lambda_{\underline{s}}^{p^j}\cdot
(\bar{u}^{\underline{s}})^{p^j} \cdot
d\log(\bar{u}^{\underline{s}})\;.
\end{equation}
Since $\Omega^1_k$ is freely generated by
$d\bar{u}_1/\bar{u}_1,\ldots,d\bar{u}_r/\bar{u}_r$, then, by assuming
$\lambda_{\underline{0}}=0$, $d_{\F^j}(\lambda)$ determines
$\lambda$. We denote by $k'\subset k$, the  sub-$k^p$-vector space
of $k$ with basis $\{\bar{u}^{\underline{s}}\}_{\underline{s}\in
I'_r}$, i.e. whose elements $\lambda=\sum_{\underline{s}\in
I_r}\lambda_{\underline{s}}\bar{u}^{\underline{s}}$,
$\lambda_{\underline{s}}\in k^p$, satisfy
$\lambda_{\underline{0}}=0$.

\begin{proposition}{\protect{(\cite[10.12]{Ab-Sa})}}
Let $d=np^m>0$, $(n,p)=1$, $m=v_p(d)$. Every element $y_d$ of
$\B\mathrm{Gr}_d(\Omega^1_{\E})$ can be uniquely written as
\begin{equation}\label{sum htydfghjkiii}
y_d\;=\;t^{-d}\cdot\Bigl(\lambda_0^{p^{m-1}}d(\lambda_0) - n
\lambda_0^{p^m}d\log(t)+\sum_{1\leq j\leq
m}d_{\F^{m-j}}(\lambda_{j}) \Bigr)\;\in\;
t^{-d}\cdot\Omega^1_k(\log)\simto\B\mathrm{Gr}_d(\Omega^1_{\E})\;,
\end{equation}
with unique choice of $\lambda_1,\ldots,\lambda_{m}\in k'$, and
$\lambda_0\in k$. We will write
$y_d=y_d(\lambda_0,\lambda_1,\ldots,\lambda_{m})$. It is
understood that, if $m=0$, this sum reduces to $y_d =
t^{-d}\cdot(d(\lambda_0) - n \lambda_0 d\log(t))$ (cf. Remark
\ref{BGrd for m=0}).
 \hfill\CVD
\end{proposition}
\begin{remark}
The reader should notice the analogy with Section
\ref{representation of W_m/p..}.
\end{remark}

In \cite{Ab-Sa} one proves the following description of $\psi_d$
(or better its inverse $\rho_d$) by means of the group
$\B\mathrm{Gr}_d(\Omega^1_{\E})$ (cf. \cite[10.13]{Ab-Sa}).

\begin{proposition}[\protect{(\cite[10.14]{Ab-Sa})}]\label{Abbes-Saito psi iso}
For all $d>0$, $d=np^{v_p(d)}$, the image of $\psi_{d}$ is
$\B\mathrm{Gr}_d\Omega^1_{\E}$. Let
\begin{equation}
\rho_d\;:\;\B\mathrm{Gr}_d(\Omega^1_{\E})\xrightarrow[]{\;\;\sim\;\;}
\mathrm{Gr}_d(\mathrm{H}^1(\mathrm{G}_{\E},\mathbb{Q}_p/\mathbb{Z}_p))
\end{equation}
be the inverse of $\psi_d$. Let $\lambda_0\in k$,
$\lambda_{1},\ldots,\lambda_{v_p(d)-1}\in k'$ and let
$y_d(\lambda_0,\lambda_1,\ldots,\lambda_{v_p(d)})$ be the element
\eqref{sum htydfghjkiii} of $\B\mathrm{Gr}_d(\Omega^1_{\E})$. For
$j=1,\ldots,v_p(d)$ write $\lambda_j=\sum_{\underline{s}\in
I_r'}\lambda_{j,\underline{s}}\bar{u}^{\underline{s}}$. Then
$y_d(\lambda_0,\lambda_1,\ldots,\lambda_{v_p(d)})$ is sent, by
$\rho_d$ into
\begin{equation}
\rho_d(y_d(\lambda_0,\lambda_1,\ldots,\lambda_{v_p(d)})) \;
=\mathrm{pr}_d\Bigl( \theta_{v_p(d)}(\lambda_0t^{-n})+\sum_{1\leq
j\leq v_p(d)}\sum_{\underline{s}\in
I'_r}\theta_{v_p(d)-j}(\lambda_{j,\underline{s}}t^{-np^j})\Bigr)\;,
\end{equation}
where
$\theta_j:\E\to\mathrm{H}^1(\G_{\E},\mathbb{Q}_p/\mathbb{Z}_p)$ is
the composite of the $j$-th Teichm\"uller map
$f\mapsto(f,0,\ldots,0)\in\W_j(\E)$ with the Artin-Schreier-Witt
morphism
$\delta:\W_{j}(\E)\to\mathrm{H}^1(\G_{\E},\mathbb{Z}/p^{j+1}\mathbb{Z})\subset
\mathrm{H}^1(\G_{\E},\mathbb{Q}_p/\mathbb{Z}_p)$ (cf.
\eqref{artin-screier-diagram}), and $\mathrm{pr}_d$ is the
canonical projection of $\mathrm{H}^1$ into
$\mathrm{Gr}_d(\mathrm{H}^1)$.\hfill\CVD
\end{proposition}

\subsection{Explicit description of $\psi_d$ in terms of Witt
co-vectors.}\label{Explicit description of psi_d in terms of Witt
co-vectors}

The following Theorem improves the description given by
proposition \ref{Abbes-Saito psi iso}.
\begin{theorem}
Let $d>0$, $d=np^{m}$, $m=v_p(d)$. Let
\begin{equation}
\overline{\delta}_d:\W_{v_p(d)}(k)/p\W_{v_p(d)}(k)
\xrightarrow[]{\;\;\sim\;\;}
\mathrm{Gr}_d(\mathrm{H}^1(\G_{\E},\mathbb{Q}/\mathbb{Z}))
\end{equation}
be the isomorphism of Theorem \ref{description of H^1}. Then the 
composite map
\begin{equation}
\psi_d\circ\overline{\delta}_d:\W_{v_p(d)}(k)/p\W_{v_p(d)}(k)
\xrightarrow[]{\;\;\sim\;\;}\B\mathrm{Gr}_d(\Omega^1_{\E})
\end{equation}
is given by
\begin{equation}
\psi_d\circ\bar{\delta}_d(\lambda_0,\ldots,\lambda_{v_p(d)})=
t^{-d}\cdot \Bigl(-n \lambda_0^{p^{v_p(d)}}d\log(t)+\sum_{0\leq
j\leq v_p(d)} \lambda_j^{p^{v_p(d)-j}}\cdot d\log(\lambda_j)
\Bigr)\;,
\end{equation}
where we represent every  element of
$\W_{v_p(d)}(k)/p\W_{v_p(d)}(k)$ by a unique Witt vector
$(\lambda_0,\ldots,\lambda_{v_p(d)})\in\W_{v_p(d)}(k)$ satisfying
$\lambda_0\in k$, $\lambda_1,\ldots,\lambda_{v_p(d)}\in k'$, as in
Section \ref{representation of W_m/p..}. Moreover every element of
$\B\mathrm{Gr}_d(\Omega^1_{\E})$ can be uniquely written as
\begin{equation}
t^{-d}\cdot\Bigl(-n\lambda_0^{p^{v_p(d)}}d\log(t) +\sum_{0\leq
j\leq v_p(d)}\lambda_j^{p^{v_p(d)-j}}d\log(\lambda_j)\Bigr)\;,
\end{equation}
with unique $\lambda_0\in k$ and
$\lambda_1,\ldots,\lambda_{v_p(d)}\in k'$ (cf. Section
\ref{representation of W_m/p..}).
\end{theorem}
\begin{proof}
Passing to the limit (with respect to $s$) of  the Diagram
\eqref{diagram BGr_d, psi_d}, since  $\F^{s+1}d\circ\V=\F^sd$, we
define
\begin{equation}
\F^{\infty}d=\lim_{s\to\infty}\F^sd\;:\;\CW(\E) \to \Omega^1_K\;.
\end{equation}
We obtain
$\psi_s=\mathrm{gr}_d(\delta)\circ\mathrm{gr}(\F^{\infty}d)$,
where $\mathrm{gr}_d(\delta):\mathrm{Gr}_d(\CW(\E))\to
\mathrm{Gr}_d(\mathrm{H}^1(\G_{\E},\mathbb{Q}_p/\mathbb{Z}_p))$.
Since $\CW(\E)$ is graded with the $d$-th graded $\CW^{(-d)}(k)$
(cf. Corollary \ref{ Fil CW}), the map $\mathrm{gr}(\F^{\infty}d)$
(resp. $\mathrm{gr}_d(\delta)$) is nothing but the restriction of
$\F^\infty d$ (resp. $\delta$) to the subgroup $\CW^{(-d)}(k)$.

The elements of $\CW^{(-d)}(k)$ have the form $\lb_dt^{-d}=
(\cdots,0,0,\lambda_0t^{-n},\ldots,\lambda_{m}t^{-np^m})$, with
$\lb_d:=(\lambda_0,\ldots,\lambda_m)\in\W_m(k)$. One  easily sees that
(cf. \eqref{definition of F^sd...lj....})
\begin{equation}
\F^{\infty}(\lb_d t^{-d} )= \sum_{0\leq j\leq m}(\lambda_j
t^{-np^j})^{p^{m-j}}
d\log(\lambda_jt^{-np^j})=t^{-d}\cdot\Bigl(\sum_{0\leq j\leq
m}\lambda_j^{p^{m-j}} d\log(\lambda_jt^{-np^j})\Bigr)\;.
\end{equation}
This proves the first assertion. To conclude we observe that the
diagram
\begin{equation}
\xymatrix{
\W_m(k)/p\W_m(k)\ar[d]^-{\bar{\delta}}_-{\wr}\ar[r]_-{\sim}^-{\F^md}&
\B\mathrm{Gr}_d(\Omega^1_{\E})\\
\mathrm{Gr}_d(\mathrm{H}^1(\G_{\E},\mathbb{Q}_p/\mathbb{Z}_p))\ar[ur]^{\sim}_{\psi_d}&
}
\end{equation}
implies that the map
$\psi_d\circ\bar{\delta}:\W_m(k)/p\W_m(k)\to\mathrm{Gr}_d(\Omega^1_{\E})$
is injective with image $\B\mathrm{Gr}_d(\Omega^1_{\E})$. By
section \ref{representation of W_m/p..} every element of
$\W_m(k)/p\W_m(k)$ can be uniquely represented by a Witt vector
$(\lambda_0,\ldots,\lambda_m)\in\W_m(k)$ satisfying $\lambda_0\in
k$, $\lambda_1,\ldots,\lambda_m\in k'$. This concludes the proof.
\end{proof}

\section{Explicit computation of Fontaine's functor in rank one
case} \label{computation of the functor-yt-}

In this section we compute the $(\phi,\nabla)$-module
$\mathrm{D}^{\dag}(\V(\alpha))$ (cf. Section \ref{definition of
the functor--}) for rank one representations with finite local
monodromy. Following  Section \ref{decomposition of G_E^ab} we
decompose a general character $\alpha:\G_{\E}\to\O_K^{\times}$ into
$\alpha=\alpha_{\mathrm{wild}}\cdot\alpha_{\mathrm{tame}}\cdot\alpha_{k}$.
We study then separately the \emph{residual case}, the \emph{tame
case}, and the \emph{wild case}.

\if{ Moreover since the images of $\alpha_{\mathrm{wild}}$ and
$\alpha_{\mathrm{tame}}$ are finite groups, we decompose such
finite abelian groups as product of cyclic sub-groups. This
implies that $\alpha_{\mathrm{tame}}$ and $\alpha_{\mathrm{wild}}$
are product of cyclic characters of order a power of a prime
number (i.e. characters $\alpha$ whose image $\alpha(\G_{\E})$ is
a cyclic group of order a power of a prime number).

\emph{In the sequel we will always assume that
$\alpha_{\mathrm{wild}}$ and $\alpha_{\mathrm{tame}}$ are cyclic
character of order a power of a prime number.} }\fi 
\if{Our description is not complete since we are not able to
describe completely the Frobenius and the connection arising from
representations of $\G_k$. The reason is that no specific
hypotheses are made on $k$.

Nevertheless our computation is sufficient to prove the equality
between the ``arithmetic Swan conductor'' of $\V(\alpha)$ (i.e.
the Kato's definition of Swan conductor) and the ``differential
Swan conductor'' of $\mathrm{D}^{\dag}(\V(\alpha))$ (i.e. the
Kedlaya's definition (cf. Def. \ref{Definition of Diff-Swan})).
Indeed the representations of $\G_k$ have arithmetic and
differential Swan conductors both equal to $0$ (cf. section
\ref{Residual and tame conductors =0}).}\fi

\subsection{The residual case}

Let $\alpha_k:\G_{\E}\to \O_K^{\times}$ be a rank one
representation with finite local monodromy such that
$\alpha_k|_{\mathcal{I}_{\mathrm{G}^{\mathrm{ab}}_{\E}}}=1$. One
sees, directly from the definition of the functor (cf. Section
\ref{definition of the functor--}), that the
$(\varphi,\nabla)$-module $\D:=\mathrm{D}^{\dag}(\V(\alpha_k))$ is
trivialized by some unramified extension of $\O_L$. In other
words,   $\D$ comes from a
$(\varphi,\nabla)$-module over $\O_L$ by scalar extension. Let $\e_{\D}$ be basis of
such a lattice, then
\begin{equation}\label{matrix in residual case}
\left\{\begin{array}{rcl} \phi^{\D}(\e_{\D})&=&\theta(\underline{u})\cdot\e_{\D}\;,\\
\nabla_{T}^{\D}(\e_{\D})&=&0\;,\\
\nabla_{u_i}^{\D}(\e_{\D})&=& g_k(\underline{u})\cdot \e_{\D}\;,
\end{array}\right.
\end{equation}
where $\underline{u}:=(u_1,\ldots,u_r)$, and
$\theta(\underline{u}),g_k(\underline{u})\in\O_L$.

\begin{remark}
If the image of $\alpha_k$ is finite, it is actually possible, by
using the theorem \cite[Th.2.8]{Rk1}, to express
$\theta(\underline{u})$ as value at $T=1$ of a certain
overconvergent function, but this is not necessary for our
purposes.
\end{remark}

\subsection{The tamely ramified case} \label{tame case}

Let $\alpha_{\mathrm{tame}}:\G_{\E}\to \O_K^{\times}$ be a 
character, with
$\V(\alpha_{\mathrm{tame}})\in\mathrm{Rep}_{\O_K}^{\mathrm{fin}}(\G_{\E})$, such
that
$\alpha_{\mathrm{tame}}|_{\mathcal{P}_{\G^\mathrm{ab}_{\E}}\oplus\G_{k}^{\mathrm{ab}}}=1$
(in the decomposition \eqref{decomposition of G_E = P + I/P +
G_k}). One has $\alpha_{\mathrm{tame}}(\G_{\E})=
\bs{\mu}_{N}\subset\O_K^{\times}$, for some $(N,p)=1$. In
particular $\bs{\mu}_{N}(\mathbb{F}_p^{\mathrm{alg}})\subset
\mathbb{F}_q\subseteq k$. By Kummer theory, the kernel of
$\alpha_{\mathrm{tame}}$ defines an extension of $k(\!(t)\!)$ of
the type $k(\!(t^{1/N})\!)/\E$ (i.e.
$\Ker(\alpha_{\mathrm{tame}})=\Gal(\E^{\mathrm{sep}}/k(\!(t^{1/N})\!))$).
This is the smallest extension trivializing
$\V(\alpha_{\mathrm{tame}})$, and the action of $\G_{\E}$ on
$t^{1/N}$ is given by
\begin{equation}
\gamma(t^{1/N})=\alpha_{\mathrm{tame}}(\gamma)\cdot
t^{1/N}\;,\quad\textrm{for all }\gamma\in\G_{\E}\;.
\end{equation}
Then the unramified extension of $\O_{\Ed_L}$ corresponding to
$k(\!(t^{1/N})\!)/k(\!(t)\!)$ is given by
\begin{equation}
\O_{\Ed_{L,T^{1/N}}}/\O_{\Ed_{L,T}}\;,
\end{equation}
where the notation indicates the variable of the rings as in
\cite{Tsu-swan}. The inclusion
$\O_{\Ed_{L,T}}\subset\O_{\Ed_{L,T^{1/N}}}$ is given by sending
$T$ into $(T^{1/N})^N$. Let $\e_{\V}\in\V(\alpha)$ be a basis in
which $\G_{\E}$ acts as
\begin{equation}
\gamma(\e_{\V})=\alpha_{\mathrm{tame}}(\gamma)\cdot\e_{\V}\;,\quad\textrm{for
all }\gamma\in\G_{\E}\;.
\end{equation}
This shows that a basis of $\D:=\mathrm{D}^{\dag}(\V(\alpha))=(
\V(\alpha) \otimes_{\O_K} \O_{\widetilde{\Ed_L}} )^{\G_{\E}}$ is
given by
\begin{equation}
\e_{\D}\;:=\;\e_{\V}\otimes T^{-1/N}\;.
\end{equation}
Indeed, for all $\gamma\in\G_{\E}$, one has
$\gamma(\e_{\D})=\gamma(\e_{\V})\otimes
\gamma(T^{-1/N})=\alpha_{\mathrm{tame}}(\gamma)\e_{\V}\otimes
\alpha_{\mathrm{tame}}(\gamma)^{-1}T^{-1/N}=\e_{\D}$. In this
basis the action of $\phi$ and $\nabla$ are given by
\begin{equation}\label{matrix in the tame case}
\left\{\begin{array}{rcl}
\phi^{\D}(\e_{\D})&=&T^{(1-q)/N}\cdot\e_{\D}\;,\\
\nabla_T^{\D}(\e_{\D})&=&-\frac{1}{N}T^{-1}\cdot\e_{\D}\;,\\
\nabla_{u_i}^{\D}(\e_{\D})&=&0\;.
\end{array}\right.
\end{equation}
In particular, the solution of this differential equation is
$T^{-1/N}$, which is simultaneously  the Kummer generator of
the smallest extension of $\Ed_L$ trivializing
$\V(\alpha_{\mathrm{tame}})$.

\subsection{The wild ramified case}

 To compute $\D^{\dag}(\V(\alpha_{\mathrm{wild}}))$ we need to know a Kummer
generator of the smallest unramified extension of $\Ed_L$
trivializing $\D^{\dag}(\V(\alpha_{\mathrm{wild}}))$. In the
\emph{tame} case (cf. section \ref{tame case}), the Kummer
generator was $T^{-1/N}$, and it was at the same time the solution
of the differential equation. In the \emph{wild} case, the good
Kummer generator will be a so called
\emph{$\bs{\pi}$-exponential}, and it will be at the same time the
solution of the differential equation defined by
$\D^{\dag}(\V(\alpha_{\mathrm{wild}}))$ too.

All results and proofs of this section come from \cite{Rk1}. We
outline briefly the contents of \cite[Sections 2.3.1, 3.1,
3.2]{Rk1}, and in Section \ref{computation of the functor in the
wild case} we adapt \cite[Section 4.6]{Rk1} to our context.

\subsubsection{$\bs{\pi}$-exponentials.}
Let $\mathfrak{G}(X,Y)\in\mathbb{Z}_p[[X,Y]]$ be a Lubin-Tate
formal group law isomorphic (but not necessarily equal) to
$\mathbb{G}_m$. Let $\bs{\pi}:=(\pi_m)_{m\geq 0}$ be a
\emph{fixed} generator of the Tate module of $\mathfrak{G}$. In
other words, $\{\pi_m\}_m\subset\mathbb{Q}_p^{\mathrm{alg}}$
verifies $[p]_{\mathfrak{G}}(\pi_{0})=0$, $\pi_0\neq 0$,
$|\pi_0|<1$, and $[p]_{\mathfrak{G}}(\pi_{m+1})=\pi_m$, for all
$m\geq 0$, where $[p]_{\mathfrak{G}}(X)\in\mathbb{Z}_p[[X]]$ is
the multiplication by $p$ in $\mathfrak{G}$.

\begin{definition}[(\protect{\cite[Def.3.1]{Rk1}})]
Let $\bs{f}^{-}(T)=(\cdots,0,0,0,f_0^-(T),\ldots,f_m^-(T))$ be an element of 
$\CW(T^{-1}\O_L[T^{-1}])$. We call $\bs{\pi}-$exponential relative
to $\bs{f}^-(T)$ the following power series in $T^{-1}$:
\begin{equation}\label{definition of et_p^infty(f,1)}
\mathrm{e}_{p^\infty}(\bs{f}^-(T),1):=\exp\Bigl(\pi_{m}\phi_{0}^{-}(T)+
\pi_{m-1}\frac{\phi_1^-(T)}{p}+\cdots+\pi_0
\frac{\phi_{m}^{-}(T)}{p^{m}}\Bigr)\;,
\end{equation}
where $\phi_j^-(T):=\sum_{k=0}^j p^k\cdot f^-_k(T)^{p^{j-k}}\in
T^{-1}\O_L[T^{-1}]$ is the $j$-th phantom component of the Witt
vector $(f_0^-(T),\ldots,f_m^-(T))\in\W_m(T^{-1}\O_L[T^{-1}])$.
\end{definition}

\begin{proposition}
Let $L_m:=L(\pi_m)$, and let $L_\infty:=\cup_{m\geq 0}L(\pi_m)$.
Then:
\begin{enumerate}
\item One has
\begin{equation}
\mathrm{e}_{p^{\infty}}((\cdots,0,0,0,f_0^-(T),\ldots,f_m^-(T)),1)
\;\;\in\;\; 1+\pi_mT^{-1}\O_{L_m}[[T^{-1}]]\;.
\end{equation}

\item The map $\bs{f}^-(T) \mapsto
\mathrm{e}_{p^\infty}(\bs{f}^-(T),1)$ is a group homomorphism
\begin{equation}
\CW(T^{-1}\O_L[T^{-1}])\xrightarrow[]{\qquad} \bigcup_{m\geq
0}1+T^{-1}\O_{L_m}[[T^{-1}]] \;\; \subset \;\;
1+T^{-1}\O_{L_\infty}[[T^{-1}]]\;.
\end{equation}
In particular, for every $\bs{f}^-(T)$, the power series
$\mathrm{e}_{p^\infty}(\bs{f}^-(T),1)$ converges at least for
$|T|>1$.

\item The power series $\mathrm{e}_{p^\infty}(\bs{f}^-(T),1)$ is
over-convergent (i.e. converges for $|T|>1-\varepsilon$, for some
$\varepsilon>0$) if and only if the reduction
$\overline{\bs{f}^-(T)}$ lies in $(\Fb-1)(\CW(t^{-1}k[t^{-1}]))$
(i.e. the Artin-Schreier-Witt character
$\delta(\overline{\bs{f}^-(T)})$ defined by
$\overline{\bs{f}^-(T)}$ via the Equation
\eqref{artin-schreier-diagram-covectors} is equal to $0$). %
\end{enumerate}
\end{proposition}
\begin{proof}
See \cite[Th.3.2 and Section 3.2]{Rk1}.
\end{proof}

\subsubsection{The function $\mathrm{e}_{p^\infty}(\bs{f}^-(T),1)$ as explicit Kummer generator.}
\label{computation of the functor in the wild case}

Let $\alpha_{\mathrm{wild}}:\G_{\E}\to \O_K^{\times}$,
$\V(\alpha_{\mathrm{wild}})\in\mathrm{Rep}^{\mathrm{fin}}_{\O_K}(\G_{\E})$, be a
character such that
$\alpha_{\mathrm{wild}}|_{(\mathcal{I}_{\G^{\mathrm{ab}}_{\E}}/
\mathcal{P}_{\G^{\mathrm{ab}}_{\E}})\oplus\G_k^{\mathrm{ab}}}=1$
in the decomposition \eqref{decomposition of G_E = P + I/P + G_k}.
In particular, observe that
\begin{equation}
\alpha_{\mathrm{wild}}(\mathcal{P}_{\G^{\mathrm{ab}}_{\E}})=\bs{\mu}_{p^{m+1}}\subset
\O_K^{\times}\;,
\end{equation}
for some $m\geq 0$.

\begin{remark}
Observe that,  since  we assume that
$\mathfrak{G}\simto\mathbb{G}_m$, it follows  by Lubin-Tate theory that 
$\bs{\mu}_{p^{m+1}}\subset\O_K^{\times}$ if and only if
$\pi_0,\ldots,\pi_m\in\O_K$.
\end{remark}

\begin{definition}\label{definition of psi_m}
Let $\psi_m:\mathbb{Z}/p^{m+1}\mathbb{Z}\simto\bs{\mu}_{p^{m+1}}$
be the isomorphism sending $\overline{1}$ into the unique
primitive $p^{m+1}$-th root of unity $\xi_{p^m}$ satisfying:
\begin{equation}\label{condition of psi_m}
|\pi_m-(\xi_{p^m}-1)|<|\pi_m|=|p|^{\frac{1}{p^m(p-1)}}\;.
\end{equation}
\end{definition}

Let $\mathrm{F}/\E$ be the cyclic extension of degree
$p^{m+1}$ defined by $(\psi_m^{-1}\circ\alpha_{\mathrm{wild}})$
(i.e.
$\mathrm{Ker}(\psi_m^{-1}\circ\alpha_{\mathrm{wild}})=\Gal(\E^{\mathrm{sep}}/\F)$,
cf. Section \ref{Notation in IAS theory}), and let
$\mathcal{F}^{\dag}/\Ed_L$ be the cyclic unramified extension
whose residue field is $\F/\E$.

\begin{proposition}\label{e(f(T),1) as a Kummer generator}
Let
$\overline{\bs{f}^-}(t)=(\cdots,0,0,0,\overline{f_0^-}(t),\ldots,\overline{f_m^-}(t))
\in \CW(t^{-1}k[t^{-1}])$ be a co-vector of length $m$
\footnote{For example one can take $\overline{\bs{f}^-}(t)$ as a
\emph{minimal lifting} of $\alpha_{\mathrm{wild}}$ (cf. Section
\ref{minimal lifting}).} defining
$\psi_m^{-1}\circ\alpha_{\mathrm{wild}}$ (i.e.
$\delta(\overline{\bs{f}^-}(t))=(\psi_m^{-1}\circ\alpha_{\mathrm{wild}})$
in the sequence \eqref{artin-schreier-diagram-covectors}). Let
$\bs{f}^-(T)=(\cdots,0,0,0,f_0^-(T),\ldots,f_m^-(T))\in\CW(T^{-1}\O_L[T^{-1}])$
be a lifting of $\overline{\bs{f}^-}(t)$ of length $m$. Then:
\begin{enumerate}
\item $\mathrm{e}_{p^\infty}
(\bs{f}^-(T),1)^{p^{m+1}}\in\O_{\Ed_L}^{\times}$; %
\item $\mathrm{e}_{p^\infty} (\bs{f}^-(T),1)$ is a Kummer
generator of $\mathcal{F}^{\dag}/\Ed_L$:
\begin{equation}
\mathcal{F}^{\dag}\;=\;\Ed_L(\;\mathrm{e}_{p^\infty}(\bs{f}^-(T),1)\;)\;;
\end{equation}
\item For all $\gamma\in\G_{\E}\simto
\Gal(\widetilde{\Ed_L}/\Ed_L)$ (cf. Equation \eqref{def. of tilde
Ed_L}), one has
\begin{equation}
\gamma(\mathrm{e}_{p^\infty}
(\bs{f}^-(T),1))=\alpha_{\mathrm{wild}}(\gamma)^{-1}\cdot\mathrm{e}_{p^\infty}
(\bs{f}^-(T),1)\;.
\end{equation}
\end{enumerate}
\end{proposition}
\begin{proof}
See \cite[Section 2.3]{Rk1}. Observe that in \cite[Section
2.3]{Rk1} the author was working in a more general context, and
for this reason he used a function called
``$\theta_{p^s}(\bs{\nu},1)$'', which actually coincides with
$\mathrm{e}_{p^\infty} (\bs{f}^-(T),1)$ modulo a $p^{m+1}$-th root
of unity (cf. \cite[Equation (3.1)]{Rk1}).
\end{proof}

We can now proceed as in the tame case (cf. Section \ref{tame
case}). We preserve the notation of proposition \ref{e(f(T),1) as
a Kummer generator}. Let $\e_{\V}\in\V(\alpha_{\mathrm{wild}})$ be
a basis in which $\G_{\E}$ acts as
\begin{equation}
\gamma(\e_{\V})=
\alpha_{\mathrm{wild}}(\gamma)\cdot\e_{\V}\;,\quad \textrm{for all
}\gamma\in\G_{\E}\;.
\end{equation}
Then a basis of
$\D^{\dag}(\V(\alpha_{\mathrm{wild}}))=(\V(\alpha_{\mathrm{wild}})\otimes_{\O_K}\O_{\widetilde{\Ed_L}})^{\G_{\E}}$
is given by
\begin{equation}\label{Basis wild -gt}
\e_{\mathrm{D}}\; := \;\e_{\V}\otimes
\et_{p^{\infty}}(\bs{f}^-(T),1)\;.
\end{equation}
Indeed, for all $\gamma\in\G_{\E}$, one has
\begin{equation}
\gamma(\e_{\mathrm{D}})= \gamma(\e_{\V})\otimes
\gamma(\et_{p^{\infty}}(\bs{f}^-(T),1))=\alpha_{\mathrm{wild}}(\gamma)\e_{\V}
\otimes\alpha_{\mathrm{wild}}(\gamma)^{-1}\et_{p^m}(\bs{f}^-(T),1)=\e_{\mathrm{D}}\;.
\end{equation}

\subsubsection{} We compute now the action of $\phi^{\D}$ and $\nabla^{\D}$.
\label{e(f(T),1) is a solution of D(V)} In the basis
$\bs{\e}_{\D}$ (cf. \eqref{Basis wild -gt}), the function
$\mathrm{e}_{p^{\infty}}(\bs{f}^-(T),1)$ is the Taylor solution at
$\infty$ of the $\nabla$-module underlying
$\D^{\dag}(\V(\alpha_{\mathrm{wild}}))$. We recall that
$\bs{f}^-(T)=\bs{f}^-(\underline{u},T)=(f_0^-(\underline{u},T),\ldots,f_m^-(\underline{u},T))$
has coefficients which depend also on
$\underline{u}=(u_1,\ldots,u_r)$. In the basis $\e_{\mathrm{D}}$
the action of $\phi$ and $\nabla$ are given by
\begin{equation}
\left\{\begin{array}{rcl}
\phi^{\D}(\e_{\mathrm{D}}) &=& \theta_{p^m}(\bs{f}^-(\underline{u},T),1)\cdot\e_{\mathrm{D}}\;,\\
\nabla_T^{\D}(\e_{\mathrm{D}}) &=& g_{\bs{f}^-}^0(u_1,\ldots,u_r,T)\cdot\e_{\mathrm{D}}\;,\\
\nabla_{u_i}^{\D}(\e_{\mathrm{D}}) &=&
g_{\bs{f}^-}^i(u_1,\ldots,u_r,T)\cdot\e_{\mathrm{D}}\;, %

\end{array}\right.
\end{equation}
where:
\begin{equation}
\label{matrix in the wild case}
\theta_{p^m}(\bs{f}^-(\underline{u},T),1):=
\et_{p^m}(\varphi(\bs{f}^-(\underline{u},T))-\bs{f}^-(\underline{u},T),1)
\end{equation}
\begin{equation}
\label{eq. g_0,f} 
\begin{split}
g_{\bs{f}^-}^0(u_1,\ldots,u_r,T)&:=\sum_{j=0}^m\pi_{m-j}\frac{\frac{d}{dT}(\phi^-_j(T))}{p^j}\\
&=
\sum_{j=0}^m\pi_{m-j}\sum_{k=0}^jf_k^{-}(\underline{u},T)^{p^{j-k}}
\frac{\frac{d}{dT}(f^-_k(\underline{u},T))}{f^-_k(\underline{u},T)},
\end{split}
\end{equation}
\begin{equation}
\label{eq. g_i,f}
\begin{split}
g_{\bs{f}^-}^i(u_1,\ldots,u_r,T)&:=\sum_{j=0}^m\pi_{m-j}\frac{\frac{d}{du_i}(\phi^-_j(T))}{p^j} \\
&=\sum_{j=0}^m\pi_{m-j}\sum_{k=0}^jf_k^{-}(\underline{u},T)^{p^{j-k}}
\frac{\frac{d}{du_i}(f^-_k(\underline{u},T))}{f^-_k(\underline{u},T)},
\end{split}
\end{equation}
where $\varphi$ acts on $\CW(T^{-1}\O_L[T^{-1}])$ coefficient by
coefficient. Observe that we have chosen
$\varphi(T)=T^p\in\O_L[T]$ (cf. Def. \ref{Definition of
Frobenius}). Observe also that Equations \eqref{eq. g_0,f} and
\eqref{eq. g_i,f} are obtained by taking the logarithmic
derivative of the Definition \eqref{definition of et_p^infty(f,1)}
as
$g_{\bs{f}^-}^0=\frac{\frac{d}{dT}(\mathrm{e}_{p^{\infty}}(\bs{f}^-(T),1))}{\mathrm{e}_{p^{\infty}}(\bs{f}^-(T),1)}$,
and
$g_{\bs{f}^-}^i=\frac{\frac{d}{du_i}(\mathrm{e}_{p^{\infty}}(\bs{f}^-(T),1))}{
\mathrm{e}_{p^{\infty}}(\bs{f}^-(T),1)}$. In other words the
knowledge of the solution leads to recover the matrix of the
connection.

Since we have chosen $\psi_m$ satisfying Equation \eqref{condition
of psi_m}, hence this construction does not depend on the choice
of $\bs{\pi}:=(\pi_m)_{m\geq 0}$.

\section{Comparison between arithmetic and differential Swan conductors}
\label{Comparison between arithmetic and differential Swan
conductors}

The object of this section is to proving the following
\begin{theorem}\label{comparison-theorem in rank one}
Let $\V\in\mathrm{Rep}^{\mathrm{fin}}_{\O_K}(\G_{\E})$, be a rank
one representation. Then:
\begin{equation}
\mathrm{sw}(\V)=\mathrm{sw}^{\nabla}(\D^{\dag}(\V))\;,
\end{equation}
where $\mathrm{sw}(\V)$ is the arithmetic Swan conductor (cf. Def.
\ref{Kato definition of Swan conductor} and Def. \ref{Definition
of Swan for finite char}), and
$\mathrm{sw}^{\nabla}(\D^{\dag}(\V))$ is the differential Swan
conductor of $\D^{\dag}(\V)\in
(\phi,\nabla)-\Mod(\O_{\Ed_L}/\O_K)$, considered as an object of
$\nabla-\Mod(\R_L/K)$ (cf. Def. \ref{Definition of Diff-Swan}).
\end{theorem}
\if{The proof is structured as follows. Firstly we reduce ourself
to the case of a wild character (cf. Lemma \ref{diff swan mod
equal to 0} and Remark \ref{remark on the strategy}). We can
assume that the wild character is defined by a co-vector of the
form given in Remark \ref{minimal lifting}. I particular we are
able to reduce the study to a co-monomial of the type $\lb
t^{-d}$, with $\lb\in\W_{v_p(d)}(k)-p\W_{v_p(d)}(k)$. In this case
we will prove that $T(\D^{\dag}(\V),\rho)=\rho^{v_p(d)}$, for all
$\rho\in]0,1[$. So the differential Swan conductor is $v_p(d)$ and
coincides with the arithmetic Swan conductor (cf. Corollary
\ref{swan conductor of a comonomial}). }\fi

\subsection{A Small Radius Lemma}

\label{explicit computation -small radius}

In this section we prove that if  $T(\M,\rho)$   is
``small''    (see  Definition  \ref{definition of T(M,rho)}), then we are able to link $T(\M,\rho)$ to the valuation of the coefficients of the equation. The following ``Small Radius
Lemma'' generalizes the analogous result \cite[6.2 and 
6.4]{Ch-Me}, \cite{Young}.

\begin{lemma}[(Small Radius)]\label{Small radius}
Let $\M_\rho$ be a \emph{rank one} $\nabla$-module over
$\mathcal{F}_{L,\rho}$. For $i=0,\ldots,r$ let $ g_n^{i}\in
\mathcal{F}_{L,\rho}$ be the matrix of $(\nabla_{u_i}^{\M})^n$
(resp. if $i=0$, $g^0_n$ is the matrix of $(\nabla^{\M}_T)^n$).
Write (cf. Lemma \ref{False lemma}):
\begin{eqnarray}
\omega &\stackrel{Lemma \ref{False lemma}}{=}&
\frac{|d/dT|_{\mathcal{F}_{L,\rho},\mathrm{Sp}}}{|d/dT|_{\mathcal{F}_{L,\rho}}}\;=\;
\frac{|d/du_1|_{\mathcal{F}_{L,\rho},\mathrm{Sp}}}{|d/du_1|_{\mathcal{F}_{L,\rho}}}\;=\;\ldots\;=\;
\frac{|d/du_r|_{\mathcal{F}_{L,\rho},\mathrm{Sp}}}{|d/du_r|_{\mathcal{F}_{L,\rho}}}
\;,\nonumber\\
\omega_{\M}(\rho)&\stackrel{\mathrm{def}}{:=}& \min\Bigl(\;
\frac{|d/dT|_{\mathcal{F}_{L,\rho},\mathrm{Sp}}}{|g_{1}^0|_\rho}\;,
\;\frac{|d/du_1|_{\mathcal{F}_{L,\rho},\mathrm{Sp}}}{|g_{1}^1|_\rho}\;,\;\ldots\;,
\;\frac{|d/du_r|_{\mathcal{F}_{L,\rho},\mathrm{Sp}}}{|g_{1}^r|_\rho}\;\Bigr)\;.
\nonumber
\end{eqnarray}
Then:
\begin{itemize}
\item[$\bs{(1)}$] $T(\M,\rho)\; \geq\;
\min(\;\omega\;,\;\omega_{\M}(\rho)\;)$.%
\item[$\bs{(2)}$] The following conditions are equivalent:
\begin{itemize}
\item[$\bs{(a)}$] $T(\M,\rho) \;<\; \omega$;%

\item[$\bs{(b)}$] $\omega_{\M}(\rho)\;<\;\omega$;%

\item[$\bs{(c)}$] %
$|g_1^{0}|_{\rho} > \rho^{-1}$, or $|g_1^{i}|_{\rho} > 1$  for
some $i\in\{1,\ldots,r\}$.
\end{itemize}

\item[$\bs{(3)}$] If one of the equivalent conditions of point
$\bs{(2)}$ is verified, then one has
\begin{equation}
T(\M,\rho)= \min\Bigl(\;
\frac{\omega\cdot\rho^{-1}}{|g_{1}^0|_\rho}\;,
\;\frac{\omega}{|g_{1}^1|_\rho}\;,\;\ldots\;,
\;\frac{\omega}{|g_{1}^r|_\rho}\;\Bigr)\;.
\end{equation}
\end{itemize}
\end{lemma}
\dem By Lemma \ref{False lemma}, one has
$\omega=\frac{|d/dT|_{\mathcal{F}_{L,\rho},\mathrm{Sp}}}{|d/dT|_{\mathcal{F}_{L,\rho}}}=
\frac{|d/du_i|_{\mathcal{F}_{L,\rho},\mathrm{Sp}}}{|d/du_i|_{\mathcal{F}_{L,\rho}}}$,
for all $i=1,\ldots,r$. This proves that
$\bs{(b)}\Leftrightarrow\bs{(c)}$. We prove now $\bs{(1)}$. The
matrices $g_n^i$ verify the inductive relations
$g_{n+1}^0=d/dT(g_{n}^0)+g_{n}^0\cdot g_1^0$, and
$g_{n+1}^i=d/du_i(g_{n}^i)+g_{n}^i\cdot g_1^i$. By induction, one
has
$|g_{n}^0|_\rho\;\leq\;\max(|d/dT|_{\mathcal{F}_{L,\rho}},|g_1^0|_\rho)^n$,
and
$|g_{n}^i|_\rho\;\leq\;\max(|d/du_i|_{\mathcal{F}_{L,\rho}},|g_1^i|_\rho)^n$.
This proves that
\begin{eqnarray}
\protect{[\liminf_{n}|g_{n}^0|_{\rho}^{-1/n}]}&\geq&
\min\Bigl(\frac{1}{|d/dT|_{\mathcal{F}_{L,\rho}}},\frac{1}{|g_1^0|_\rho}\Bigr)=
\min\Bigl(\;\rho\;,\;|g_1^0|_\rho^{-1}\;\Bigr)\;,\label{eq-1-g_n=g^n}\\
\protect{[\liminf_{n}|g_{n}^i|_{\rho}^{-1/n}]}&\geq&
\min\Bigl(\frac{1}{|d/du_i|_{\mathcal{F}_{L,\rho}}},\frac{1}{|g_1^i|_\rho}\Bigr)=
\min\Bigl(\;1\;,\;|g_1^i|_\rho^{-1}\;\Bigr)\;.\label{eq-2-g_n=g^n}
\end{eqnarray}

Hence, by formula \eqref{S(S_i, rho) simpler basis}, the point
$\bs{(1)}$ holds. Moreover, the same computation proves the
following sub-lemma:
\begin{lemma}\label{sub-lemma-radius}
The following conditions are equivalent :
\begin{itemize}
\item[$\bs{(c')}$] $|g_1^0|_\rho>\rho^{-1}$ \;\;(resp.
$|g_1^i|_\rho>1$)\;,%
\item[$\bs{(a')}$]
$\protect{[\liminf_{n}|g_{n}^0|_{\rho}^{-1/n}]}< \rho$ (resp.
$\protect{[\liminf_{n}|g_{n}^i|_{\rho}^{-1/n}]}< 1$)\;.
\end{itemize}
\end{lemma} \hfill $\Box$
\if{
\begin{proof} We prove the sub-lemma only for $d/dT$, the case $d/du_i$ is
similar. If $|g_1^0|_\rho
> \rho^{-1}=|d/dT|_{\mathcal{F}_{L,\rho}}$, then the above induction proves
that \eqref{eq-1-g_n=g^n} is an equality since
$\omega=\frac{|d/du_i|_{\mathcal{F}_{L,\rho},\mathrm{Sp}}}{|d/du_i|_{\mathcal{F}_{L,\rho}}}=
\frac{|d/dT|_{\mathcal{F}_{L,\rho},\mathrm{Sp}}}{|d/dT|_{\mathcal{F}_{L,\rho}}}$
(cf. Lemma \ref{False lemma}). Hence
$[\liminf_{n}|g_{n}^0|_{\rho}^{-1/n}]=|g^0_1|_{\rho}^{-1}<\rho$.
Conversely if
$\protect{[\liminf_{n}|g_{n}^0|_{\rho}^{-1/n}]}<\rho$, then, by
equation \eqref{eq-1-g_n=g^n}, one has
$\min(\rho,|g^0_1|_{\rho}^{-1})<\rho$.}\fi 

\emph{Continuation of the proof of Lemma \ref{Small radius}:}
By using again  Formula \eqref{S(S_i, rho) simpler basis}, Lemma
\ref{sub-lemma-radius} proves that $\bs{(c)}$ implies $\bs{(a)}$,
and that the point $\bs{(3)}$ holds. Clearly, by assertion $\bs{(1)}$,
condition $\bs{(a)}$ implies $\omega > T(\M,\rho)\geq
\min(\omega,\omega_{\M}(\rho))$, hence $\bs{(b)}$ holds. \CVD

%



\begin{lemma}\label{radius at zero -yhnct}
Let $\M$ be a \emph{solvable} rank one $\nabla$-module over
$\a_L(]0,1[)$. For $i=1,\ldots,r$, let $g_1^{i} \in \a_{L}(]0,1[)$
be the matrix of $\nabla_{u_i}^{\M}$ (resp. if $i=0$, $g^0_1$ is
the matrix of $\nabla_T^{\M}$). Assume that, for all
$i=0,1,\ldots,r$, the matrix $g_1^i$ is of the form
\begin{equation}
g_1^i=\sum_{j\geq -n_i}a_j^{(i)}T^j \in
\a_L(]0,1[)\;,\quad\textrm{ with }a_{-n_i}^{(i)}\neq 0.
\end{equation}
Assume moreover that $n_0,n_1,\ldots,n_r<+\infty$ satisfy $n_0\geq
2$, or $n_i\geq 1$, for some $i=1,\ldots,r$. Then:
\begin{enumerate}
\item $|a_{-n_i}^{(i)}|\leq \omega$, for all $i=0,\ldots,r$; %
\item For $\rho$ sufficiently close to $0$, one has
\begin{equation}\label{radius near Zero}
 T(\M,\rho)\;\;=\;\; \omega\cdot\min_{i=1,\ldots,r}
\Bigl(\; |a_{-n_0}^{(0)}|^{-1}\cdot\rho^{n_0-1}\;,\;
|a_{-n_i}^{(i)}|^{-1}\cdot\rho^{n_i}\;\Bigr)\;;
\end{equation}
\item If $|a_{-n_i}^{(i)}|=\omega$ for some $i=0,\ldots,r$, then
equation \eqref{radius near Zero} holds for all $\rho\in ]0,1[$,
and
\begin{equation}
\mathrm{sw}^{\nabla}(\M) = \max\{\; \epsilon_0  (n_0-1)\;,\;
\epsilon_1  n_1\;,\;\ldots\;,\;\epsilon_r  n_r \;\}
\end{equation}
where $\epsilon_i=0$ if $|a_{-n_i}^{(i)}|<\omega$, and
$\epsilon_i= 1$ if $|a_{-n_i}^{(i)}|=\omega$.
\end{enumerate}
\end{lemma}
\begin{proof} By assumption, $\lim_{\rho\to 1^{^{-}}}T(\M,\rho)=1$. Moreover
we know that $T(\M,\rho)$ is continuous and $\log$-concave. We
have then only two possibilities: $T(\M,\rho)=1$ for all $\rho\in
]0,1[$, or there exists $\beta>0$ such that
$T(\M,\rho)\leq\rho^{\beta}$, for all $\rho\in]0,1[$. In the first
case, $T(\M,\rho)=1>\omega$, for all $\rho<1$. Hence by the Small
Radius Lemma \ref{Small radius}, we have $|g_1^0|_\rho \leq
\rho^{-1} $, and $|g_1^i|_\rho \leq 1$, for all $i = 1, \ldots,
r$, and for all $\rho\in]0,1[$. This contradicts our assumptions.
Indeed if $\rho$ is close to zero, one has $|g_1^{i}|_\rho =
|a_{-n_i}^{(i)}| \rho^{-n_i}$, for all $i = 0, \ldots, r$, and the
assumption $\max(\frac{n_0}{2},n_1,\ldots,n_r)\geq 1$ implies
that, for $\rho$ close to $0$, one has
$|g_1^{0}|_\rho>\rho^{-1}>1$, or that $|g_1^{i}|_\rho>1$, for some
$i=1,\ldots,r$.

Hence we are in the second case: $T(\M,\rho)\leq\rho^{\beta}$, for
all $\rho<1$, where $\beta=\sw^{\nabla}(\M)$ (This  follows from Definition  \ref{RK1} and the fact that $T(\M,\rho)$ is log-concave). Since $\beta>0$, if
$\rho$ is sufficiently close to $0$, one has
$T(\M,\rho)\leq\rho^{\beta}<\omega$. So Small Radius Lemma
\ref{Small radius} applies, and $T(\M,\rho)=\omega\cdot
\min_{i=1,\ldots,r} \Bigl(\frac{\rho^{n_0-1}}{|a_{-n_0}^{(0)}|},
\frac{\rho^{n_i}}{|a_{-n_i}^{(i)}|}\Bigr)$. Now, since
$\lim_{\rho\to 1^-}T(\M,\rho)=1$, and since the function
$\rho\mapsto T(\M,\rho)$ is $\log$-concave, its $\log$-slope at
$0^+$ is greater than its $\log$-slope at $1^{-}$ (i.e.
$\beta=\sw^{\nabla}(\M)$). Hence one has the inequality
\begin{equation}
\min_{i=0,\ldots,r} \frac{\omega}{|a_{-n_i}^{(i)}|}=\lim_{\rho\to
1^{-}}\omega\cdot \min_{i=1,\ldots,r} \Bigl(\;
\frac{\rho^{n_0-1}}{ |a_{-n_0}^{(0)}|}\;\;,\;\;\frac{
\rho^{n_i}}{|a_{-n_i}^{(i)}|}\;\Bigr)\geq \lim_{\rho\to
1^-}T(\M,\rho)=1\;,
\end{equation}
as in the following picture:
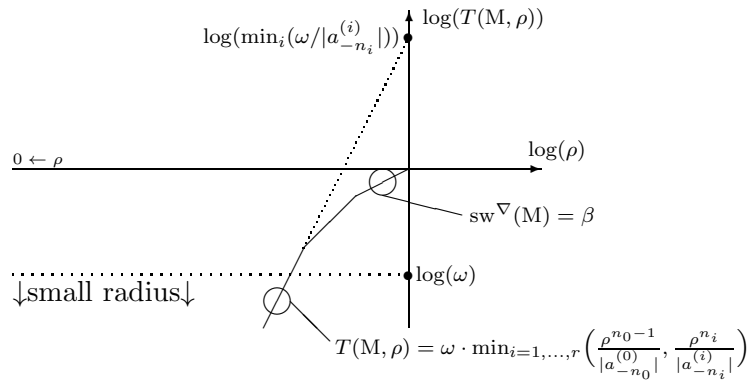
\begin{center}
\begin{scriptsize}
\begin{picture}(200,130)
\put(150,10){\vector(0,1){120}} \put(0,70){\vector(1,0){200}}
\put(195,75){$\log(\rho)$} \put(155,125){$\log(T(\M,\rho))$}
\put(0,72){\begin{tiny}$0\leftarrow\rho$\end{tiny}}
\put(95,10){\line(1,2){15}} 
\put(110,40){\line(1,1){20}} 
\put(130,60){\line(2,1){20}} 
\put(140,65){\circle{10}}     
\put(140,60){\line(4,-1){30}}
\put(100,20){\circle{10}}     
\put(105,20){\line(1,-1){15}}
\put(122,0){$T(\M,\rho)=\omega\cdot \min_{i=1,\ldots,r}
\Bigl(\frac{\rho^{n_0-1}}{|a_{-n_0}^{(0)}|},
\frac{\rho^{n_i}}{|a_{-n_i}^{(i)}|}\Bigr)$}
\put(173,50){$\mathrm{sw}^{\nabla}(\M)=\beta$}
\put(0,20){\begin{normalsize}$\downarrow$small
radius$\downarrow$\end{normalsize}}
\put(147.5,117.5){$\bullet$} 
\put(72,117.5){$\log(\min_{i}(\omega/|a_{-n_i}^{(i)}|))$}
\put(147.5,27.5){$\bullet$} 
\put(152.5,27.5){$\log(\omega)$}
\qbezier[40](110,40)(130,80)(150,120)
\qbezier[40](0,30)(75,30)(150,30)
\end{picture}
\end{scriptsize}
\end{center}
This implies $|a_{-n_i}^{(i)}|\leq\omega$, for all $i=0,\ldots,r$.
Moreover if $|a_{-n_i}^{(i)}|=\omega$, for some $i=0,\ldots,r$,
then this graphic is a line, and assertion iii) holds.
\end{proof}

\subsection{Proof of theorem \ref{comparison-theorem in rank one}}

Let $\V(\alpha)\in\mathrm{Rep}^{\mathrm{fin}}_{\O_K}(\G_{\E})$. As
usual, we decompose $\alpha = \alpha_{\mathrm{wild}} \cdot
\alpha_{\mathrm{tame}} \cdot \alpha_{k}$ (cf. Equation
\eqref{alpha_k,tame,wild}). By Section \ref{vanishing of tame and
residual katos conductors}, in both residual and tame cases the
arithmetic Swan conductor is equal to zero. On the other hand, we
have the  following:

\begin{lemma}\label{diff swan mod equal to 0}
The differential Swan conductors of
$\mathrm{D}^{\mathrm{\dag}}(\V(\alpha_k))$ and
$\mathrm{D}^{\mathrm{\dag}}(\V(\alpha_{\mathrm{tame}}))$ are equal
to $0$. More precisely there exist  bases of  $\mathrm{D}^{\mathrm{\dag}}(\V(\alpha_k))$ and
$\mathrm{D}^{\mathrm{\dag}}(\V(\alpha_{\mathrm{tame}}))$ in which the  connections are defined over 
$\mathcal{A}_L(]0,1[)$ and 
$T(\mathrm{D}^{\mathrm{\dag}}(\V(\alpha_k)),\rho)=T(\mathrm{D}^{\mathrm{\dag}}(\V(\alpha_{\mathrm{tame}})),\rho)=1$,
for all $\rho\in]0,1[$.
\end{lemma}
\begin{proof}
Let $\D_k:=\D^{\dag}(\V(\alpha_k))$ and
$\D_{\mathrm{tame}}:=\D^{\dag}(\V(\alpha_{\mathrm{tame}}))$. By
equations \eqref{matrix in residual case}, since the matrices of
$\nabla_T^{\D_k}$ and $\nabla_{u_i}^{\D_k}$ belongs to $\O_L$,
hence $T(\D_k,\rho)$ does not depend on $\rho$, and its log-slope
is equal to $0$ (one has actually $T(\D_k,\rho)=1$, for all
$\rho<1$). On the other hand, in the notation of Lemma
\ref{expliciting the definition of the Swan conductor}, since
$g_0^0\in\mathbb{Z}_p\cdot T^{-1}$ (cf. Equations \ref{matrix in
the tame case}),  the function $\rho\mapsto
\liminf_n|g^0_n|_\rho^{-1/n}$ is a constant function (cf.
\cite[Lemma 1.4]{Rk1}). Moreover, the matrix of
$\nabla_{u_i}^{\D_{\mathrm{tame}}}$ is equal to $0$, hence
$\rho\mapsto \liminf_n|g^i_n|_\rho^{-1/n}$ is the constant
function equal to $\infty$. So $T(\D_{\mathrm{tame}},\rho)=1$,
for all  $\rho<1$, and its log-slope is equal to $0$. \end{proof}

\subsubsection{}\label{remark on the strategy} Since in
both residual and tamely ramified cases the arithmetic and the differential
Swan conductors are  equal to zero, then Remark \ref{V_1
otimes V_2} and Section \ref{M_1 otimes M_2}, allow us to reduce to
prove the theorem for $\alpha_{\mathrm{wild}}$. The proof of
Theorem \ref{comparison-theorem in rank one} is then carried out  in
two steps: first we prove  the theorem for the case in which
$\alpha_{\mathrm{wild}} =\psi_m\circ
\delta(\overline{\lb}_{-d}t^{-d})$, with
$\lb_{-d}\in\W_{v_p(d)}(k)-p\W_{v_p(d)}(k)$, where
$\psi_m:\mathbb{Z}/p^{m+1}\mathbb{Z}\simto\bs{\mu}_{p^{m+1}}\subset\O_K$
is the identification of Definition \ref{definition of psi_m}, and
$\delta$ is the Artin-Schreier-Witt morphism of  Sequence
\eqref{artin-schreier-diagram-covectors} (cf. Lemma
\ref{fe-monomial} below). The second step consists of extending  the
theorem to every character using Section  \ref{minimal lifting}.

\begin{lemma}\label{fe-monomial}
Let $\overline{\lb}\cdot t^{-np^m}$ be a co-monomial satisfying
$\overline{\lb}\in\W_m(k)-p\W_m(k)$ (cf. Section \ref{minimal
lifting}), let
$\alpha_{\mathrm{wild}}:=\psi_m\circ\delta(\overline{\lb}
t^{-np^m})$. We denote by $\M$ the $\nabla$-module over $\R_L$
defined by $\mathrm{D}^{\dag}(\V(\alpha_{\mathrm{wild}}))$. Then
\begin{equation}
\mathrm{sw}^{\nabla}(\M)=np^{m}=
\mathrm{sw}(\alpha_{\mathrm{wild}})\;.
\end{equation}
More precisely, one has $T(\M,\rho)=\rho^{np^m}$, for all
$\rho\in]0,1[$.
\end{lemma}
\begin{proof} By Corollary \ref{swan conductor of a comonomial}
one has $\mathrm{sw}(\alpha_{\mathrm{wild}})=np^m$, hence it is
enough to prove the last assertion. To prove this, we now use the
assertion iii) of Lemma \ref{radius at zero -yhnct}. We verify
that all assumptions of Lemma \ref{radius at zero -yhnct} are
verified.

Let $\lb=(\lambda_0,\ldots,\lambda_m)\in \W_m(\O_{L_0})$ be an
arbitrary lifting of $\overline{\lb}= (\overline{\lambda}_0,
\ldots, \overline{\lambda}_m)$. Then the solution of $\M$ in this
basis is the $\bs{\pi}$-exponential (cf. Section \ref{e(f(T),1) is
a solution of D(V)})
\begin{equation}
\et_{p^\infty}(\lb
T^{-np^m},1)=\exp\Bigl(\;\pi_m\phi_0T^{-n}+\pi_{m-1}\phi_1\frac{T^{-np}}{p}+
\cdots+\pi_0\phi_m\frac{T^{-np^m}}{p^m}\;\Bigr)\;,
\end{equation}
where $\ph{\phi_0,\ldots,\phi_m}\in\O_{L_0}^{m+1}$ is the phantom
vector of $\lb\in\W_m(\O_{L_0})$. Indeed the phantom vector of
$\lb T^{-np^m}$ is
$\ph{\phi_0T^{-n},\phi_1T^{-np},\ldots,\phi_mT^{-np^m}}$. Writing
$\lb$ as function of $\underline{u}:=(u_1,\ldots,u_r)$, the
matrices of the connections are explicitly given by (cf. Equation
\eqref{matrix in the wild case})
\begin{eqnarray*}
 g^0_{\bs{f}^-}(u_1,\ldots,u_r,T)  &=&
 -n\;(\;\pi_m \phi_0 T^{-n-1}+ \pi_{m-1}\phi_1T^{-np-1}+\cdots+ \pi_0\phi_m T^{-np^m-1}\;) \;,\\
 g^i_{\bs{f}^-}(u_1,\ldots,u_r,T)  &=&
 \sum_{j=0}^m\pi_{m-j}T^{-np^{j}}\cdot
 \sum_{k=0}^j\lambda_{k}(\underline{u})^{p^{j-k}-1}\frac{d}{du_i}(\lambda_k(\underline{u})) \;.
\end{eqnarray*}
The coefficients of lowest degree (with respect to $T$) in the
matrix $g^i_{\bs{f}^-}(u_1,\ldots,u_r,T)$ are respectively
\begin{eqnarray}
\qquad a_{-np^m-1}^{(0)}&:=&-n\pi_0\phi_{m}\;,\\
a_{-np^m}^{(i)}&:=&\pi_0\cdot
p^{-m}\cdot\frac{d}{du_i}(\phi_m(u_1,\ldots,u_m))\;,\quad
\textrm{for all }i=1,\ldots,r\;.
\end{eqnarray}
More explicitly, for $i=1,\ldots,r$, one has
\begin{equation}
a_{-np^m}^{(i)}\; :=\;
\pi_{0}\cdot\sum_{k=0}^j\lambda_{k}(\underline{u})^{p^{j-k}-1}\frac{d}{du_i}(\lambda_k(\underline{u}))
\;\in\; \O_L\;.
\end{equation}
Since $\lb\in\W_m(\O_{L_0})$ has coefficients in $\O_{L_0}$ (which
is a Cohen ring), one sees that
\begin{equation}
|\phi_m|=|p|^{s(\overline{\lb})}\;,
\end{equation}
where
$s(\overline{\lambda}_0,\ldots,\overline{\lambda}_m):=\min\{s\geq
0\;|\;\overline{\lambda}_s\neq 0\}$. We study separately two
cases: $\overline{\lambda}_0\neq 0$, and $\overline{\lambda}_0=0$.
Assume first that $\overline{\lambda}_0\neq 0$ (i.e.
$s(\overline{\lb})=0$). Since $n>0$, then the $T$-adic valuations
of $ g^0_{\bs{f}^-}$ and $g^i_{\bs{f}^-}$ satisfy conditions of
Lemma \ref{radius at zero -yhnct}. Since $|\pi_0|=\omega$ and
$|\phi_m|=1$, then $|a_{-np^m-1}^{(0)}|=\omega$. Hence the point
iii) of Lemma \ref{radius at zero -yhnct} applies, and one has
$T(\M,\rho)=\rho^{np^m}$, for all $\rho\in ]0,1[$.

Assume now that $\overline{\lambda}_0 = 0$, i.e.
$s(\overline{\lb})\geq 1$. By assumption,
$\overline{\lb}\in\W_m(k)-p\W_m(k)$, but since
$\overline{\lambda}_0=0$, this is equivalent to (cf. Equation
\eqref{explicit def of pW(R)})
\begin{equation}
\overline{\lb}\notin\W_m(k^p)\;.
\end{equation}
The following Lemma \ref{tretretyi--} proves that
\begin{equation}
\max_{i=1,\ldots,r}|a_{-np^m}^{(i)}| < \omega
\quad\Longrightarrow\quad \overline{\lb}\in \W_m(k^p)\;,
\end{equation}
which contradicts our assumption. Hence
$|a_{-np^m}^{(i)}|=\omega$, for some $i=1,\ldots,r$. By applying
point iii) of Lemma \ref{radius at zero -yhnct}, as above, we find
$T(\M,\rho)=\rho^{np^m}$, for all $\rho\in]0,1[$.
\end{proof}

\begin{lemma}\label{tretretyi--}
Let $k$ be a field of characteristic $p>0$. Let
$\overline{\lb} = (\overline{\lambda}_0,\overline{\lambda}_1,\ldots,\overline{\lambda}_m)\in
\W_m(k)$ be a Witt vector. The following assertions are
equivalent:
\begin{enumerate}
\item[$\bs{(1)}$] For all $i=1,\ldots,r$ one has
\begin{equation}\label{rgyjil}
\overline{\lambda}_{0}^{p^{m}-1}\frac{d}{d\bar{u}_i}(\overline{\lambda}_0)+
\overline{\lambda}_{1}^{p^{m-1}-1}\frac{d}{d\bar{u}_i}(\overline{\lambda}_1)+\cdots+
\overline{\lambda}_{m-1}^{p-1}\frac{d}{d\bar{u}_i}(\overline{\lambda}_{m-1})
+\frac{d}{d\bar{u}_i}(\overline{\lambda}_m)=0\;.
\end{equation}

\item[$\bs{(2)}$] $\overline{\lb}\in\W_m(k^p)$.
\end{enumerate}
\end{lemma}
\begin{proof} Clearly $\bs{(2)}$ implies $\bs{(1)}$. Assume then that $\bs{(1)}$ holds. The condition $\bs{(1)}$ is equivalent to
\begin{equation}\label{rgyjil-2-2}
\overline{\lambda}_{0}^{p^{m}-1} d(\overline{\lambda}_0)+
\overline{\lambda}_{1}^{p^{m-1}-1} d(\overline{\lambda}_1)+\cdots+
\overline{\lambda}_{m-1}^{p-1} d(\overline{\lambda}_{m-1})
+d(\overline{\lambda}_m)=0\;,
\end{equation}
in $\Omega^1_{k/\mathbb{F}_p}$, where $d:k\to
\Omega^1_{k/\mathbb{F}_p}$ is the canonical derivation. We proceed
by induction on $m\geq 0$. The case $m=0$ is evident, since
$d(\overline{\lambda}_0)=0$ in $\Omega^1_{k/\mathbb{F}_p}$ implies
$\overline{\lambda}_0\in k^p$ (cf. \cite[Ch.0, 21.4.6]{EGAIV-1}). Let now $m\geq 1$. Assume  Equation
\eqref{rgyjil-2-2} holds. Following \cite[II.6]{Cart},
let
$Z^{1}(k):=\mathrm{Ker}(d:\Omega^1_{k/\mathbb{F}_p}\to\Omega^2_{k/\mathbb{F}_p})$,
and let
\begin{equation}
C^{-1} \;:\; \Omega^1_{k/\mathbb{F}_p}\;
\xrightarrow[]{\;\;\sim\;\;}\; Z^{1}(k)/d(k)
\end{equation}
be the Cartier isomorphism (cf. \cite[II.6]{Cart}). Here
$d(k)$ is the image of the map $d:k\to\Omega^1_{k/\mathbb{F}_p}$.
We recall that $C^{-1}(ad(b)) = \overline{a^{p}b^{p-1}d(b)}$,
where $\overline{\omega}$ is the class of $\omega\in Z(k)$ modulo
$d(k)$.

Let
$E_m(\overline{\lambda}_0,\ldots,\overline{\lambda}_m):=\overline{\lambda}_{0}^{p^{m}-1}
d(\overline{\lambda}_0)+ \overline{\lambda}_{1}^{p^{m-1}-1}
d(\overline{\lambda}_1)+\cdots+ \overline{\lambda}_{m-1}^{p-1}
d(\overline{\lambda}_{m-1})
+d(\overline{\lambda}_m)\in\Omega^{1}_{k/\mathbb{F}_p}$. Since
$E_m(\overline{\lambda}_0,\ldots,\overline{\lambda}_m)=0$, then
$E_m(\overline{\lambda}_0,\ldots,\overline{\lambda}_m)\in Z(k)$,
and the class of
$\overline{E_m(\overline{\lambda}_0,\ldots,\overline{\lambda}_m)}\in
Z(k)/d(k)$ is equal to
\begin{equation}
0=\overline{E_m(\overline{\lambda}_0,\ldots,\overline{\lambda}_m)}=\overline{\overline{\lambda}_{0}^{p^{m}-1}
d(\overline{\lambda}_0)+ \overline{\lambda}_{1}^{p^{m-1}-1}
d(\overline{\lambda}_1)+\cdots+ \overline{\lambda}_{m-1}^{p-1}
d(\overline{\lambda}_{m-1})}\;.
\end{equation}
By definition, one has
$C^{-1}\Bigl(\overline{\lambda}_i^{p^{m-1-i}-1}d(\overline{\lambda}_i)\Bigr)
=
\overline{\overline{\lambda}_i^{p^{m-i}-1}d(\overline{\lambda}_i)}$,
hence we find
\begin{equation}\label{tytytgtyhgt}
0=\overline{E_m(\overline{\lambda}_0,\ldots,\overline{\lambda}_m)}=
C^{-1}(E_{m-1}(\overline{\lambda}_0,\ldots,\overline{\lambda}_{m-1}))\;.
\end{equation}

Since $C^{-1}$ is an isomorphism (cf. \cite[II.6]{Cart}), then the equation \eqref{tytytgtyhgt} implies
$E_{m-1}(\overline{\lambda}_0,\ldots$ $ \dots ,\overline{\lambda}_{m-1})=0$.
By induction, one finds then
$\overline{\lambda}_0,\ldots,\overline{\lambda}_{m-1}\in k^p$, but
this implies
$E_m(\overline{\lambda}_0,\ldots,\overline{\lambda}_m)=d(\overline{\lambda}_m)=0$,
hence one has also $\overline{\lambda}_m\in k^p$ (cf.
\cite[Ch.0, 21.4.6]{EGAIV-1}).
\end{proof}

\begin{proof}[\protect{End of the Proof of Theorem
\ref{comparison-theorem in rank one} :}] Let
$\overline{\bs{f}^-}(t)=\sum_{n\in\J} \overline{\lb}_{-n}
t^{-np^{m(n)}}$ be a minimal lifting  of
$\psi_m^{-1}\circ\alpha_{\mathrm{wild}}$ in $\CW(\E)$ (cf. Section  
\ref{minimal lifting}). Again by  Section  \ref{minimal lifting}, one has
\begin{equation}
\mathrm{sw}(\delta(\overline{\bs{f}^-}(t)))=
\max_{n\in\J}\;\mathrm{sw}(\delta(\overline{\lb}_{-n}t^{-np^{m(n)}}))\;,
\end{equation}
where $\delta$ is the Artin-Schreier-Witt morphism
(cf. Equation \eqref{artin-schreier-diagram-covectors}). Now we recall that
$\V(\alpha_{\mathrm{wild}})= \otimes_{n\in\J}\V(\delta(\overline{\lb}_{-n}T^{-np^{m(n)}}))$,
and
\begin{equation}
\mathrm{D}^{\dag}(\V(\alpha_{\mathrm{wild}}))=\otimes_{n\in\J}
\mathrm{D}^{\dag}(\V(\delta(\overline{\lb}_{-n}T^{-np^{m(n)}})))\;.
\end{equation}
By Lemma \ref{fe-monomial}, one has
\begin{equation}
\mathrm{sw}^{\nabla}(\mathrm{D}^{\dag}(\V(\delta(\overline{\lb}_{-n}T^{-np^{m(n)}}))))=np^{m(n)}\;.
\end{equation}
Now if $n_1\neq n_2$ then $n_1\cdot p^{m(n_1)}\neq n_2\cdot
p^{m(n_2)}$. Hence, by Section \ref{M_1 otimes M_2}, one has
\begin{equation}
\mathrm{sw}^{\nabla}( \mathrm{D}^{\dag}(\V(\alpha_{\mathrm{wild}})) )\;=\;\max_{n\in\J}\;np^{m(n)}\;.
\end{equation}
This proves Theorem \ref{comparison-theorem in rank one}.
\end{proof}

\bibliographystyle{alpha}
\bibliography{Article6}

\end{document}